\documentclass[hidelinks,onefignum,onetabnum]{siamart220329}

\usepackage{lipsum}
\usepackage{amsfonts}
\usepackage{graphicx}
\usepackage{epstopdf}
\usepackage{algorithmic}
\ifpdf
  \DeclareGraphicsExtensions{.eps,.pdf,.png,.jpg}
\else
  \DeclareGraphicsExtensions{.eps}
\fi

\newsiamremark{remark}{Remark}
\newsiamremark{Ex}{Example}
\newsiamremark{hypothesis}{Hypothesis}
\crefname{hypothesis}{Hypothesis}{Hypotheses}
\newsiamthm{claim}{Claim}

\headers{Limiter-based fully-discrete entropy stable explicit DG schemes}{}

\title{A limiter-based approach to construct high-order fully-discrete entropy stable explicit DG schemes for hyperbolic conservation laws}

\author{Yuchang Liu\thanks{School of Mathematical Sciences,
         University of Science and Technology of China,
         Hefei, Anhui 230026, P.R. China.
  (\email{lissandra@mail.ustc.edu.cn}).}
\and Wei Guo\thanks{Department of Mathematics and Statistics, Texas Tech University, Lubbock, TX, 70409, USA. 
  (\email{weimath.guo@ttu.edu}).  }
\and Yan Jiang\thanks{ School of Mathematical Sciences,
         University of Science and Technology of China, Hefei, 
         Anhui 230026, P.R. China.  
         (\email{jiangy@ustc.edu.cn}). }
\and Zheng Sun\thanks{Department of Mathematics, The University of Alabama,
		Tuscaloosa, AL 35487, USA. 
(\email{zsun30@ua.edu}). }
}

\usepackage{amsopn}

\ifpdf
\hypersetup{
	pdftitle={A limiter-based approach to construct high-order fully-discrete entropy stable explicit DG schemes for hyperbolic conservation laws},
    pdfauthor={}
}
\fi

\allowdisplaybreaks

\newtheorem{assumption}{Assumption}[section]
\usepackage{subfigure}

\begin{document}
	
	\maketitle
	
	\begin{abstract}
		This paper presents a class of novel high-order fully-discrete entropy stable (ES) discontinuous Galerkin (DG) schemes with explicit time discretization. The proposed methodology exploits a critical observation from \cite{carlier2023invariant} that the cell averages of  classical DG solutions with forward Euler time stepping satisfy an ``entropy-stable-like'' property. Building on this result, fully-discrete entropy stability is rigorously enforced through a simple Zhang--Shu-type scaling limiter \cite{zhang2010positivity} applied as a post-processing step, without modifying the underlying spatial discretization. Furthermore, the proposed methodology can simultaneously enforce multiple discrete cell entropy inequalities, a capability unavailable in existing ES DG schemes.
		High-order accuracy in time is achieved by using strong-stability-preserving multistep methods. Theoretically, we prove that the proposed scheme indeed maintains high-order accuracy and establish a Lax--Wendroff-type theorem showing that the limit of the numerical solutions, if it exists, satisfies the desired entropy inequality. Extensive numerical tests for scalar equations and systems, including the nonconvex Buckley--Leverett problem and extreme examples of Euler equations, demonstrate optimal accuracy,  enforcement of multiple entropy conditions, and strong robustness.
	\end{abstract}
	
	
        
	\section{Introduction}
	\label{sec1}
	    
	Hyperbolic conservation laws are a class of partial differential equations describing the evolution of conserved quantities, with wide-ranging applications in physics and engineering. Prominent examples include the inviscid Euler equations in gas dynamics, the ideal magnetohydrodynamics (MHD) equations in plasma physics, and shallow water equations over a flat bottom. Over the past few decades, the rapid development of high-order numerical methods  across various discretization frameworks has enabled the accurate resolution of complex solution structures using relatively coarse meshes. Among these, the discontinuous Galerkin (DG) method, originally introduced by Reed and Hill in 1973 for neutron transport problems \cite{reed1973triangular}, has attracted significant research attention for its local conservation property, high-order accuracy, and suitability for parallel computation. Building on this foundation, Cockburn and Shu subsequently developed the Runge–Kutta (RK) DG framework \cite{cockburn1989tvb, cockburn1991runge}, which combines DG spatial discretization with strong-stability-preserving (SSP) RK time integration and nonlinear limiters to control spurious oscillations, establishing a cornerstone in numerical solutions of hyperbolic conservation laws.
	
	A well-known property of hyperbolic conservation laws is that even with smooth initial data, solutions may develop discontinuities in finite time, rendering classical smooth solutions invalid. This necessitates the introduction of weak solutions. However, weak solutions are generally not unique, and additional physical principles are therefore required to identify the physically relevant solution. Among these, the entropy condition plays a fundamental role: a weak solution that satisfies the entropy inequality for all admissible entropy pairs is called an \textit{entropy solution}. For scalar conservation laws, entropy solutions are known to be unique \cite{godlewski2013numerical}.

	From the perspective of numerical algorithm design, ensuring a \textit{discrete entropy condition} poses a significant challenge. The monotone schemes and E-schemes are known to satisfy the entropy inequality for any entropy pair \cite{leveque1992numerical,osher1988convergence}. However, both approaches are at best first-order accurate. For high-order numerical schemes, one typically seeks to enforce the entropy inequality only for a prescribed entropy pair. A seminal contribution in this direction was made by Tadmor, who proposed the framework of two-point entropy conservative and entropy stable (ES) fluxes \cite{tadmor1987numerical, tadmor2003entropy}. This technique was subsequently extended by LeFloch \emph{et al.} \cite{lefloch2002fully}, who constructed entropy conservative fluxes with arbitrarily high-order accuracy. These foundational developments have since been widely used in the design of ES finite difference schemes \cite{duan2019high, duan2021high, agnihotri2024second, singh2025entropy}, including the well-known TeCNO method \cite{fjordholm2012arbitrarily}, which utilizes the sign-property of ENO reconstruction.
	
In the context of DG discretizations,  a classical result of Jiang and Shu \cite{jiang1994cell} shows that standard semi-discrete DG schemes satisfy a cell entropy inequality for the square entropy in scalar conservation laws; this result was later extended to symmetric systems by Hou and Liu \cite{hou2007solutions}. However, these results are restricted to the square entropy and rely on exact integration. For general entropy functions and practical quadrature rules, there are two principal approaches to achieving discrete entropy stability.
	The first approach, developed primarily within the nodal DG framework, exploits the summation-by-parts (SBP) property of the discrete differentiation operator and employs flux differencing techniques combined with entropy conservative and ES numerical fluxes \cite{chen2017entropy, liu2018entropy, bohm2020entropy, rueda2021entropy, liu2024entropy, liu2025structure, liu2025globally, liu2025entropyMHD}. 
	The second approach, originally proposed by Abgrall \cite{abgrall2018general} in a general discretization setting, seeks to directly control entropy production by introducing carefully designed artificial viscosity terms into the DG formulation \cite{gaburro2023high, liu2024non, chan2025artificial}.
	A comprehensive review of ES DG methods can be found in \cite{chen2020review}.
	To date, entropy stability for both approaches has been rigorously established at the semi-discrete level. Furthermore, fully-discrete entropy stability can be attained by employing implicit time discretizations. For instance, in \cite{lefloch2002fully}, a general framework for high-order ES implicit time-stepping schemes is introduced. Alternatively, space-time discretization techniques can be used to achieve fully-discrete entropy stability \cite{hiltebrand2014entropy, zakerzadeh2017entropy}.
	
On the other hand, discrete entropy stability for high-order explicit time discretizations remains largely open beyond special entropy settings. For the square entropy, namely $L^2$ stability, several positive theoretical 
results are available. For the linear case, general frameworks for analyzing 
explicit RKDG schemes can be found in \cite{xu20192,sun2019strong}. In particular, it has been shown that all $s$-stage, $s$th-order RK methods 
coupled with arbitrarily high-order DG spatial discretizations are 
monotonically (or strongly) stable under the usual CFL condition 
if $s \equiv 3 \pmod{4}$. These linear analyses have 
also been extended to variants of RKDG methods, including methods with 
stage-dependent numerical fluxes \cite{xu2024stability} and methods with reduced inner-stage 
polynomial orders \cite{sun2025reducing}.
For the nonlinear case, Wei and Shu recently proved the 
$L^2$ stability of a second-order RKDG method with $\mathbb{P}^1$ elements for 
two classes of nonlinear convex scalar conservation laws \cite{wei20252}. They also 
disproved the monotone (or strong) $L^2$ stability for a class of third-order 
RKDG methods with $\mathbb{P}^2$ elements via a counterexample with Burgers' 
equation \cite{wei2025counterexample}, and proved the $L^2$ stability for another class of 
third-order-in-time RKDG methods for strictly convex scalar conservation laws 
with $\mathbb{P}^1$ elements and for Burgers' equation with $\mathbb{P}^2$ 
elements \cite{wei2025fully}. While these results represent important 
progress for the square entropy, the corresponding fully discrete stability 
theories for general entropy functions and systems of conservation laws 
remain difficult to establish. In addition, entropy stable relaxation RK techniques were developed in \cite{ranocha2020relaxation}. These methods first perform a standard one-step RK update and then compute a relaxation factor by solving a global nonlinear equation; the time step is rescaled by this factor to obtain fully discrete global entropy conservation or entropy stability. However, solving a nonlinear equation at each time step introduces additional computational cost. Local relaxation extensions can also enforce cell entropy
inequalities \cite{ranocha2020local}. For a general entropy function, these methods require
a scalar nonlinear solve for each local region. If the local
regions are chosen as individual mesh cells, one such equation
is solved per cell. 
	
Recently, a class of fully-discrete ES finite difference schemes coupled with forward Euler time discretization was proposed in \cite{kivva2022entropy, kivva2024entropy}. These schemes are formally first-order accurate in time. The central idea relies on the algebraic flux correction (AFC) framework, in which entropy stability is achieved by forming a convex combination of a high-order flux and a suitably chosen low-order ES flux. A key ingredient of these developments is the use of the concept of a proper numerical entropy flux \cite{zhao1997muscl, karlsen2014l1}, which provides a convenient and rigorous mechanism for enforcing a fully-discrete local cell entropy inequality, and thus global entropy stability. The AFC technique has also been extended to continuous finite element methods \cite{kuzmin2022limiter} to construct first-order explicit fully-discrete ES schemes.
	
In this paper, we present a novel high-order explicit DG framework that achieves a fully-discrete local cell entropy inequality and global entropy stability. In \cite{carlier2023invariant}, the authors showed that cell averages of the standard DG solution with forward Euler time stepping satisfy an ``entropy-stable-like" property. Although it does not directly lead to genuine entropy stability, it provides a crucial insight that motivates the present work. In particular, we first reformulate the ``entropy-stable-like" property \cite{carlier2023invariant} by leveraging the concept of proper numerical entropy fluxes \cite{kivva2022entropy, kivva2024entropy}. This reformulation yields a computable upper bound for the cell entropy at the next time level, while guaranteeing that the entropy of the cell average does not exceed this bound.
Furthermore, in \cite{chen2017entropy}, it is shown that scaling a numerical solution toward its cell average cannot increase the cell entropy. Leveraging this property, we apply a Zhang–Shu-type limiter to rescale the DG solution toward the cell average, enforcing that the limited cell entropy remains below the upper bound. 
This procedure allows us to rigorously achieve entropy stability at the fully-discrete level without modification of the underlying DG spatial discretization. An additional advantage of this framework is that it naturally allows the simultaneous enforcement of multiple entropy inequalities, each associated with a different entropy pair. This feature is closely related to invariant-domain-preserving schemes \cite{guermond2018second, carlier2023invariant, guermond2024first} or the monolithic convex limiting approach \cite{kuzmin2020monolithic, rueda2024monolithic}, which can enforce multiple invariant constraints simultaneously. 
This capability is particularly beneficial for the simulation of nonconvex hyperbolic conservation laws. 
Forward Euler time discretization  can be replaced by SSP multistep methods to achieve high-order accuracy in both space and time \cite{gottlieb2001strong}. Theoretically, we show that under appropriate assumptions, the proposed method indeed maintains high-order accuracy, and the limit of numerical solutions, if it exists, satisfies the desired entropy inequality in the sense of the classical Lax--Wendroff theorem. 
   We also note the concurrent and independent work of Wu \cite{wu2026epo}, which shares a closely related radial-scaling idea for enforcing fully discrete entropy inequalities. While the two approaches are developed independently and share certain common ingredients, they place emphasis on different aspects of the problem. The present work is primarily concerned with the construction and analysis of entropy-stable DG schemes, whereas \cite{wu2026epo} presents a broader framework that also incorporates positivity preservation and oscillation suppression.
	
In summary, our method has the following key features: 
\begin{itemize}

        \item[1.] It is fully explicit.
		
		\item[2.] It achieves high-order accuracy in both space and time.
		
		\item[3.] It strictly satisfies fully-discrete local cell entropy inequalities as well as global entropy stability. In particular, cell entropy inequalities can be simultaneously enforced for any finite number of prescribed entropy pairs.
		
		\item[4.] It is simple and non-intrusive: entropy stability is enforced without modifying the semi-discrete DG formulation or introducing nonlinear solves; it is achieved through a conservative cell-wise scaling post-processing step.
		
\end{itemize}
	
The rest of the paper is organized as follows. In Section \ref{sec2}, we review the fundamental entropy analysis for hyperbolic conservation laws. Section \ref{sec3} introduces the proposed scheme for scalar conservation laws. Section \ref{sec4} extends the proposed methodology to systems of conservation laws. Section \ref{sec5} provides comprehensive numerical experiments to validate the efficiency and effectiveness of the proposed scheme. Finally, Section \ref{sec6} concludes the paper with a summary and discussion of future research directions.

	\section{Entropy analysis}\label{sec2}
	
	Consider a general $d$-dimensional system of conservation laws
	\begin{equation}\label{eq:hcl-d}
		\frac{\partial\mathbf u}{\partial t}+\sum\limits_{i=1}^d \frac{\partial \mathbf f_i(\mathbf u)}{\partial x_i} = \mathbf 0,\quad \mathbf x\in \mathbb R^d,\ t>0
	\end{equation}
	with initial data $\mathbf u(\mathbf x,0)=\mathbf u_0(\mathbf x)$. Here, $\mathbf u(\mathbf x,t):\mathbb R^d\times \mathbb R^+\to \mathcal D\subset  \mathbb R^p,\  \mathbf u_0(\mathbf x):\mathbb R^d\to\mathcal D$, and the flux functions $\mathbf f_i:\mathcal D\to\mathbb R^p$. A characteristic of equation \eqref{eq:hcl-d} is that shock waves can form in finite time even from smooth initial data. Hence, one seeks weak solutions of \eqref{eq:hcl-d}. However, since weak solutions are not unique, an entropy condition is typically imposed to select the ``physically relevant" one.
	
	\begin{definition}[Entropy function]
		A convex function $\mathcal{U}(\mathbf{u}): \mathcal{D} \rightarrow \mathbb{R}$ is called an entropy function for the system \eqref{eq:hcl-d} if there exists an entropy flux $$\boldsymbol{\mathcal{F}}(\mathbf{u}) = (\mathcal{F}_1(\mathbf{u}), \dots, \mathcal{F}_d(\mathbf{u})): \mathcal{D} \rightarrow \mathbb{R}^d$$ such that 
		\begin{equation}\label{eq:entropy}
			\mathcal{F}_i'(\mathbf{u}) = \mathcal{U}'(\mathbf{u}) \,\mathbf{f}_i'(\mathbf{u}), \quad i = 1, \dots, d.
		\end{equation}
		We call $(\mathcal U,\boldsymbol{\mathcal F})$ an entropy pair.
	\end{definition}

	Let $\mathbf v=\mathcal U'(\mathbf u)$ denote the \textit{entropy variable}. For the following symmetrization discussion, assume that $\mathcal U\in C^2(\mathcal D)$ has a positive definite Hessian on $\mathcal D$. Then the mapping $\mathbf u\to \mathbf v$ is one-to-one. We may rewrite \eqref{eq:hcl-d} in terms of entropy variables:
	\begin{equation}\label{eq:hcl-v}
		\mathbf u'(\mathbf v)\frac{\partial \mathbf v}{\partial t}+\sum\limits_{i=1}^d \mathbf f_i'(\mathbf u)\mathbf u'(\mathbf v)\frac{\partial\mathbf v}{\partial x_i} = \mathbf 0.
	\end{equation}
	Under this assumption, $\mathbf u'(\mathbf v)=\big(\mathcal U''(\mathbf u)\big)^{-1}$ is positive definite as well. If the matrices $\mathbf f_i'(\mathbf u)\mathbf u'(\mathbf v)$ are symmetric for all $i$,  \eqref{eq:hcl-v} is called a \textit{symmetrization} of \eqref{eq:hcl-d}, and \eqref{eq:hcl-d} is said to be symmetrizable. A classical result in \cite{godlewski2013numerical} shows that the symmetrizability of \eqref{eq:hcl-d} is equivalent to the existence of an entropy function.
	
	\begin{lemma}[Mock]\label{lem:mock}
		A $C^2$ function $\mathcal U$ with a positive definite Hessian serves as an entropy function of \eqref{eq:hcl-d} if and only if \eqref{eq:hcl-v} is a symmetrization of \eqref{eq:hcl-d}, 
	\end{lemma}
	
	If the system admits an entropy pair, then left-multiplying \eqref{eq:hcl-d} by $\mathcal U'(\mathbf u)$ yields
	\begin{equation}\label{eq:EC}
		\frac{ \partial\mathcal U(\mathbf u)}{\partial t}+\sum\limits_{i=1}^d\frac{\partial \mathcal F_i(\mathbf u)}{\partial x_i} = 0
	\end{equation}
	for smooth solutions, while for discontinuous solutions, the vanishing viscosity method yields
	\begin{equation}\label{eq:ES}
		\frac{ \partial\mathcal U(\mathbf u)}{\partial t}+\sum\limits_{i=1}^d\frac{\partial \mathcal F_i(\mathbf u)}{\partial x_i} \le  0
	\end{equation}
	in the sense of distributions \cite{leveque1992numerical}, i.e., for any smooth test function $\phi\in C_0^{\infty}(\mathbb R^d\times \mathbb R^+)$ with $\phi \ge 0$, there holds
	\begin{equation}\label{eq:ESweak}
		\int_{\mathbb R^+}\int_{\mathbb R^d} \left( \mathcal U(\mathbf u)\phi_t+\boldsymbol{\mathcal F}(\mathbf u)\cdot\nabla\phi\right) \mathrm d\mathbf x\mathrm dt   \ge -\int_{\mathbb R^d}\mathcal U(\mathbf u_0(\mathbf x))\phi(\mathbf x,0)\,\mathrm d\mathbf x.
	\end{equation}
	This is called the \textit{entropy inequality} or \textit{entropy condition}. A solution $\mathbf u$ is called an \textit{entropy solution} if $\mathbf u$ satisfies \eqref{eq:ES} (or \eqref{eq:ESweak}) for any entropy pair. For scalar problems, it is proven that the entropy solution is unique, while the uniqueness is still an open problem for systems. Integrating \eqref{eq:ES} on $\mathbb R^d$ yields
	\begin{equation}\label{eq:ESglobal}
		\frac{\mathrm d}{\mathrm dt}\int_{\mathbb R^d}\mathcal U(\mathbf u)\mathrm d\mathbf x \le 0,
	\end{equation}
	which is called the ES property. That is, the total entropy of an entropy solution is non-increasing in time. In the fully-discrete sense, this corresponds to
	\begin{equation}\label{eq:ESglobal_fully}
		\int_{\mathbb R^d}\mathcal U(\mathbf u(\mathbf x,t+\Delta t))\mathrm d\mathbf x\le \int_{\mathbb R^d}\mathcal U(\mathbf u(\mathbf x,t))\mathrm d\mathbf x,\quad \forall t,\Delta t>0.
	\end{equation}
	In the following, we focus on constructing ES DG methods that approximate the entropy solution and preserve the fully-discrete ES property \eqref{eq:ESglobal_fully}. For clarity of presentation, we restrict our discussion to one-dimensional (1D) scalar equations and systems of conservation laws. The extension to higher dimensions on tensor-product meshes follows in a dimension-by-dimension manner and is therefore straightforward. In particular, we will assume $\mathcal U\in C^2(\mathcal D)$ is strictly convex to facilitate the analysis.

	\section{Fully-discrete explicit entropy stable scheme for scalar problems}
	\label{sec3}
	
	\subsection{DG formulation}
	
	In this section, we consider the 1D scalar conservation law
	\begin{equation}\label{eq:hcl}
		\left\{\begin{array}{l} u_t+f(u)_x=0,\quad x\in\mathbb R,\ t>0,\\ u(x,0)=u_0(x). \end{array}\right.
	\end{equation}
    Let the spatial domain be partitioned into uniform cells 
    $$ \mathbb R=\bigcup\limits_{i\in\mathbb Z} I_i,\quad I_i=[x_{i-1/2},x_{i+1/2}],
    $$
    with mesh size $\Delta x = x_{i+1/2}-x_{i-1/2} $. The cell centers are defined by $x_i = (x_{i+1/2}+x_{i-1/2})/2$. Define the finite element space as
	\begin{equation}
		V_h^k = \{w(x): w(x)|_{I_i}\in \mathbb P^k(I_i),\  \forall i  \}, 
	\end{equation}
	where $\mathbb P^k(I_i)$ is the space of polynomials of degree at most $k$ on cell $I_i$. Then, the semi-discrete DG scheme reads: Find $u_h\in V_h^k$ such that for all $w\in V_h^k$ and every cell $I_i$,
	\begin{equation}\label{eq:DG_standard}
		\int_{I_i}{\frac{\partial u_h}{\partial t} w \,\mathrm{d}x}
		=\int_{I_i}{f\left( u_h \right) \frac{\partial w}{\partial x}\mathrm{d}x}
		-\hat{f}_{i+1/2}w _{i+1/2}^{-}+\hat{f}_{i-1/2}w _{i-1/2}^{+}.
	\end{equation}
	Here, $\hat f_{i+1/2} = \hat f (u_h|^-_{i+1/2}, u_h|^+_{i+1/2})$ denotes a consistent, Lipschitz continuous, and monotone numerical flux defined at $x_{i+1/2}$. In this work, we employ the local Lax--Friedrichs (LF) flux
	\begin{equation}\label{eq:LF} 
	\hat f(u^-,u^+)=\frac{1}{2}(f(u^+)+f(u^-))-\frac{1}{2}\alpha(u^+-u^-),\quad \alpha=\max\limits_{u\in I(u^-,u^+)}\left|f'(u)\right|. 
	\end{equation}
    Here, $I(a,b)$ represents the interval with endpoints $a$ and $b$. The partial derivatives of $\hat f$ with respect to $u^-$ and $u^+$ are denoted by $\hat f_1$ and $\hat f_2$, respectively, with the property $\hat f_1\geq0$ and $\hat f_2\leq0$. Within the method of lines (MOL) framework, the semi-discrete DG discretization \eqref{eq:DG_standard} leads to an ODE system
	\begin{equation}\label{eq:MOL} 
		\frac{\mathrm d u_h}{\mathrm dt} = \mathcal L_h(u_h),
	\end{equation}
	where $\mathcal L_h(u_h)\in V_h^k$. The solution is then advanced in time by employing suitable time discretization methods, such as the forward Euler method, an SSP-RK method \cite{ketcheson2008highly}, or an SSP multistep method \cite{gottlieb2001strong}.

	\subsection{Forward Euler ES scheme}
	
	We assume the time domain is also discretized using uniform time levels
	$$ 0=t^0< t^1 < \cdots < t^n < \cdots,\quad t^n=n\Delta t. $$
	For illustrative purposes, we first consider forward Euler time discretization. For this case, the fully-discrete scheme is given by
	\begin{equation}\label{eq:EFDG0} 
		u_h^{n+1}=u_h^n+\Delta t \mathcal L_h(u_h^n).
	\end{equation}
	However, it is well-known that the scheme \eqref{eq:EFDG0} does not necessarily satisfy the ES property.  We therefore denote the intermediate solution computed from \eqref{eq:EFDG0} as $u_h^{n+1,(\mathrm{pre})}$, i.e.,
	\begin{equation}\label{eq:EFDG} 
		u_h^{n+1,(\mathrm{pre})}=u_h^n+\Delta t \mathcal L_h(u_h^n).
	\end{equation}
	Below, we introduce a correction step applied to $u_h^{n+1,(\mathrm{pre})}$ to obtain $u_h^{n+1}$, which satisfies the fully-discrete entropy inequality in an appropriate sense. Moreover, to preserve conservation, the correction is required to maintain the cell averages, namely,
	$\bar u_{i}^{n+1}=\bar u_{i}^{n+1,(\mathrm{pre})}$. Here, $$\bar u_i^n:=\displaystyle\frac{1}{\Delta x}\int_{I_i}u_h^n(x)\mathrm dx$$ denotes the cell average of $u_h^n$ on $I_i$.

   We begin by recalling the concept of a proper numerical entropy flux
    \cite{zhao1997muscl,karlsen2014l1}, which is used in
    \cite{kivva2022entropy} to construct fully-discrete first-order ES schemes.
	
	\begin{definition}\label{def:proper-nef}
		A two-point flux function $\hat{\mathcal F}(u^-,u^+)$ associated with the numerical flux $\hat f(u^-,u^+)$ is called a proper numerical entropy flux corresponding to the entropy pair $(\mathcal U,\mathcal F)$ if it is consistent, i.e., $\hat{\mathcal F}(u,u)=\mathcal F(u)$, and satisfies
		\begin{equation}
			\label{eq:properflux}
			\frac{\partial \hat{\mathcal{F}}(u^-,u^+)}{\partial u^{\pm}}=\mathcal{U}' ( u^{\pm} ) \frac{\partial \hat{f}(u^-,u^+)}{\partial u^{\pm}}.
		\end{equation}
		Hence, a proper numerical entropy flux preserves the relation \eqref{eq:entropy}.
	\end{definition}
    
	As an example, for the LF flux \eqref{eq:LF}, the corresponding proper numerical entropy flux is given by \cite{kivva2022entropy} 
	\begin{equation}\label{eq:LFF} \hat {\mathcal F}(u^-,u^+)=\frac{1}{2}(\mathcal F(u^+)+\mathcal F(u^-))-\frac{1}{2}\alpha(\mathcal U(u^+)-\mathcal U(u^-)), 
	\end{equation}
	which satisfies the condition \eqref{eq:properflux}. Here, $\alpha$ should be regarded as an independent parameter and held fixed when taking partial derivatives with respect to the left and right states.

    \begin{remark}
The requirement that $\alpha$ be independent of the two arguments
of the numerical flux is important here. If $\alpha$ is instead
regarded as a state-dependent function $\alpha(u^-,u^+)$, then
its derivatives with respect to $u^-$ and $u^+$ generally prevent
\eqref{eq:LFF} from satisfying Definition~\ref{def:proper-nef}.
In the analysis, we therefore regard $\alpha$ as a fixed wavespeed
bound over the range of numerical states. It may be updated between
time steps according to the current solution range, provided that
it is held fixed when defining the numerical flux for each time step.
This treatment does not appear to cause significant difficulties
in practice, as also demonstrated by the numerical experiments.
\end{remark}

    \begin{remark} 
    For a fixed strictly convex entropy, 
    we provide necessary and sufficient conditions for existence of proper numerical entropy fluxes in Appendix~\ref{app:properflux}. Roughly speaking,  \textit{a proper $\hat{\mathcal F}$ corresponding to $\hat f$ exists uniquely if and only if $\hat f$ is additively separable}, i.e.
    $$ \hat f(a,b)=h(a)+g(b). $$
    This property is closely related to the entropy-flux-splitting framework; see \cite{chen1995entropy}. 
\end{remark}

In the following, for any given entropy pair $(\mathcal U,\mathcal F)$, we assume that $\hat f$ admits a corresponding proper numerical entropy flux
$\hat{\mathcal F}$.  Entropy stability can be established for any entropy pair with the first-order three-point finite volume scheme
	\begin{equation}\label{eq:1st} 
	\bar u_i^{n+1}=\bar u_i^n-\lambda\left( \hat f(\bar u_i^n, \bar u_{i+1}^n)-\hat f(\bar u_{i-1}^n,\bar u_i^n) \right),\quad \lambda=\Delta t/\Delta x,
	\end{equation}
	as summarized in the following lemma. A proof can be found in \cite{kivva2022entropy}. This result forms the foundation for the generalization to high-order ES DG schemes.
	
	\begin{lemma}\label{lem:ES1st}
		For any given entropy pair $(\mathcal U,\mathcal F)$, the first-order scheme \eqref{eq:1st} satisfies the ES property
		\begin{equation}\label{eq:ES1st} \mathcal U(\bar u_i^{n+1})\le \mathcal U(\bar u_i^n)-\lambda(\hat{\mathcal F}(\bar u_{i}^n,\bar u_{i+1}^n)-\hat{\mathcal F}(\bar u_{i-1}^n,\bar u_{i}^n)) \end{equation}
		under the CFL condition 
			\begin{equation}\label{eq:CFLBP}
				\lambda\le \left(\hat f_1(\bar u_{i}^n,\bar u_{i+1}^n)-\hat f_2(\bar u_{i-1}^n,\bar u_i^n)\right)^{-1} ,
			\end{equation} 
			\begin{equation}\label{eq:CFL1st}
				\lambda \le 2\frac{\mathcal{U} '\left( \bar{u}_i^n \right) \left[ \hat{f}\left( \bar{u}_i^n,\bar{u}_{i+1}^n \right) -\hat{f}\left( \bar{u}_{i-1}^n,\bar{u}_i^n \right) \right] -\hat{\mathcal{F}}\left( \bar{u}_i^n,\bar{u}_{i+1}^n \right) +\hat{\mathcal{F}}\left( \bar{u}_{i-1}^n,\bar{u}_i^n \right)}{\left( \max\limits_{u\in [m_i,M_i]} \mathcal{U} ''\left( u \right) \right)\left[ \hat{f}\left( \bar{u}_i^n,\bar{u}_{i+1}^n \right) -\hat{f}\left( \bar{u}_{i-1}^n,\bar{u}_i^n \right) \right] ^2} 
			\end{equation} 
			for each $i$. Here, $m_i=\min\{\bar u_{i-1}^n,\bar u_i^n,\bar u_{i+1}^n\}$ and $M_i=\max\{\bar u_{i-1}^n,\bar u_i^n,\bar u_{i+1}^n\}$.
		\end{lemma}

		\begin{remark}
			
			The CFL condition \eqref{eq:CFLBP} introduced in \cite{zhang2012maximum} ensures that the cell average at the next step satisfies $\bar u_i^{n+1}\in [m_i,M_i]$. Without this constraint, \eqref{eq:CFL1st} will change to the following form to ensure entropy stability:
			\begin{equation}\label{eq:CFL1stim}
				\lambda \le 2\frac{\mathcal{U} '\left( \bar{u}_i^n \right) \left[ \hat{f}\left( \bar{u}_i^n,\bar{u}_{i+1}^n \right) -\hat{f}\left( \bar{u}_{i-1}^n,\bar{u}_i^n \right) \right] -\hat{\mathcal{F}}\left( \bar{u}_i^n,\bar{u}_{i+1}^n \right) +\hat{\mathcal{F}}\left( \bar{u}_{i-1}^n,\bar{u}_i^n \right)}{\left(\max\limits_{u\in I(\bar u_i^n,\bar u_i^{n+1})} \mathcal{U} ''\left( u \right)\right)\left[ \hat{f}\left( \bar{u}_i^n,\bar{u}_{i+1}^n \right) -\hat{f}\left( \bar{u}_{i-1}^n,\bar{u}_i^n \right) \right] ^2}. 
			\end{equation} 
			
		\end{remark}

		\begin{remark}
			Note that the right-hand side (RHS) of \eqref{eq:CFL1st} will not deviate significantly from the standard CFL condition. For example, consider the linear equation $u_t+u_x=0$ with the square entropy $\mathcal U(u)=u^2/2$. Using the LF flux \eqref{eq:LF} with the proper numerical entropy flux \eqref{eq:LFF}, we have
			$$ \mathcal U'(u)=u,\quad \mathcal U''(u)=1,\quad \mathcal F(u)=\frac{u^2}{2},\quad \hat f(u^-,u^+)=u^-,\quad \hat{\mathcal F}(u^-,u^+)=\frac{(u^-)^2}{2}. $$
			By direct computation, the CFL condition   \eqref{eq:CFL1st} reduces to
			$$\lambda \le \frac{2}{1}\cdot \frac{\displaystyle\frac{(\bar{u}_{i}^n)^{2}}{2}-\bar{u}_i^n\bar{u}_{i-1}^n+\frac{(\bar{u}_{i-1}^n)^{2}}{2}}{(\bar{u}_{i}^n)^{2}-2\bar{u}_i^n\bar{u}_{i-1}^n+(\bar{u}_{i-1}^n)^{2}}=1,$$
			which coincides exactly with the standard CFL condition for the linear advection equation. Moreover, it is proven in \cite{kivva2022entropy} that the numerator of \eqref{eq:CFL1st} is non-negative.
		\end{remark}
        
        \begin{remark}
        Lemma \ref{lem:ES1st} states that the scheme \eqref{eq:1st} with a fixed numerical flux is ES for \emph{any} given entropy pair in the sense of \eqref{eq:ES1st}, provided that the time step size is sufficiently small. This does not contradict the fact that the proper numerical entropy flux $\hat{\mathcal{F}}$ in \eqref{eq:properflux} varies with the choice of the entropy function $\mathcal{U}$, since $\hat{\mathcal{F}}$ is not used in the construction of the numerical scheme. Indeed, the lemma indicates that for any entropy function $\mathcal{U}$, the entropy production on a given cell should be measured by an entropy flux $\hat{\mathcal{F}}$ derived from that specific entropy. This lemma provides the foundation for the simultaneous satisfaction of multiple cell entropy inequalities by a given numerical scheme. In addition, other fully discrete first-order monotone schemes can also satisfy entropy inequalities for all convex entropies. Relevant results can be found in earlier literature; see, for example, \cite[pp. 142--143, Section 13.4]{leveque1992numerical} for a proof that the Godunov scheme satisfies the cell entropy inequality for an arbitrary entropy function. 
        \end{remark}

		Setting $w = 1$ in \eqref{eq:DG_standard} and applying forward Euler time discretization yields the following update formula for the cell average
		\begin{equation}\label{eq:Euler}
			\bar u_i^{n+1}=\bar u_i^n-\lambda(\hat f_{i+1/2}-\hat f_{i-1/2}).
		\end{equation}
		Although the high-order DG scheme combined with the forward Euler method \eqref{eq:EFDG} is not ES in general, it was shown in \cite{carlier2023invariant} that, even without any modification, the cell average updating scheme \eqref{eq:Euler} still satisfies an ES-like property. For completeness, we reinterpret this result and provide an alternative proof within our framework. This proof follows the approach of \cite{zhang2010positivity}, relying on Gauss--Lobatto quadrature and the convexity of $\mathcal U$. 
		
		\begin{theorem}\label{thm:ESlike}
			Let $2N-3\ge k$. Denote the Gauss--Lobatto points in cell $I_i$ by $\{\hat x_{i,q}\}_{q=1}^N$, and the quadrature weights in $I=[0,1]$ by $\{\omega_q\}_{q=1}^N$. Then, the forward Euler DG scheme \eqref{eq:Euler} satisfies the ES-like property
			\begin{equation}\label{eq:ESlike} \mathcal U(\bar u^{n+1}_i)\le \tilde{\mathcal U}_i^n-\lambda( \hat{\mathcal F}_{i+1/2}-\hat{\mathcal F}_{i-1/2} ),\end{equation}
			under the CFL condition 
			\begin{equation}\label{eq:EFCFLBP}
				\lambda\le \omega_N\left((\hat f_1)_{i+1/2} - (\hat f_2)_i
            \right)^{-1},\quad \lambda\le \omega_1\left((\hat f_1)_i-(\hat f_2)_{i-1/2}\right)^{-1},
			\end{equation} 
			\begin{equation}\label{eq:EFCFL1}
            \lambda \le \frac{2\omega_N}{\max\limits_{x} \mathcal{U} ''\left( u_h^n \right)}\frac{\mathcal{U}'\left( u_{i+1/2}^{n,-} \right) \left( \hat{f}_{i+1/2}-\hat{f}_i \right) -\hat{\mathcal{F}}_{i+1/2}+\hat{\mathcal{F}}_i}{\left( \hat{f}_{i+1/2}-\hat{f}_i \right) ^2},
			\end{equation}
			\begin{equation}\label{eq:EFCFL2}
			\lambda \le \frac{2\omega_1}{\max\limits_{x} \mathcal{U}''\left( u_h^n \right)}\frac{\mathcal{U}'\left( u_{i-1/2}^{n,+} \right) \left( \hat{f}_i-\hat{f}_{i-1/2} \right) -\hat{\mathcal{F}}_i+\hat{\mathcal{F}}_{i-1/2}}{\left( \hat{f}_i-\hat{f}_{i-1/2} \right) ^2}
			\end{equation} 
			for all $i$. Here,
			\begin{equation}\label{eq:Ubar} \tilde{\mathcal U}_i^n:=\frac{1}{\Delta x}\int_{I_i}^{\langle N\rangle}\mathcal U(u_h^n(x))\mathrm dx,\quad \hat{\mathcal F}_{i+1/2}:=\hat{\mathcal F}(u_{i+1/2}^{n,-}, u_{i+1/2}^{n,+}), \end{equation}
			and $\hat f_i:=\hat f(u_{i-1/2}^{n,+},u_{i+1/2}^{n,-}),\ \hat{\mathcal F}_i:=\hat{\mathcal F}(u_{i-1/2}^{n,+},u_{i+1/2}^{n,-})$.
			The notation $$\displaystyle\int_{I_i}^{\langle N\rangle} u\,\mathrm dx:=\Delta x\sum\limits_{q=1}^N \omega_qu(\hat x_{i,q})$$ represents the $N$-point Gauss--Lobatto quadrature of $u$ on the cell $I_i$.
			\end{theorem}
			\begin{proof}
				For the cell average, we have the exact identity 
				$$\bar u_i^n=\frac{1}{\Delta x} \int_{I_i}^{\langle N\rangle} u_h^n\mathrm dx =\displaystyle\sum\limits_{q=1}^N \omega_q u_h^n(\hat x_{i,q}).$$ 
				Denote $u_{i,q}:=u_h^n(\hat x_{i,q})$. Then, we can rewrite \eqref{eq:Euler} as
				\begin{align*}
            \bar{u}_{i}^{n+1}=&\sum_{q=1}^N{\omega _qu_{i,q}}-\lambda \left( \hat{f}\left( u_{i,N},u_{i+1,1} \right) -\hat{f}\left( u_{i-1,N},u_{i,1} \right) \right) 
						\\
						=&\sum_{q=2}^{N-1}{\omega _qu_{i,q}}+\omega _N\left\{ u_{i,N}-\frac{\lambda}{\omega _N}\left( \hat{f}\left( u_{i,N},u_{i+1,1} \right) -\hat{f}\left( u_{i,1},u_{i,N} \right) \right) \right\} 
						\\
						&+\omega _1\left\{ u_{i,1}-\frac{\lambda}{\omega _1}\left( \hat{f}\left( u_{i,1},u_{i,N} \right) -\hat{f}\left( u_{i-1,N},u_{i,1} \right) \right) \right\} 
						\\
						=:&\sum_{q=2}^{N-1}{\omega _qu_{i,q}}+\omega _NH_N+\omega _1H_1.
				\end{align*}
				Using the convexity of $\mathcal U$ and Lemma \ref{lem:ES1st} yields
				\begin{align*}
						\mathcal{U} &\left( \bar{u}_{i}^{n+1} \right) =\ \mathcal{U} \left( \omega _NH_N+\omega _1H_1+\sum_{q=2}^{N-1}{\omega _qu_{i,q}} \right) 
						\\
						\le&\  \omega _N\mathcal{U} \left( H_N \right) +\omega _1\mathcal{U} \left( H_1 \right) +\sum_{q=2}^{N-1}{\omega _q\,\mathcal{U} \left( u_{i,q} \right)}
						\\
						\le&\  \omega _N\left\{ \mathcal{U} \left( u_{i,N} \right) -\frac{\lambda}{\omega _N}\left( \hat{\mathcal{F}}\left( u_{i,N},u_{i+1,1} \right) -\hat{\mathcal{F}}\left( u_{i,1},u_{i,N} \right) \right) \right\} 
						\\
						&+\omega _1\left\{ \mathcal{U} \left( u_{i,1} \right) -\frac{\lambda}{\omega _1}\left( \hat{\mathcal{F}}\left( u_{i,1},u_{i,N} \right) -\hat{\mathcal{F}}\left( u_{i-1,N},u_{i,1} \right) \right) \right\} +\sum_{q=2}^{N-1}{\omega _q\,\mathcal{U} \left( u_{i,q} \right)}
						\\
						=&\sum_{q=1}^N{\omega _q\,\mathcal{U} \left( u_{i,q} \right)}-\lambda \left( \hat{\mathcal{F}}_{i+1/2} -\hat{\mathcal{F}}_{i-1/2} \right) .
					\end{align*}
				Here, the second inequality follows from the definitions of $H_N$ and $H_1$ along with Lemma~\ref{lem:ES1st}. This completes the proof. 
			\end{proof}
            
            \begin{remark}
            We note here that the Gauss--Lobatto quadrature rule is only used to define the discrete entropy average~$\tilde{\mathcal{U}}_i^n$. It does not need to be employed in the actual implementation of~\eqref{eq:DG_standard}. As seen from the proof and similar to the study of the bound-preserving schemes, it is not strictly necessary to choose the Gauss--Lobatto weights, and any other convex decomposition of $\bar{u}_i^n$ can be used. For the 1D case, it is known that the Gauss--Lobatto weights lead to the optimal CFL condition \cite{zhang2012maximum,cui2024optimal}; for the 2D case, the Gauss--Lobatto weights are no longer optimal. We refer to \cite{cui2023classic,cui2024optimal} for further studies of the optimal cell average decompositions. However, different decomposition selections may exert a minor influence on the compression magnitude of the limiter, which will be introduced later.     
            \end{remark}

			However, as pointed out in \cite{carlier2023invariant}, there is still a gap between \eqref{eq:ESlike} and the genuine ES property; that is, it may happen that
			\begin{equation}\label{eq:nonES} \tilde{\mathcal U}_i^{n+1,(\mathrm{pre})}> \tilde{\mathcal U}_i^n-\lambda(\hat{\mathcal F}_{i+1/2}-\hat{\mathcal F}_{i-1/2}).
			\end{equation}
			Therefore, we seek a modified solution $u_h^{n+1}$ that satisfies
			\begin{equation}\label{eq:ESreal} \tilde{\mathcal U}_i^{n+1}\le \tilde{\mathcal U}_i^n-\lambda(\hat{\mathcal F}_{i+1/2}-\hat{\mathcal F}_{i-1/2}).
			\end{equation}
			It is shown in \cite{chen2017entropy} that taking a convex combination of the numerical solution and its cell average does not increase the cell entropy. The proof relies solely on the convexity of $\mathcal U$ and can be found therein.
			
			\begin{lemma}\label{lem:ES}
				Let $\theta_i\in [0,1]$. Denote $u_h^{n+1,(\mathrm{mod})}|_{I_i}:=\bar u_i^{n+1}+\theta_i(u_h^{n+1,(\mathrm{pre})}|_{I_i}-\bar u_i^{n+1})$. Then 
				$$ \tilde{\mathcal U}_i^{n+1,(\mathrm{mod})}\le \tilde{\mathcal U}_i^{n+1,(\mathrm{pre})}. $$
			\end{lemma}

			We further define the univariate function 
            \begin{equation} \varphi(\theta):=\displaystyle\frac{1}{\Delta x}\int_{I_i}^{\langle N\rangle} \mathcal U(\bar u_i^{n+1}+\theta(u_h^{n+1,(\mathrm{pre})}-\bar u_i^{n+1}))\mathrm dx .
            \end{equation}
            According to Lemma~\ref{lem:ES}, the function $\varphi$ satisfies the following properties.
			\begin{proposition}
				$\varphi(\theta)$ is a continuous and non-decreasing function on $[0,1]$. Moreover, 
				\begin{equation}
					\varphi(0)=\mathcal U(\bar u_i^{n+1}),\qquad \varphi(1)=\tilde {\mathcal U}_i^{n+1,(\mathrm{pre})}.
				\end{equation}
			\end{proposition}

			Therefore, to achieve \eqref{eq:ESreal}, a simple and direct approach is to apply the Zhang--Shu limiter \cite{zhang2010positivity}, which rescales the intermediate solution toward its cell average, i.e., 
			\begin{equation}\label{eq:limiter} 
            u_h^{n+1}|_{I_i}=u_h^{n+1,(\mathrm{mod})}|_{I_i}=\bar u_i^{n+1}+\theta_i\,(u_h^{n+1,(\mathrm{pre})}|_{I_i}-\bar u_i^{n+1}).
			\end{equation}
			The remaining task is to determine a proper scaling parameter $\theta_i\in[0,1]$, such that
			\begin{equation}\label{eq:theta}
				\tilde {\mathcal U}_i^{n+1}\le \mathcal{U} ^{\mathrm{up}},
			\end{equation}
			while maintaining accuracy as much as possible. Here, $\mathcal U^{\mathrm{up}}:=\tilde{\mathcal U}_i^n-\lambda(\hat{\mathcal F}_{i+1/2}-\hat{\mathcal F}_{i-1/2})$ denotes the RHS in \eqref{eq:ESreal}. Importantly, such a $\theta_i$ must exist, because the limiting case with $\theta_i=0$ reduces to \eqref{eq:ESlike}.  In addition, the upper bound $\mathcal{U} ^{\mathrm{up}}$ is explicitly computable.
			
			As noted in \cite{zhang2010positivity}, the optimal value of $\theta_i$ can be obtained by solving the nonlinear equation
			\begin{equation}\label{eq:limiter_max} \frac{1}{\Delta x}\int_{I_i}^{\langle N\rangle}{\mathcal{U} \left( \bar{u}_{i}^{n+1}+\theta_i^{(\mathrm{up})} \left( u_{h}^{n+1,(\mathrm{pre})}-\bar{u}_{i}^{n+1} \right) \right) \mathrm{d}x}= \mathcal{U} ^{\mathrm{up}}.
			\end{equation}
			 However, this will introduce additional computational cost. Instead, we propose an explicit choice of $\theta_i$, given by
	    \begin{equation}
        \label{eq:theta_exp}\theta_i=\min\left\{\frac{\mathcal U^{\mathrm{up}}-\mathcal U^{\mathrm{1st}}}{\mathcal U^{\mathrm{high}}-\mathcal U^{\mathrm{1st}}}, 1\right\},
		\end{equation}
		where
			$$ \mathcal U^{\mathrm{1st}}=\mathcal U(\bar u^{n+1}_i),\quad \mathcal U^{\mathrm{high}}=\tilde{\mathcal U}_i^{n+1,(\mathrm{pre})}. $$
		Note that $\theta_i\ge0$. Indeed, Theorem \ref{thm:ESlike} ensures that $\mathcal U^{\mathrm{up}}\ge \mathcal U^{\mathrm{1st}}$, and Lemma \ref{lem:ES} ensures that $\mathcal U^{\mathrm{high}}\ge \mathcal U^{\mathrm{1st}}$. Moreover, if $\mathcal U^{\mathrm{high}} < \mathcal U^{\mathrm{up}}$, then $\theta_i=1$. In this case, the limiter is inactive on $I_i$, and hence  $u_h^{n+1}|_{I_i}=u_h^{n+1,(\mathrm{pre})}|_{I_i}$.

			The following theorem shows that the choice of $\theta_i$ introduced in \eqref{eq:theta_exp} satisfies the requirement \eqref{eq:theta}. 
						
			\begin{theorem}
				The ES limiter \eqref{eq:limiter} with $\theta_i$ given by \eqref{eq:theta_exp} enforces \eqref{eq:theta}.  
			\end{theorem}
			\begin{proof}
				Without loss of generality, we assume $\theta_i<1$. Utilizing the convexity of $\mathcal U$, it is straightforward to verify that
				\begin{equation*}\begin{aligned}
						\tilde {\mathcal U}_i^{n+1} &= \frac{1}{\Delta x} \int_{I_i}^{\langle N\rangle}{\mathcal{U} \left( \bar{u}_{i}^{n+1}+\theta _i\left( u_{h}^{n+1,(\mathrm{pre})}-\bar{u}_{i}^{n+1} \right) \right) \mathrm{d}x}
						\\
						&\le \theta _i\frac{1}{\Delta x}\int_{I_i}^{\langle N\rangle}{\mathcal{U} \left( u_{h}^{n+1,(\mathrm{pre})} \right) \mathrm{d}x}+\left( 1-\theta _i \right) \mathcal{U} \left( \bar{u}_{i}^{n+1} \right) 
						\\
						&=\theta _i\,\mathcal{U} ^{\mathrm{high}}+\left( 1-\theta _i \right) \mathcal{U} ^{\mathrm{1st}}=\mathcal{U} ^{\mathrm{up}},
					\end{aligned}
				\end{equation*}
				i.e. $u_h^{n+1}$ satisfies \eqref{eq:theta}.
			\end{proof}
			
			\begin{remark}
				When $u_h^n$ is close to a constant in cell $I_i$, division by zero may occur in formula \eqref{eq:theta_exp}. In the implementation, we introduce a truncation mechanism: If $\left|\mathcal U^{\mathrm{high}}-\mathcal U^{\mathrm{1st}}\right|< \varepsilon$, then we set $\theta_i=1$. Here, $\varepsilon$ is a small positive number close to machine precision. In this work, it is set to $10^{-14}$. 
			\end{remark}
			
			\begin{remark}
				Evidently, the original forward Euler DG scheme satisfies the entropy inequality in cell $I_i$ if and only if $\theta_i=1$. In other words, if the ES limiter is activated on $I_i$, then the original DG scheme must violate the entropy inequality locally on $I_i$ in the sense of  \eqref{eq:nonES}.
			\end{remark}

			Moreover, to preserve accuracy, $\theta_i$ is expected to be as close to 1 as possible. In fact, when $\theta_i<1$, we have
			$$ 1-\theta_i=\frac{\mathcal U^{\mathrm{high}}-\mathcal U^{\mathrm{up}}}{\mathcal U^{\mathrm{high}}-\mathcal U^{\mathrm{1st}}}. $$
            Therefore, the magnitude of $1-\theta_i$ is essentially determined by $\mathcal U^{\mathrm{high}} - \mathcal U^{\mathrm{up}}$.
			 For simplicity, we assume that $u_h^n$ and $\mathcal L_h(u_h^n)$ are sufficiently accurate approximations of $u(x,t^n)$ and $\partial_t u(x,t^n)$, respectively. Under this assumption, we will show that
            $\mathcal U^{\mathrm{high}}-\mathcal U^{\mathrm{up}}$ is a reasonably small term, so that its effect on the accuracy of $u_h^{n+1}$ can be tolerated. Due to the strong assumptions we have made, we acknowledge that the arguments below do not constitute a rigorous error estimate. However, they shed light on the effectiveness of the proposed limiting approach.
			\begin{theorem}\label{thm:accEuler}
				Suppose $u$ and $\mathcal U$ are both smooth functions. For the numerical solution of \eqref{eq:EFDG}, under the following assumptions:
					\begin{enumerate}
						\item $\partial_x^qu_h^n(x)-\partial_x^qu(x,t^n)=\mathcal O(\Delta x^{k+1-q}),\quad q=0,\cdots,k$;
						\item $\partial_x^q\mathcal L_h(u_h^n)(x)-\partial_x^q\partial_tu(x,t^n)=\mathcal O(\Delta x^{k+1-q}), \quad q=0,\cdots,k$; 
						\item $ \mathcal U(u_h^{n+1,(\mathrm{pre})}(x))-\mathcal U(u_h^n(x))=\Delta t\cdot\mathcal U'(u_h^n(x))\mathcal L_h(u_h^n)(x) + \mathcal O(\Delta t^2) $;
					\end{enumerate}
					we have the estimate
				\begin{equation}
                \label{eq:accUhigh}
					\mathcal U^{\mathrm{high}}-\mathcal U^{\mathrm{up}}=\mathcal O\left(\Delta x ^{2k+1}\Delta t+\Delta x^{2N-2}\Delta t+\Delta t^2\right).
				\end{equation}
			\end{theorem}
			
			\begin{proof}
                 	We have 
                 	\begin{equation}\label{eq:accUEF1new}\begin{aligned}
					\frac{\Delta x}{\Delta t}&\left( \mathcal{U} ^{\mathrm{high}}-\mathcal{U} ^{\mathrm{up}} \right) \\=&\,\frac{1}{\Delta t}\left( \int_{I_i}^{\langle N\rangle}{\mathcal U(u_h^{n+1,(\mathrm{pre})})\mathrm{d}x}-\int_{I_i}^{\langle N\rangle}{\mathcal U(u_h^{n})\mathrm{d}x} \right) +\hat{\mathcal F}_{i+1/2}^n-\hat{\mathcal F}_{i-1/2}^n
						\\
						=&\,\int_{I_i}^{\langle N\rangle}\mathcal L_h(u_h^n)\,\mathcal U'(u_h^n)\,\mathrm dx +\hat{\mathcal F}_{i+1/2}^n-\hat{\mathcal F}_{i-1/2}^n+\mathcal O(\Delta x\Delta t).
					\end{aligned}
				\end{equation}
                Note that $u_h^n$ and $\mathcal{L}_h(u_h^n)$ have bounded derivatives. Using the quadrature error of the Gauss--Lobatto rule, we get
                \begin{equation}\label{eq:FEDG-entropy-error}
                \begin{aligned}
                    \frac{\Delta x}{\Delta t}&\left( \mathcal{U} ^{\mathrm{high}}-\mathcal{U} ^{\mathrm{up}} \right)\\
                    = &\,\int_{I_i}\mathcal L_h(u_h^n)\,\mathcal U'(u_h^n)\,\mathrm dx +\hat{\mathcal F}_{i+1/2}^n-\hat{\mathcal F}_{i-1/2}^n+\mathcal O(\Delta x^{2N-1} + \Delta x\Delta t).
                    \end{aligned}
                \end{equation}
				Let $\Pi_h$ denote the $L^2$ projection operator onto $\mathbb P^k$ on each cell. For a smooth function $w\in W^{k+1,\infty}(I_i)$, the standard estimate gives that
					$$ \left|w-\Pi_h w\right|\le C_i\Delta x^{k+1},\quad  \forall x\in I_i $$
					where $C_i$ depends on $w$ and its derivatives on $I_i$. Hence, denoting the projection of the entropy variable by $v_h^n=\Pi_hv(u_h^n)=\Pi_h(\mathcal U'(u_h^n))$, utilizing the smoothness of $\mathcal U$ and the boundedness of $\partial_x^qu_h$, we can find a uniform $C$ independent of $i$ and  $\Delta x$, such that 
					$$ \left|v(u_h^n)-v_h^n\right|\le C\Delta x^{k+1},\quad \forall x\in \mathbb R. $$
				 With the superscript $n$ omitted, the term $\displaystyle \int_{I_i}\mathcal L_h(u_h^n)\,\mathcal U'(u_h^n)\,\mathrm dx$ can be written as				\begin{align}
                 \int_{I_i}&{\mathcal L_h(u_h)\,\mathcal U'\left( u_{h} \right) \mathrm{d}x}=\int_{I_i}{\mathcal L_h(u_h) v_{h}\mathrm{d}x}\label{eq:UaccF}
						\\
						=&\,\int_{I_i}{f\left( u_{h} \right) \left( v_{h} \right) _x\mathrm{d}x}-\hat{f}_{i+1/2}v_{i+1/2}^{-}+\hat{f}_{i-1/2}v_{i-1/2}^{+}\notag
						\\
						=&\,-\int_{I_i}{v_{h}f\left( u_{h} \right) _x\mathrm{d}x}+f_{i+1/2}^{-}v_{i+1/2}^{-}-f_{i-1/2}^{+}v_{i-1/2}^{+}-\hat{f}_{i+1/2}v_{i+1/2}^{-}+\hat{f}_{i-1/2}v_{i-1/2}^{+}\notag
						\\
                        =&\,-\int_{I_i} v(u_h)f(u_h)_x\mathrm dx +\int_{I_i}(v(u_h)-v_h)f(u_h)_x\mathrm dx\notag
                        \\
                        &+f_{i+1/2}^{-}v_{i+1/2}^{-}-f_{i-1/2}^{+}v_{i-1/2}^{+}-\hat{f}_{i+1/2}v_{i+1/2}^{-}+\hat{f}_{i-1/2}v_{i-1/2}^{+}\notag
                        \\
                        =&\,-\int_{I_i} \mathcal U'(u_h)f'(u_h)(u_h)_x\mathrm dx +\int_{I_i}(v(u_h)-\Pi_hv(u_h))f(u_h)_x\mathrm dx\notag
                        \\
                        &+f_{i+1/2}^{-}v_{i+1/2}^{-}-f_{i-1/2}^{+}v_{i-1/2}^{+}-\hat{f}_{i+1/2}v_{i+1/2}^{-}+\hat{f}_{i-1/2}v_{i-1/2}^{+}\notag
                        \\
						=&\,-\int_{I_i}{\mathcal{F} \left( u_{h} \right) _x\mathrm{d}x}+\int_{I_i}{\left( v\left( u_{h}\right)-\Pi_h v(u_h) \right)  \left( f\left( u_{h}\right) _x-\Pi _h(f\left( u_{h} \right) _x) \right)}\mathrm{d}x\notag
						\\
						&+f_{i+1/2}^{-}v_{i+1/2}^{-}-f_{i-1/2}^{+}v_{i-1/2}^{+}-\hat{f}_{i+1/2}v_{i+1/2}^{-}+\hat{f}_{i-1/2}v_{i-1/2}^{+}\notag
						\\
						=&\,-\mathcal{F} _{i+1/2}^{-}+\mathcal{F} _{i-1/2}^{+}+f_{i+1/2}^{-}v_{i+1/2}^{-}-f_{i-1/2}^{+}v_{i-1/2}^{+}-\hat{f}_{i+1/2}v_{i+1/2}^{-}+\hat{f}_{i-1/2}v_{i-1/2}^{+}\notag
						\\
						&+\mathcal O(\Delta x^{2k+3}).\notag
				\end{align}
				We remark here that $v_{i\pm 1/2}^{\pm}$ represents the value of $v_h$ but not $v(u_h)$. Combining \eqref{eq:accUEF1new} and \eqref{eq:UaccF}, we get
				\begin{align*}\frac{\Delta x}{\Delta t}\left( \mathcal U^{\mathrm{high}}-\mathcal U^{\mathrm{up}} \right) =&\,\left( \hat{\mathcal{F}}_{i+1/2}-\mathcal{F} _{i+1/2}^{-}-\left( \hat{f}_{i+1/2}-f_{i+1/2}^{-} \right) v_{i+1/2}^{-} \right) 
					\\
					&-\left( \hat{\mathcal{F}}_{i-1/2}-\mathcal{F} _{i-1/2}^{+}-\left( \hat{f}_{i-1/2}-f_{i-1/2}^{+} \right) v_{i-1/2}^{+} \right) 
					\\
					&+\mathcal{O} \left( \Delta x^{2k+3}+\Delta x^{2N-1}+\Delta x\Delta t \right). 
				\end{align*}
				For the first term, utilizing the property of the proper numerical entropy flux \eqref{eq:properflux}, we have
				\begin{align*}\hat{\mathcal{F}}_{i+1/2}-\mathcal{F} _{i+1/2}^{-}&=\hat{\mathcal{F}}\left( u_{i+1/2}^{-},u_{i+1/2}^{+} \right) -\hat{\mathcal{F}}\left( u_{i+1/2}^{-},u_{i+1/2}^{-} \right) 
					\\
					&=\int_{u_{i+1/2}^{-}}^{u_{i+1/2}^{+}}{\hat{\mathcal{F}}_2\left( u_{i+1/2}^{-},\xi \right) \mathrm{d}\xi}
					\\
					&=\int_{u_{i+1/2}^{-}}^{u_{i+1/2}^{+}}{v\left( \xi \right)\hat{f}_2\left( u_{i+1/2}^{-},\xi \right) \mathrm{d}\xi},
				\end{align*}
				and 
				$$\left( \hat{f}_{i+1/2}-f_{i+1/2}^{-} \right) v_{i+1/2}^{-} =\int_{u_{i+1/2}^{-}}^{u_{i+1/2}^{+}}{v_{i+1/2}^{-}\hat{f}_2\left( u_{i+1/2}^{-},\xi \right) \mathrm{d}\xi}.$$
				Hence, we can get the estimate
				\begin{align*}
						&\left| \hat{\mathcal{F}}_{i+1/2}-\mathcal{F} _{i+1/2}^{-}-\left( \hat{f}_{i+1/2}-f_{i+1/2}^{-} \right) v_{i+1/2}^{-} \right|
						\\
						&=\left| \int_{u_{i+1/2}^{-}}^{u_{i+1/2}^{+}}{\left( v\left( \xi \right) -v_{i+1/2}^{-} \right) \hat{f}_2\left( u_{i+1/2}^{-},\xi \right) \mathrm{d}\xi} \right|
						\\
						&\le \left|\int_{u_{i+1/2}^{-}}^{u_{i+1/2}^{+}}{\left( \left| v\left( \xi \right) -v\left( u_{i+1/2}^{-} \right) \right|+\left| v\left( u_{i+1/2}^{-} \right) -v_{i+1/2}^{-} \right| \right) \left| \hat{f}_2\left( u_{i+1/2}^{-},\xi \right) \right|\mathrm{d}\xi}\right|
						\\
						&\le \left|\int_{u_{i+1/2}^{-}}^{u_{i+1/2}^{+}}{\left( \left| v'\left( \zeta \right) \right|\left| \xi -u_{i+1/2}^{-} \right|+\left| v\left( u_{i+1/2}^{-} \right) -v_{i+1/2}^{-} \right| \right) \left| \hat{f}_2\left( u_{i+1/2}^{-},\xi \right) \right|\mathrm{d}\xi}\right|
						\\
						&\le C\Delta x^{k+1}\left| u_{i+1/2}^{+}-u_{i+1/2}^{-} \right| \le C\Delta x^{2k+2}.
					\end{align*}
				Similarly, $\hat{\mathcal{F}}_{i-1/2}-\mathcal{F} _{i-1/2}^{+}-\left( \hat{f}_{i-1/2}-f_{i-1/2}^{+} \right) v_{i-1/2}^{+} =\mathcal O(\Delta x^{2k+2}) $. Finally, we have
				\begin{equation*}
					\begin{aligned}\mathcal{U} ^{\mathrm{high}}-\mathcal{U} ^{\mathrm{up}}&=(\Delta t/\Delta x)\cdot \mathcal{O} \left( \Delta x^{2k+2}+\Delta x^{2k+3}+\Delta x^{2N-1}+\Delta x\Delta t \right) 
						\\
						&=\mathcal{O} \left( \Delta x^{2k+1}\Delta t+\Delta x^{2N-2}\Delta t+\Delta t^2 \right). 
					\end{aligned}
				\end{equation*}
                This completes the proof.
			\end{proof}

			\begin{remark}
				It follows from the proof that $\mathcal U^{\mathrm{high}}-\mathcal U^{\mathrm{up}}$ is of high order in space; consequently, the overall accuracy is dominated by the time discretization error. To address this, higher-order time discretization can be used, which will be introduced in the next section.  The preservation of high-order spatial accuracy is also analyzed in \cite{chan2025artificial}. Moreover, from \eqref{eq:accUhigh}, we know that choosing $N\ge k+2$ will not affect the spatial accuracy.
			\end{remark}

            \begin{remark}
            In practice, the error of the numerical solution for the forward Euler DG scheme satisfies $u - u_h^n = \mathcal{O}(\Delta x^{k+1} + \Delta t)$. Here, to highlight the spatial accuracy of the limiter, we do not take the temporal error at $t = t^n$ into account. This also allows us to reuse the estimate of the right-hand side in \eqref{eq:FEDG-entropy-error} when proving the higher-order cases.
            \end{remark}

	Then, we can establish the global ES property by the following theorem.
			
			\begin{theorem}
				The forward Euler DG scheme \eqref{eq:EFDG}, equipped with the ES limiter \eqref{eq:limiter}, is ES in the sense that
				\begin{equation}\label{eq:ESEF}
					\int_{\mathbb R}^{\langle N\rangle}\mathcal U\left(u_h^{n+1}\right)\mathrm dx\le \int_{\mathbb R}^{\langle N\rangle}\mathcal U(u_h^n)\mathrm dx.
				\end{equation}
				Here, $\displaystyle\int_{\mathbb R}^{\langle N\rangle}u\,\mathrm dx:=\sum\limits_{i}\int_{I_i}^{\langle N\rangle} u\,\mathrm dx$.
			\end{theorem}
			\begin{proof}
				Since the limited solution $u_h^{n+1,(\mathrm{mod})}$ satisfies \eqref{eq:ESreal}, summing over $i$ yields
				$$\sum\limits_{i}\tilde{\mathcal U}_i^{n+1}\le \sum\limits_i\tilde {\mathcal U}_i^n-\lambda(\hat{\mathcal F}_{i+1/2}-\hat{\mathcal F}_{i-1/2})=\sum\limits_{i}\tilde{\mathcal U}_i^n,$$
				which implies \eqref{eq:ESEF}.
			\end{proof}

			As is well known, the Lax–Wendroff theorem is a fundamental theoretical result stating that, under suitable assumptions, the solution of a conservative numerical scheme converges to a weak solution of the underlying conservation law \cite{leveque1992numerical, shi2018local}.
Similarly, for our fully discrete ES scheme, it can be shown under appropriate assumptions that, if the numerical solution converges, then its limit satisfies the entropy inequality \eqref{eq:ESweak}. Here, the notation $B_i=\{x:\left|x-x_i\right|\le c\Delta x\}$ with $c>1$ is defined in \cite{shi2018local}, and we can take $c=1.5$ for a uniform mesh, i.e. $B_i=I_{i-1}\cup I_i\cup I_{i+1}$.
			
			\begin{assumption}[Convergence]\label{assu:conv}
				For the initial condition, the numerical solution satisfies
				\begin{equation}\label{eq:convergeic}
					\int_{\mathbb R}\left|u_h^0(x)-u_0(x)\right|\mathrm dx\to 0\quad \mathrm{as}\ \  \Delta x\to 0,
				\end{equation}
				and $u_h$ converges boundedly 
				to a function $u^\star(x,t)$ in the sense of
				\begin{equation}\label{eq:convergeu}
					\sum\limits_{n}\int_{t^n}^{t^{n+1}}\int_{\mathbb R}\left|u_h^{n+1}(x)-u^\star(x,t)\right|\mathrm dx \mathrm dt\to 0\quad \mathrm{as}\ \  \Delta x,\Delta t\to 0.
				\end{equation}
			
            \end{assumption}
			
			\begin{assumption}[TVB-like property]\label{assu:tvd}
				The numerical solution satisfies
				\begin{equation}\label{eq:tvd}
					\sup\limits_{n}\Delta x\sum\limits_{i} \max\limits_{x\in B_i}\left|u_h^n(x)-u_h^n(x_i)\right|\to 0\quad \mathrm{as}\ \ \Delta x,\Delta t\to 0.
				\end{equation}
			\end{assumption}
            
			We remark here that Assumption \ref{assu:tvd} can be derived from the TVB property \cite{shi2018local}.  Then, the Lax–Wendroff-type theorem can be established as follows; the proof is given in Appendix \ref{app:LW1}.
				
				\begin{theorem}\label{thm:LW1}
					Suppose the numerical solution computed by \eqref{eq:EFDG} with the ES limiter \eqref{eq:limiter} satisfies Assumptions \ref{assu:conv} and \ref{assu:tvd}.  For any smooth test function $\phi\in C_0^\infty(\mathbb R\times\mathbb R^+)$ with $\phi \ge 0$, the limit solution $u^\star$ satisfies the entropy inequality in the sense of distributions
					\begin{equation}\label{eq:ESweak1}
						\int_{\mathbb R^+}\int_{\mathbb R} \left(\mathcal U(u^\star)\phi_t+\mathcal F(u^\star)\phi_x\right) \mathrm dx\mathrm dt\ge -\int_{\mathbb R}\mathcal U(u_0(x))\phi(x,0)\,\mathrm dx.
					\end{equation}
				\end{theorem}

			\subsection{High-order ES scheme}
			
			The forward Euler scheme achieves only first-order accuracy in time. To achieve high-order accuracy in time, we can employ the following $m$-step SSP multistep method:
			\begin{equation}\label{eq:MS}
				u_h^{n+1,(\mathrm{pre})}=\sum_{l=1}^m{\left( \alpha _lu_h^{n+1-l}+\Delta t\beta _l\mathcal{L}_h^l \left( u_h^{n+1-l} \right) \right)}
			\end{equation}
			with $\alpha_l\ge 0$, where
			$$
			\mathcal L_h^l=\left\{ \begin{array}{ll}
				\mathcal L_h ,& \beta_l\ge 0,\\	
				\tilde{\mathcal L}_h ,& \beta_l < 0.\\	
			\end{array}\right.
			$$
			Here, $\tilde {\mathcal L}_h $ denotes the associated 
			semi-discrete operator corresponding to stepping backward in time. 
            When $\beta_l<0$, $\tilde{\mathcal L}_h$ uses the downwind
            fluxes $\hat f(b,a)$ and $\hat{\mathcal F}(b,a)$.
            For simplicity of notation, the operator and flux symbols below denote
            those of the corresponding summand, with the tilde suppressed.
            We assume $\alpha_l>0$ whenever $\beta_l\ne0$; the CFL
            restrictions apply only to these terms. 
            The order of \eqref{eq:MS} is denoted by $r$. That is, for a smooth function $w(t)\in C^{r+1}$, we have
			\begin{equation}\label{eq:multitrun} w(t^{n+1})=\sum\limits_{l=1}^m \left(\alpha_lw(t^{n+1-l})+\Delta t\beta_lw'(t^{n+1-l})\right) + \mathcal O(\Delta t^{r+1}). 
			\end{equation}
			For the $\mathbb{P}^k$ approximation, we employ a high-order multistep method \eqref{eq:MS} for time discretization to achieve uniformly high-order accuracy in both space and time. For example, a 6-step, 4th-order SSP multistep method is given by \cite{gottlieb2001strong}
			\begin{equation}\label{eq:MS4}
				\begin{aligned}
					u_{h}^{n+1,(\mathrm{pre})}=&\,\frac{747}{1280}u_{h}^{n}+\frac{237}{128}\Delta t\mathcal{L} _h\left( u_{h}^{n} \right) +\frac{81}{256}u_{h}^{n-4}+\frac{165}{128}\Delta t\mathcal{L} _h\left( u_{h}^{n-4} \right) 
					\\&+\frac{1}{10}u_{h}^{n-5}-\frac{3}{8}\Delta t\tilde{\mathcal{L}} _h\left( u_{h}^{n-5} \right).
			\end{aligned}\end{equation}
			Since an SSP multistep method can be expressed as a convex combination of entropy stable Euler substeps, the following result follows directly.
			
			\begin{theorem}
				The SSP multistep DG method \eqref{eq:MS} satisfies the ES-like property
				$$
				\mathcal{U} \left( \bar{u}_{i}^{n+1} \right) \le \sum_{l=1}^m{\left( \alpha _l\,\tilde {\mathcal{U}}_{i}^{n+1-l}-\beta _l\lambda \left( \hat{\mathcal{F}}_{i+1/2}^{n+1-l}-\hat{\mathcal{F}}_{i-1/2}^{n+1-l} \right) \right)}
				$$
				with the CFL condition 
				\begin{equation}
					\lambda\le \frac{\alpha_l}{\beta_l}\omega_N\left((\hat f_1)_{i+1/2}^{n+1-l}-(\hat f_2)_{i}^{n+1-l}\right)^{-1},\quad \lambda\le\frac{\alpha_l}{\beta_l}\omega_1\left((\hat f_1)_{i}^{n+1-l}-(\hat f_2)_{i-1/2}^{n+1-l}\right)^{-1},
				\end{equation} 
				\begin{equation}
                \label{eq:CFL_high}
					\lambda \le \frac{2\alpha _l}{\left|\beta _l\right|}\min \left\{ A_i^{n+1-l},B_i^{n+1-l} \right\},\quad l=1,\cdots,m,\quad \beta_l\ne 0,
				\end{equation}
				where 
				$$A_i^{n+1-l}=\left. \operatorname{sign}(\beta_l)\omega_N\left( \frac{\mathcal{U}'\left( u_{i+1/2}^{-} \right) \left( \hat{f}_{i+1/2}-\hat{f}_i \right) -\hat{\mathcal{F}}_{i+1/2}+\hat{\mathcal{F}}_i}{\left( \max\limits_{x} \mathcal{U}''\left( u_h\right) \right)\left( \hat{f}_{i+1/2}-\hat{f}_i \right) ^2} \right) \right|_{t=t^{n+1-l}},
				$$
				$$
				B_i^{n+1-l}=\left. \operatorname{sign}(\beta_l)\omega_1\left(  \frac{\mathcal{U}'\left( u_{i-1/2}^{+} \right) \left( \hat{f}_i-\hat{f}_{i-1/2} \right) -\hat{\mathcal{F}}_i+\hat{\mathcal{F}}_{i-1/2}}{\left( \max\limits_{x}\mathcal U''(u_h) \right)\left( \hat{f}_i-\hat{f}_{i-1/2} \right) ^2} \right) \right|_{t=t^{n+1-l}}.$$
		
			\end{theorem}
			
			We remark that the discussion in \cite{kivva2022entropy} ensures that $A,B \ge 0$ in the above theorem, as well. 
			Then, we can also use the Zhang--Shu limiter with the parameter \eqref{eq:theta_exp} to scale the polynomial after one time step \eqref{eq:MS}. Here, $\mathcal U^{\mathrm{up}}$ is replaced by the high-order approximation
			\begin{equation}\label{eq:Uuphigh}
				\mathcal U^{\mathrm{up}} = \displaystyle \sum_{l=1}^m{\left( \alpha _l\,\tilde {\mathcal{U}}_{i}^{n+1-l}-\beta _l\lambda \left( \hat{\mathcal{F}}_{i+1/2}^{n+1-l}-\hat{\mathcal{F}}_{i-1/2}^{n+1-l} \right) \right)} .
			\end{equation}
			The following theorem indicates that the limiter maintains the high-order accuracy.

			\begin{theorem}\label{thm:acc}
				Let $k\ge 1$. For a uniformly strongly convex entropy function $\mathcal U$, 
				\begin{itemize}
					\item[(1)] the ES limiter satisfies
					\begin{equation}\label{eq:thm_add}
						\left| u_{h}^{n+1}-u_{h}^{n+1,(\mathrm{pre})} \right| \cdot \max_{x\in I_i}\left|u_h^{n+1,(\mathrm{pre})}-\bar{u}_i^{n+1}\right| \leq \frac{C_k}{\min\limits_{u}\mathcal U''(u)} \left( \mathcal{U} ^{\mathrm{high}}-\mathcal{U} ^{\mathrm{up}} \right)_+
					\end{equation}
					for all $x\in I_i$. 
Here, $z_+:=\max\{z,0\}$.
 Here, $C_k$ is a constant only depending on $k$.

					\item[(2)] Suppose $u$ and $\mathcal U$ are both smooth functions. Under the following assumptions:
						\begin{enumerate}
							\item[(i)] $\partial_x^qu_h^{n+1-l}$ and $\partial _x^qu_h^{n+1,(\mathrm{pre})}$ are $(k+1-q)$th-order approximations of $\partial_x^qu$ at the corresponding time layer, $l=1,\cdots,m,\ q=0,\cdots,k$;
							\item[(ii)] $\partial_x^q \mathcal L_h(u_h^{n+1-l})$ are $(k+1-q)$th order approximations of $\partial_x^q\partial_tu$ at the  corresponding time layer, $l=1,\cdots,m$,\ $q=0,\cdots,k$;
							\item[(iii)] The numerical entropy satisfies the truncation-error relation for the multistep method \eqref{eq:multitrun}, i.e. \begin{align*}\mathcal U(u_h^{n+1,(\mathrm{pre})})=&\,\displaystyle\sum\limits_{l=1}^m\left( \alpha_l\mathcal U(u_h^{n+1-l})+\Delta t\beta_l\mathcal L_h(u_h^{n+1-l})\,\mathcal U'(u_h^{n+1-l}) \right)
                            \\ &+\mathcal O(\Delta t^{r+1}), 
                            \end{align*}
						\end{enumerate}
					we have $\mathcal U^{\mathrm{high}}-\mathcal U^{\mathrm{up}}=\mathcal O(\Delta x^{2k+1}\Delta t+\Delta x^{2N-2}\Delta t+\Delta t^{r+1})$, and the ES limiter \eqref{eq:limiter} satisfies
					\begin{equation}\label{eq:acck+1}
						u_h^{n+1}=u_h^{n+1,(\mathrm{pre})}+\frac{\mathcal O(\Delta x^{2k+1}\Delta t+\Delta x^{2N-2}\Delta t+\Delta t^{r+1})}{\mathcal O(\Delta x)}.
					\end{equation}

					\item[(3)] If we further assume $\Delta t=\mathcal O(\Delta x),\ N\ge k+2,\ r\ge k+2$, and the second-order spatial derivative of $u(x,t^{n+1})$ is nonzero at each extreme point, then for $k\ge 2$, we have
					$$u_h^{n+1}=u_h^{n+1,(\mathrm{pre})}+\mathcal O(\Delta x^{k+1}).$$
				\end{itemize}
			\end{theorem}
			
			\begin{proof}

First, we prove (1). If $\theta_i=1$, the left-hand side
of \eqref{eq:thm_add} vanishes, so the estimate holds immediately.
We therefore assume $\theta_i<1$, for which
$\mathcal U^{\mathrm{high}}>\mathcal U^{\mathrm{up}}$ and
$(\mathcal U^{\mathrm{high}}-\mathcal U^{\mathrm{up}})_+
=\mathcal U^{\mathrm{high}}-\mathcal U^{\mathrm{up}}$.
 According to the definition of $\theta_i$, in cell $I_i$ we have 
				\begin{align}
						&\left| u_{h}^{n+1}-u_{h}^{n+1,(\mathrm{pre})} \right| \label{eq:acc1}
						\\ &= \left( 1-\theta _i \right) \left| u_{h}^{n+1,(\mathrm{pre})}-\bar{u}_{i}^{n+1} \right|=\frac{\left( \mathcal{U} ^{\mathrm{high}}-\mathcal{U} ^{\mathrm{up}} \right) \left| u_{h}^{n+1,(\mathrm{pre})}-\bar{u}_{i}^{n+1} \right|}{\mathcal{U} ^{\mathrm{high}}-\mathcal{U} ^{\mathrm{1st}}}\notag
						\\ &= \frac{\left( \mathcal{U} ^{\mathrm{high}}-\mathcal{U} ^{\mathrm{up}} \right) \left| u_{h}^{n+1,(\mathrm{pre})}-\bar{u}_{i}^{n+1} \right|}{\displaystyle\frac{ 1 }{\Delta x}\int_{I_i}^{\langle N\rangle}(\mathcal U(u_h^{n+1,(\mathrm{pre})})-\mathcal U(\bar u_i^{n+1}))\mathrm dx}\notag
						\\ &= \frac{\left( \mathcal{U} ^{\mathrm{high}}-\mathcal{U} ^{\mathrm{up}} \right) \left| u_{h}^{n+1,(\mathrm{pre})}-\bar{u}_{i}^{n+1} \right|}{\displaystyle\frac{ 1 }{\Delta x}\int_{I_i}^{\langle N\rangle} \left[\mathcal U'(\bar u_i^{n+1})(u_h^{n+1,(\mathrm{pre})}-\bar u_i^{n+1})+\frac{1}{2}\mathcal U''(\xi)(u_h^{n+1,(\mathrm{pre})}-\bar u_i^{n+1})^2\right]\mathrm dx}\notag
						\\ &\le  \frac{\left( \mathcal{U} ^{\mathrm{high}}-\mathcal{U} ^{\mathrm{up}} \right) \left| u_{h}^{n+1,(\mathrm{pre})}-\bar{u}_{i}^{n+1} \right|}{\displaystyle\frac{ 1 }{2\Delta x}\left(\min\limits_u\mathcal U''(u)\right)\int_{I_i}^{\langle N\rangle}\left(u_h^{n+1,(\mathrm{pre})}-\bar u_i^{n+1}\right)^2\mathrm dx}.\notag
				\end{align}
				Hence,
				\begin{equation}\begin{aligned}
						\left| u_{h}^{n+1}-u_{h}^{n+1,(\mathrm{pre})} \right|& \max_{x\in I_i}\left|u_h^{n+1,(\mathrm{pre})}-\bar{u}_i^{n+1}\right| 
						\\ &\le  \frac{\left( \mathcal{U} ^{\mathrm{high}}-\mathcal{U} ^{\mathrm{up}} \right) \max\limits_{x\in I_i}\left| u_{h}^{n+1,(\mathrm{pre})}-\bar{u}_{i}^{n+1} \right|^2}{\displaystyle\frac{ 1 }{2\Delta x}\left(\min\limits_u\mathcal U''(u)\right)\int_{I_i}^{\langle N\rangle}(u_h^{n+1,(\mathrm{pre})}-\bar u_i^{n+1})^2\mathrm dx}.
					\end{aligned}
				\end{equation}
				Utilizing norm equivalence, we know there exists a positive constant $C_k$ only depending on $k$, such that
				\begin{equation}
					\frac{\max\limits_x\left|q(x)\right|^2}{\displaystyle\frac{ 1 }{2\Delta x}\int_{I_i}^{\langle N\rangle}q^2\mathrm dx} \leq C_k,\quad \forall q\in \mathbb P^k(I_i),\  \int_{I_i}q\,\mathrm dx=0.
				\end{equation}
				Using the above estimate, since $\mathcal{U}$ is strictly convex, we have
				\begin{equation}
					\left| u_{h}^{n+1}(x)-u_{h}^{n+1,(\mathrm{pre})}(x) \right| \max_{x\in I_i}\left|u_h^{n+1,(\mathrm{pre})}-\bar{u}_i^{n+1}\right| \leq \frac{C_k}{\min\limits_u \mathcal U''(u)} \left( \mathcal{U} ^{\mathrm{high}}-\mathcal{U} ^{\mathrm{up}} \right)_+.
				\end{equation}

				For (2), similar to the forward Euler case, we have
                \begin{align*}
					&\frac{\Delta x}{\Delta t}\left( \mathcal{U} ^{\mathrm{high}}-\mathcal{U} ^{\mathrm{up}} \right) 
					\\
					&=\frac{1}{\Delta t}\int_{I_i}^{\langle N\rangle}{\left( \mathcal{U} \left( u_{h}^{n+1,\left( pre \right)} \right) -\sum_{l=1}^m{\alpha _l\mathcal{U} \left( u_{h}^{n+1-l} \right)} \right) \mathrm{d}x}+\sum_{l=1}^m{\beta _l}\left( \hat{\mathcal{F}}_{i+1/2}^{n+1-l}-\hat{\mathcal{F}}_{i-1/2}^{n+1-l} \right) 
					\\
					&=\sum_{l=1}^m{\beta _l\left\{ \int_{I_i}{\mathcal L_h(u_h^{n+1-l})\,\mathcal{U}'\left( u_{h}^{n+1-l} \right) \mathrm{d}x}+\hat{\mathcal{F}}_{i+1/2}^{n+1-l}-\hat{\mathcal{F}}_{i-1/2}^{n+1-l} \right\} }\\&+\mathcal{O} \left(\Delta x^{2N-1} + \Delta x\Delta t^r \right).
				\end{align*}
				Under the assumptions, using the analysis in Theorem \ref{thm:accEuler} yields 
				$$ \int_{I_i}{\mathcal L_h(u_h^{n+1-l})\,\mathcal{U}'\left( u_{h}^{n+1-l} \right) \mathrm{d}x}+\hat{\mathcal{F}}_{i+1/2}^{n+1-l}-\hat{\mathcal{F}}_{i-1/2}^{n+1-l} =\mathcal O(\Delta x^{2k+2}),\quad l=1,\cdots,m.$$
				Hence, we have $$\mathcal U^{\mathrm{high}}-\mathcal U^{\mathrm{up}}=\mathcal O(\Delta x^{2k+1}\Delta t+\Delta x^{2N-2}\Delta t+\Delta t^{r+1}).$$
				Moreover, we can get $\max_{x\in I_i}|u_h^{n+1,(\mathrm{pre})}-\bar u_i^{n+1}|=\mathcal O(\Delta x)$ in cell $I_i$, and then it is straightforward to verify \eqref{eq:acck+1}.

				Finally, we prove (3). If $u_{xx}(x,t^{n+1})\ne 0$ at $x_0$, which is an extreme point of $u$ in the $x$-direction, then for $k\ge 2$, according to the assumption $$(u_h^{n+1,(\mathrm{pre})})_{xx}(x_0)-u_{xx}(x_0,t^{n+1})=\mathcal O(\Delta x^{k-1}),$$
				we have $(u_h^{n+1,(\mathrm{pre})})_{xx}(x_0)\ne 0$. Therefore, if the cell $I_i$ is near the extreme point, we have
				\begin{equation}\label{eq:C2}
					\max\limits_{x\in I_i}\left|u_h^{n+1,(\mathrm{pre})}(x)-\bar u_i^{n+1}\right|>C\Delta x^2. \end{equation} 
				On the other hand, using the norm equivalence on $\left\{q\in \mathbb P^k(I_i): \displaystyle\int_{I_i} q\, \mathrm{d} x = 0 \right\}$, we have
				$$ \frac{1}{2\Delta x}\int_{I_i}^{\langle N\rangle}(u_h^{n+1,(\mathrm{pre})}-\bar u_i^{n+1})^2\mathrm dx\ge \frac{1}{C_k}\max\limits_{x\in I_i}\left| u_h^{n+1,(\mathrm{pre})}-\bar u_i^{n+1} \right|^2. $$
				Hence, we can further derive that
				$$ 
				\frac{\left| u_{h}^{n+1,(\mathrm{pre})}-\bar{u}_{i}^{n+1} \right|}{\displaystyle\frac{ 1 }{2\Delta x}\int_{I_i}^{\langle N\rangle}(u_h^{n+1,(\mathrm{pre})}-\bar u_i^{n+1})^2\mathrm dx}\le \frac{C_k}{\max\limits_{x\in I_i}\left|u_h^{n+1,(\mathrm{pre})}-\bar u_i^{n+1}\right|}<\frac{C_k}{C\Delta x^2}=\mathcal O(\Delta x^{-2}).
				$$
				Using \eqref{eq:acc1} together with $\Delta t=\mathcal O(\Delta x)$, $r\ge k+2$, and $N\geq k+2$, we have
				$$ \left| u_{h}^{n+1}-u_{h}^{n+1,(\mathrm{pre})} \right|=\mathcal O(\Delta x^{2k+1}\Delta t+\Delta x^{2N-2}\Delta t+\Delta t ^{r+1})\mathcal O(\Delta x^{-2})=\mathcal O(\Delta x^{k+1}). $$
                This completes the proof.
			\end{proof}

			\begin{remark}
				We point out that \eqref{eq:thm_add} in (1) is a general estimate that imposes no smoothness requirements on the solution, making it applicable even to discontinuous cases.
			\end{remark}

			\begin{remark}
				From (3), it seems that we need to choose a multistep method of at least $(k+2)$th order to obtain the optimal accuracy. But in fact, this is a relatively rough estimate. In numerical experiments, according to \eqref{eq:acck+1}, we found that selecting a $(k+1)$th-order multistep method is sufficient for ensuring the optimal accuracy.
			\end{remark}
			
			\begin{remark}
				Regarding the potential loss of accuracy near extreme points, we adopt the idea from the TVB limiter \cite{cockburn1989tvb} to provide an alternative referred to as the ``EB limiter" (entropy bounded). In the EB limiter, the limiting coefficient $\theta_i$ is modified as follows:
					\begin{equation}\label{eq:thetaEB}
						\hat\theta_i:=\begin{cases}
							1, &\mathrm{if}\ \ \displaystyle\int_{I_i}^{\langle N\rangle}\left| u_h^{n+1,(\mathrm{pre})}-\bar u_i^{n+1} \right|\mathrm dx\le M\Delta x^3,\\
							\theta_i, & \mathrm{otherwise}.\\
						\end{cases}
					\end{equation}
					Then, the EB limiter is formulated as
					\begin{equation}\label{eq:limiterEB}
						u_h^{n+1}=\bar u_i^{n+1}+\hat\theta_i(u_h^{n+1,(\mathrm{pre})}|_{I_i}-\bar u_i^{n+1}).
					\end{equation}
					That is, we impose no special limiting on the numerical solutions near smooth extrema. However, this modification leads to a drawback: the scheme is no longer strictly ES but may cause minor, higher-order entropy growth. Nevertheless, the scheme remains entropy bounded over any finite time. In numerical experiments, we found that the performance of the ES and EB limiters is very similar, and hence the EB limiter will not be discussed further in this paper.
			\end{remark}
			
			\begin{remark}
				Another widely used explicit time discretization is the SSP-RK family. However, if we employ the SSP-RK time discretization, the limiter \eqref{eq:limiter} must be added at each RK stage, which may significantly degrade its accuracy. This loss of accuracy is indeed observed in our numerical tests and is consistent with observations reported in prior work, such as \cite{zhang2012minimum}. 
			\end{remark}
			
			Analogous to the forward Euler case, the ES property for the DG method with SSP multistep time discretization is established in the following theorem.
			
			\begin{theorem}
				The SSP multistep DG method \eqref{eq:MS} equipped with the ES limiter \eqref{eq:limiter} is ES in the sense of
				\begin{equation}\label{eq:ESMS}
					\int_{\mathbb R}^{\langle N\rangle}{\mathcal{U} \left( u_{h}^{n+1} \right) \mathrm{d}x}\le \max_{1\le l\le m} \int_{\mathbb R}^{\langle N\rangle}{\mathcal{U} \left( u_{h}^{n+1-l} \right) \mathrm{d}x}.
				\end{equation}
			\end{theorem}
			\begin{proof}
				According to the definition of the limiter, we have
				\begin{equation}\label{eq:ESKi}\tilde{\mathcal{U}}_{i}^{n+1}\le \sum_{l=1}^m{\left( \alpha _l\,\tilde {\mathcal{U}}_{i}^{n+1-l}-\beta _l\lambda \left( \hat{\mathcal{F}}_{i+1/2}^{n+1-l}-\hat{\mathcal{F}}_{i-1/2}^{n+1-l} \right) \right)}.
				\end{equation}
				Summing over $i$ yields
				$$\int_{\mathbb R}^{\langle N\rangle}{\mathcal{U} \left( u_{h}^{n+1} \right) \mathrm{d}x}\le \sum_{l=1}^m{\left( \alpha _l\int_{\mathbb R}^{\langle N\rangle}{\mathcal{U} \left( u_{h}^{n+1-l} \right) \mathrm{d}x} \right)}\le \max_{1\le l\le m} \int_{\mathbb R}^{\langle N\rangle}{\mathcal{U} \left( u_{h}^{n+1-l} \right) \mathrm{d}x}.$$
                i.e., $u_h$ satisfies \eqref{eq:ESMS}.
			\end{proof}
			
			This theorem shows that the total entropy of the numerical solution may not decrease monotonically at every time step, but the maximum total entropy during every $m$ consecutive time steps must be non-increasing. Nevertheless, we can also establish the Lax--Wendroff type theorem for SSP multistep time discretization.  The proof is given in Appendix \ref{app:LW2}.
				
				\begin{theorem}\label{thm:LW2}
					Suppose the numerical solution computed by \eqref{eq:MS} with the ES limiter \eqref{eq:limiter} and $\mathcal U^{\mathrm{up}}$ defined by \eqref{eq:Uuphigh} satisfies Assumptions \ref{assu:conv} and \ref{assu:tvd}. 
                Assume also that the starting values satisfy
                \[
                 u_h^q\to u_0\quad\text{in }L^1_{\mathrm{loc}}(\mathbb R),
                 \qquad q=1,\ldots,m-1,
                 \quad\text{as }\Delta x,\Delta t\to0.
                \]
                Then, for any smooth test function $\phi\in C_0^\infty(\mathbb R\times\mathbb R^+)$ with $\phi \ge 0$, the limit solution $u^\star$ satisfies \eqref{eq:ESweak1}.
				\end{theorem}

			\subsection{Multiple entropy inequalities}\label{sec:multi}
			
			In the previous sections, our analysis has been confined to a single entropy inequality. It is worth noting, however, that a true entropy solution must satisfy the entropy condition with respect to all admissible entropy pairs. This presents a significant challenge, as the vast majority of ES schemes in the literature are inherently limited to enforcing only one entropy inequality at a time. 
			
			To understand this limitation, consider the ES nodal DG methods developed by Chen and Shu \cite{chen2017entropy}. Their approach fundamentally relies on constructing an entropy conservative flux $f^S$, whose explicit form is inextricably linked to the choice of entropy function. A similar dependency arises in Chan's artificial viscosity framework, where the projection $v_h(x)=\Pi_h v(u_h(x))$ onto entropy variables must be performed with respect to a specific entropy function. As a consequence, both methodologies, and indeed most existing approaches, can only guarantee entropy stability for a single entropy pair.
			
			Our novel limiting-based framework, by contrast, sidesteps this fundamental obstacle entirely. The key observation is that our limiter operates as a post-processing step: it simply compresses the polynomial solution toward the cell average through a scalar coefficient $\theta_i$, leaving the underlying DG spatial discretization completely untouched. This clean separation between discretization and entropy enforcement has a powerful implication. Since each entropy pair yields its own limiting coefficient through an independent calculation, we can straightforwardly enforce multiple entropy inequalities simultaneously by taking the minimal coefficient among all candidates. 
            
            { Specifically, consider $n$ different entropy pairs
$\{(\mathcal U^{(j)},\mathcal F^{(j)})\}_{j=1}^n$.
We use a common LF parameter $\alpha$, fixed with respect to the
arguments of the numerical flux and chosen sufficiently large to
satisfy the required wavespeed bound over the range of numerical
states under consideration. Provided that the time step satisfies
the CFL restrictions associated with all prescribed entropy pairs
simultaneously, Theorem~\ref{thm:ESlike} applies to each entropy pair
and ensures that the same cell average at the next time step satisfies
all $n$ ES-like inequalities.} Then, we can obtain $n$ limiting coefficients $\{\theta_i^{(j)}\}_{j=1}^n$ via \eqref{eq:theta_exp}. Note that $\hat{\mathcal F}$ is not used for defining the DG scheme, but only for the computation of $\theta$. Therefore, we can compute $\theta^{(j)}_i$ separately for each entropy pair using its corresponding entropy flux $\hat{\mathcal F}^{(j)}$, without the need to find a single numerical entropy flux $\hat{\mathcal F}$ that is proper for all entropy pairs $(\mathcal U^{(j)},\mathcal F^{(j)})$ simultaneously. If we take $\theta_i=\min\left\{\theta_i^{(1)},\cdots\theta_i^{(n)}\right\}$, then it is obvious that \eqref{eq:ESMS} holds for each entropy pair $(\mathcal U^{(j)},\mathcal F^{(j)})$. This represents a distinctive advantage of our approach and constitutes one of the major contributions of this work. It is worth noting that this characteristic closely resembles certain invariant-domain-preserving schemes \cite{guermond2018second, carlier2023invariant, guermond2024first} or the monolithic convex limiting approach \cite{kuzmin2020monolithic, rueda2024monolithic}, which are capable of maintaining multiple invariants.

	\section{Fully-discrete explicit entropy stable scheme for systems}
	\label{sec4}

	\subsection{DG formulation}
	
	In this section, we consider the 1D hyperbolic conservation law system
	\begin{equation}\label{eq:hcl_sys}
		\mathbf u_t + \mathbf f(\mathbf u)_x = \mathbf 0,\quad x\in\mathbb R,\ t>0, 
	\end{equation}
	where $\mathbf u=(u_1,\cdots,u_p)^T\in\mathcal D\subset\mathbb R^p$ and $\mathbf f:\mathcal D\to\mathbb R^p$. Denote $\mathbf V_h^k:=[V_h^k]^p$. The semi-discrete DG scheme is given by: Find $\mathbf u_h\in \mathbf V_h^k$, such that for any $\mathbf w\in \mathbf V_h^k$ and $I_i$,
	\begin{equation}\label{eq:DG_sys}
		\int_{I_i}{\frac{\partial \mathbf{u}_h}{\partial t}\cdot \mathbf{w}\mathrm{d}x}-\int_{I_i}{\mathbf{f}\left( \mathbf{u}_h \right) \cdot \frac{\partial \mathbf{w}}{\partial x}\mathrm{d}x}+\hat{\mathbf{f}}_{i+1/2}\cdot \mathbf{w}_{i+1/2}^{-}-\hat{\mathbf{f}}_{i-1/2}\cdot \mathbf{w}_{i-1/2}^{+}=0.
	\end{equation}
	Here, $\hat{\mathbf f}_{i+1/2}=\hat{\mathbf f}(\mathbf u_h|_{i+1/2}^-,\mathbf u_h|_{i+1/2}^+)$ is the numerical flux. We also use the LF flux:
	\begin{equation}\label{eq:LF_sys}
		\hat{\mathbf f}(\mathbf u^-,\mathbf u^+)=\frac{1}{2}(\mathbf f(\mathbf u^+)+\mathbf f(\mathbf u^-))-\frac{\alpha}{2}(\mathbf u^+-\mathbf u^-),
	\end{equation}
	where the estimate of the maximum absolute wave speed $\alpha$ satisfies 
    \begin{equation}\label{eq:alphamax}
    \alpha(\mathbf u^-,\mathbf u^+) \ge  \max\{|\sigma_{\min}(\mathbf u^-,\mathbf u^+)|, |\sigma_{\max}(\mathbf u^-,\mathbf u^+)|\}.
    \end{equation}
    Here, $\sigma_{\min}(\mathbf u^-,\mathbf u^+)$ and $\sigma_{\max}(\mathbf u^-,\mathbf u^+)$ are respectively the minimum and maximum wave speeds of the Riemann problem
    $$
		\mathbf u\left( x,0 \right) =\begin{cases}
			\mathbf u^-,\quad x\le 0,\\
			\mathbf u^+,\quad x>0.\\
		\end{cases}
		$$
    We denote \eqref{eq:DG_sys} as an ODE system \begin{equation}\label{eq:DG_semisys}
			\frac{\mathrm d\mathbf u_h}{\mathrm dt}=\mathcal L_h(\mathbf u_h).
		\end{equation}

		\subsection{Forward Euler ES scheme}
		
		With forward Euler time discretization, the scheme \eqref{eq:DG_sys} is formulated as
		\begin{equation}\label{eq:Eulersys}
			\mathbf u_h^{n+1,(\mathrm{pre})}=\mathbf u_h^n+\Delta t\cdot\mathcal L_h(\mathbf u_h^n).
		\end{equation}
		The cell average updating scheme is given by
		\begin{equation}\label{eq:Euler_ubarsys}
			\bar{\mathbf u}_i^{n+1}=\bar{\mathbf u}_i^n-\lambda (\hat{\mathbf f}_{i+1/2} - \hat{\mathbf f}_{i-1/2}).
		\end{equation}
		For systems, the definition of a proper numerical entropy flux is given as follows. 
		\begin{definition}
			A two-point consistent flux function $\hat{\mathcal F}(\mathbf u^-,\mathbf u^+)$ is called proper, if
			\begin{equation}
				\frac{\partial \hat{\mathcal F}(\mathbf u^-,\mathbf u^+)}{\partial \mathbf u^\pm}=\mathcal U'(\mathbf u^\pm)\frac{\partial \hat{\mathbf f}(\mathbf u^-,\mathbf u^+)}{\partial \mathbf u^\pm}.
			\end{equation}
			That is, a proper numerical entropy flux preserves the relation \eqref{eq:entropy}.
			
		\end{definition}
		
		As an example, the corresponding proper numerical entropy flux of the LF flux \eqref{eq:LF_sys} is
		\begin{equation}
			\hat{\mathcal F}(\mathbf u^-,\mathbf u^+)=\frac{1}{2}(\mathcal F(\mathbf u^+)+\mathcal F(\mathbf u^-)) - \frac{1}{2}\alpha(\mathcal U(\mathbf u^+)-\mathcal U(\mathbf u^-)).
		\end{equation}
		Again, by employing forward Euler time discretization, the fully-discrete scheme
		\begin{equation}\label{eq:Euler_1st}
			\bar{\mathbf u}_i^{n+1}=\bar{\mathbf u}_i^n-\lambda \left( \hat{\mathbf f}(\bar{\mathbf u}_i^n,\bar{\mathbf u}_{i+1}^n) - \hat{\mathbf f}(\bar{\mathbf u}_{i-1}^n,\bar {\mathbf u}_i^n)\right),
		\end{equation}
		is ES, as stated in the following lemma; see \cite{carlier2023invariant}.
		
		\begin{lemma}\label{lem:ESsys}
			The first-order scheme \eqref{eq:Euler_1st} is ES in the sense of
			\begin{equation}\label{eq:ES1st_sys}
				\mathcal U(\bar{\mathbf u}_i^{n+1})\le \mathcal U(\bar{\mathbf u}_i^n)-\lambda (\hat{\mathcal F}(\bar{\mathbf u}^n_{i},\bar{\mathbf u}^n_{i+1})-\hat{\mathcal F}(\bar {\mathbf u}^n_{i-1},\bar{\mathbf u}^n_i))
			\end{equation}
			under the CFL condition 
			\begin{equation}\label{eq:CFL_sys}
				\lambda\cdot\max\{\alpha(\bar{\mathbf u}_{i-1}^n,\bar{\mathbf u}_i^n),\alpha(\bar{\mathbf u}_i^n,\bar{\mathbf u}_{i+1}^n)\} \le \frac{1}{2}.
			\end{equation}
			
		\end{lemma}

		We can establish the ES-like property of the high-order DG scheme with forward Euler time discretization \eqref{eq:Euler_ubarsys}. The proof is similar to that of Theorem \ref{thm:ESlike} and is hence omitted for brevity.
	
		\begin{theorem}\label{thm:ESlike_sys}
			The forward Euler DG scheme \eqref{eq:Euler_ubarsys} satisfies the ES-like property
			\begin{equation}\label{eq:ESlike_sys} \mathcal U(\bar {\mathbf u}^{n+1}_i)\le \tilde{\mathcal U}_i^n-\lambda( \hat{\mathcal F}_{i+1/2}-\hat{\mathcal F}_{i-1/2} )\end{equation}
			under the CFL condition 
			$$\lambda\cdot\max\{\alpha_i,\alpha_{i+1/2}\}\le \frac{\omega_N}{2},\quad \lambda\cdot\max\{\alpha_{i-1/2},\alpha_i\}\le \frac{\omega_1}{2}.
			$$
			Here,
			\begin{equation}\label{eq:Ubar_sys} \tilde{\mathcal U}_i^n:=\frac{1}{\Delta x}\int_{I_i}^{\langle N\rangle}\mathcal U(\mathbf u_h^n(x))\mathrm dx,\quad \hat{\mathcal F}_{i+1/2}:=\hat{\mathcal F}(\mathbf u_{i+1/2}^{n,-}, \mathbf u_{i+1/2}^{n,+}), \end{equation}
			and 
            $$\hat {\mathbf f}_i:=\hat {\mathbf f}(\mathbf u_{i-1/2}^{n,+}, \mathbf u_{i+1/2}^{n,-}),\quad \hat{\mathcal F}_i=\hat{\mathcal F}(\mathbf u_{i-1/2}^{n,+},\mathbf u_{i+1/2}^{n,-}),$$
			$$ \alpha_{i+1/2}=\alpha(\mathbf u_{i+1/2}^{n,-},\mathbf u_{i+1/2}^{n,+}),\quad \alpha_i=\alpha(\mathbf u_{i-1/2}^{n,+},\mathbf u_{i+1/2}^{n,-}).$$
		\end{theorem}
		
		Theorem \ref{thm:ESlike_sys} implies that, to achieve entropy stability, we can scale the intermediate DG solution towards the cell average by
		\begin{equation}\label{eq:limiter_sys}
			\mathbf u_h^{n+1}|_{I_i}=\bar{\mathbf u}_i^{n+1}+\theta_i(\mathbf u_h^{n+1,(\mathrm{pre})}|_{I_i}-\bar{\mathbf u}_i^{n+1}),\quad \theta_i=\min\left\{\frac{\mathcal U^{\mathrm{up}}-\mathcal U^{\mathrm{1st}}}{\mathcal U^{\mathrm{high}}-\mathcal U^{\mathrm{1st}}},1\right\},
		\end{equation}
		where
		\begin{equation}
			\mathcal U^{\mathrm{1st}}=\mathcal U(\bar{\mathbf u}_i^{n+1}),\quad  \mathcal U^{\mathrm{high}}=\tilde{\mathcal U}_i^{n+1,(\mathrm{pre})},\quad \mathcal U^{\mathrm{up}}=\tilde{\mathcal U}_i^n-\lambda(\hat{\mathcal F}_{i+1/2}-\hat{\mathcal F}_{i-1/2}),
		\end{equation}
		as summarized in the following theorem.
		
		\begin{theorem}
			The forward Euler DG scheme \eqref{eq:Eulersys} equipped with the ES limiter \eqref{eq:limiter_sys} is ES in the sense that
			\begin{equation}\label{eq:ESEF_sys}
				\int_{\mathbb R}^{\langle N\rangle}\mathcal U\left(\mathbf u_h^{n+1}\right)\mathrm dx\le \int_{\mathbb R}^{\langle N\rangle}\mathcal U(\mathbf u_h^n)\mathrm dx.
			\end{equation}
		\end{theorem}

		\subsection{High-order ES scheme}
		
		We now introduce the high-order ES scheme for systems using SSP multistep time discretization \eqref{eq:MS}. Applying this time discretization to the semi-discrete DG system \eqref{eq:DG_semisys}, we obtain the following high-order fully discrete scheme:
		\begin{equation}\label{eq:MS_sys}
			\mathbf u_h^{n+1,(\mathrm{pre})}=\sum_{l=1}^m{\left( \alpha _l\mathbf u_h^{n+1-l}+\Delta t\beta _l\mathcal{L}_h^l \left( \mathbf u_h^{n+1-l} \right) \right)}.
		\end{equation}
        The same convention for downwind fluxes and their wavespeed
        bounds applies here.
		First, the ES-like property of this method can be established as follows.
        
		\begin{theorem}
			The SSP multistep DG scheme \eqref{eq:MS_sys} satisfies the ES-like property
			\begin{equation}
				\mathcal{U} \left( \bar{\mathbf u}_{i}^{n+1} \right) \le \sum_{l=1}^m{\left( \alpha _l\,\tilde {\mathcal{U}}_{i}^{n+1-l}-\beta _l\lambda \left( \hat{\mathcal{F}}_{i+1/2}^{n+1-l}-\hat{\mathcal{F}}_{i-1/2}^{n+1-l} \right) \right)}
			\end{equation}
			with the CFL condition 
			\begin{equation}\label{eq:CFL_syshigh}
				\lambda \le \frac{\alpha _l}{2\left|\beta _l\right|}\min \left\{ A_i^{n+1-l},B_i^{n+1-l} \right\} ,\quad l=1,\cdots ,m,\quad \beta_l\ne 0,
			\end{equation}
			where
			$$A_i^{n+1-l}=\left. \frac{\omega_N}{\max\{\alpha_{i},\alpha_{i+1/2}\} }\right|_{t=t^{n+1-l}}, \quad B_i^{n+1-l}=\left.\frac{\omega_1}{\max\{ \alpha_{i-1/2},\alpha_i \}}\right|_{t=t^{n+1-l}}.
			$$
		\end{theorem}
        
		Likewise, after the time discretization, we use the limiter \eqref{eq:limiter_sys} to scale the polynomial, where $\mathcal U^{\mathrm{up}}$ is replaced by the high-order approximation
		$$\mathcal U^{\mathrm{up}} = \displaystyle \sum_{l=1}^m{\left( \alpha _l\,\tilde {\mathcal{U}}_{i}^{n+1-l}-\beta _l\lambda \left( \hat{\mathcal{F}}_{i+1/2}^{n+1-l}-\hat{\mathcal{F}}_{i-1/2}^{n+1-l} \right) \right)} .$$
		The ES property is given by the following theorem.

		\begin{theorem}
			The numerical solution of the SSP multistep DG scheme \eqref{eq:MS_sys} with the ES limiter \eqref{eq:limiter_sys} is ES in the sense of
			\begin{equation}\label{eq:ESMS_sys}
				\int_{\mathbb R}^{\langle N\rangle}{\mathcal{U} \left( \mathbf u_{h}^{n+1} \right) \mathrm{d}x}\le \max_{1\le l\le m} \int_{\mathbb R}^{\langle N\rangle}{\mathcal{U} \left( \mathbf u_{h}^{n+1-l} \right) \mathrm{d}x}.
			\end{equation}
		\end{theorem}
		
		\begin{remark}
			For systems, multiple entropy inequalities can also be enforced simultaneously by calculating the minimum limiting coefficient, as discussed in Section \ref{sec:multi}.
		\end{remark}
		
		\subsection{Compatibility with other limiters}
		
		Even if a scheme satisfies entropy stability, numerical oscillations or unphysical phenomena such as negative pressure and negative density may still occur for problems with strong discontinuities. Therefore, additional limiters are often necessary. However, there is a critical issue: Applying other limiters may increase the numerical entropy. To avoid this problem, we require that shock limiters, bound-preserving (BP) and positivity-preserving (PP) limiters should also be applied in the same form as \eqref{eq:limiter_sys}. 
		
		It is discussed in \cite{chen2017entropy} that the BP and PP limiters \cite{zhang2012maximum, zhang2010positivity} will not increase the entropy. On the other hand, for shock  limiters, characteristic decomposition-based limiters such as the TVB limiter cannot guarantee strict entropy decrease for systems, as discussed in \cite{chen2017entropy}. In recent years, many jump-based filtering moment limiters have been proposed, with the pioneering work being the OEDG method \cite{peng2024oedg}. However, the classical OE limiting technique uses different coefficients to limit moments of different orders, which violates the requirement of Lemma \ref{lem:ES}, and thus may cause entropy increase. In \cite{liu2025globally}, a modified approach was proposed in which all moments are limited using the same coefficient, thereby achieving entropy stability. 
		
		When multiple limiters are required in practical computations, the order of application is crucial.  In this work, to verify the robustness of the proposed scheme, we do not add the shock limiter for all examples. In addition, the BP/PP limiters are still needed for some extreme examples to preserve the physical structure. In the algorithm implementation, the BP/PP limiters \cite{zhang2012maximum, zhang2010positivity} are applied after the ES limiter when necessary.

		\section{Numerical tests}\label{sec5}
		
		In this work, we use the 6-step, 4th-order SSP multistep method \eqref{eq:MS4}. The time step is computed by
		$$\Delta t= \mathrm{CFL}\cdot\frac{\Delta x}{\alpha_x}$$
		for 1D cases, and 
		$$\Delta t= \frac{\mathrm{CFL}}{\alpha_x/\Delta x+\alpha _y/\Delta y}$$
		for 2D cases. Here, $\alpha_x$ and $\alpha_y$ are the estimates of the global maximum absolute wave speed in $\Omega$ of the initial data. The CFL number is taken as $0.01$ for stability. We found that this is sufficient to satisfy the CFL conditions \eqref{eq:CFL_high} and \eqref{eq:CFL_syshigh}. The $\mathbb{P}^2$ approximation is employed unless otherwise specified, and no additional limiters are used beyond the ES limiter unless explicitly noted. 
		
		For such systems, the Euler equations are considered:
		\begin{equation}
			\left[ \begin{array}{c}
				\rho\\
				\rho \mathbf u\\
				\mathcal{E}\\
			\end{array} \right] _t+\nabla\cdot\left[ \begin{array}{c}
				\rho \mathbf u\\
				\rho \mathbf u\otimes\mathbf u+p\mathbf I_d\\
				\mathbf u\left( \mathcal{E} +p \right)\\
			\end{array} \right]=\mathbf{0}.
		\end{equation}
		Here, $\rho $ denotes the density, $\mathbf u$ is the velocity, $\mathcal E$ is the total energy, $p$ is the pressure, and $\mathbf I_d$ denotes the $d\times d$ identity matrix. The system is closed with the ideal gas EOS
		\begin{equation}
			\mathcal E=\frac{p}{\gamma - 1}+\frac{1}{2}\rho \left\|\mathbf u\right\|^2.
		\end{equation}
		Unless otherwise specified, the specific heat ratio $\gamma$ is set to 1.4. A widely used entropy pair for the Euler equations is
\begin{equation}
    \mathcal U=-\frac{\rho s}{\gamma - 1},\quad
    \boldsymbol{\mathcal F}=-\frac{\rho s}{\gamma - 1}\mathbf u,\quad
    s=\ln (p\rho^{-\gamma}).
\end{equation}
For the Euler equations, the parameter $\alpha$ in \eqref{eq:LF_sys} can be computed using the RRF approximation \cite{chen2017entropy}, which strictly ensures \eqref{eq:alphamax} for $1<\gamma\le 5/3$.

		\subsection{Scalar problems}
		
		\begin{Ex}[Linear equation]\label{ex:linear}
		We begin by testing the linear advection equation,
		$$ u_t+u_x=0,$$
		which serves as a model problem for hyperbolic conservation laws. The exact solution is taken as $u(x,t)=\sin (x-t)^4$. To test the performance of our scheme, we choose two entropy pairs:
		$$ (\mathcal U^{(1)},\mathcal F^{(1)})=\left( e^u,e^u \right),\quad (\mathcal U^{(2)},\mathcal F^{(2)})=\left( \frac{u^2}{2},\frac{u^2}{2} \right). $$
		We use ``Base" to denote the standard SSP multistep DG scheme, ``ES1'' to denote the scheme satisfying only the entropy inequality of the first entropy pair, and ``ES2" to denote the scheme satisfying both entropy inequalities. The $L^2$ errors and orders are presented in Table \ref{tab1}. It can be seen that the ES limiter maintains the optimal $(k+1)$th-order accuracy for $k=1,2,3$. Moreover, the error is quite close to that of the Base scheme.  To demonstrate the effectiveness of the limiter, we present the distribution of limited cells for the ES1 scheme with $N_x=80$ in Fig. \ref{figlinear}. The results of the ES2 scheme are similar and are not shown here to avoid repetition. 
        
        We also test the performance of the ES limiter with SSP-RK time discretization. Table \ref{tabRK} reports the results obtained with the ten-stage, fourth-order SSP-RK method \cite{ketcheson2008highly}, where the ES limiter is applied at each stage to enforce entropy stability. The results show that this stage-by-stage limiting reduces the observed accuracy to first order. This loss of accuracy is more severe than what is typically observed for BP/PP limiters \cite{zhang2010positivity, zhang2012maximum} or minimum-entropy-principle limiters \cite{zhang2012minimum}. A possible reason is that those limiters usually act only on a very small fraction of cells. In contrast, as shown in Fig. \ref{figlinear}, the ES limiter is activated in a considerably larger number of cells, which leads to a more pronounced loss of accuracy when applied at every RK stage.
		\end{Ex}

		\begin{table}[htb!]
			\centering
			\caption{Example \ref{ex:linear}: Linear equation. $L^2$ errors and orders of different schemes at $T=1$.}
			\setlength{\tabcolsep}{3.2mm}{
				\begin{tabular}{|c|cc|cc|cc|}
					\hline $N$ & Base & Order & ES1 & Order & ES2 & Order  \\ 
					\hline
					\multicolumn{7}{|c|}{$k=1$}\\
					
					\hline  
					20 &3.2080E-02 & -- &3.2080E-02 & -- &3.2080E-02 &-- \\
					\hline  
					40 &6.9836E-03 & 2.1996 &6.9863E-03 & 2.1991 &6.9872E-03 &2.1989 \\
					\hline  
					80 &1.6697E-03 & 2.0644 &1.6698E-03 & 2.0649 &1.6698E-03 &2.0650 \\
					\hline  
					160 &4.1239E-04 & 2.0175 &4.1239E-04 & 2.0176 &4.1239E-04 &2.0176 \\
					\hline
					
					\multicolumn{7}{|c|}{$k=2$}\\		
					\hline  
					20 &2.4873E-03 & -- &2.4873E-03 & -- &2.4873E-03 &-- \\
					\hline  
					40 &3.3698E-04 & 2.8838 &3.4592E-04 & 2.8461 &3.4810E-04 &2.8370 \\
					\hline  
					80 &4.3265E-05 & 2.9614 &4.4242E-05 & 2.9669 &4.4301E-05 &2.9741 \\
					\hline  
					160 &5.4480E-06 & 2.9894 &5.5249E-06 & 3.0014 &5.5277E-06 &3.0026 \\
					\hline
					\multicolumn{7}{|c|}{$k=3$}\\
					
					\hline  
					20 &1.6241E-04 & -- &1.6393E-04 & -- &1.6396E-04 &-- \\
					\hline  
					40 &9.7619E-06 & 4.0564 &9.8515E-06 & 4.0566 &9.8538E-06 &4.0565 \\
					\hline  
					80 &5.9316E-07 & 4.0407 &6.0716E-07 & 4.0202 &6.0721E-07 &4.0204 \\
					\hline  
					160 &3.7012E-08 & 4.0024 &3.7021E-08 & 4.0356 &3.7025E-08 &4.0356 \\
					\hline
					
			\end{tabular}} 
			\label{tab1}
		\end{table}
		
		\begin{figure}[htb!]
			\centering
			\subfigure[$k=1$.]{
				\includegraphics[width=0.31\linewidth]{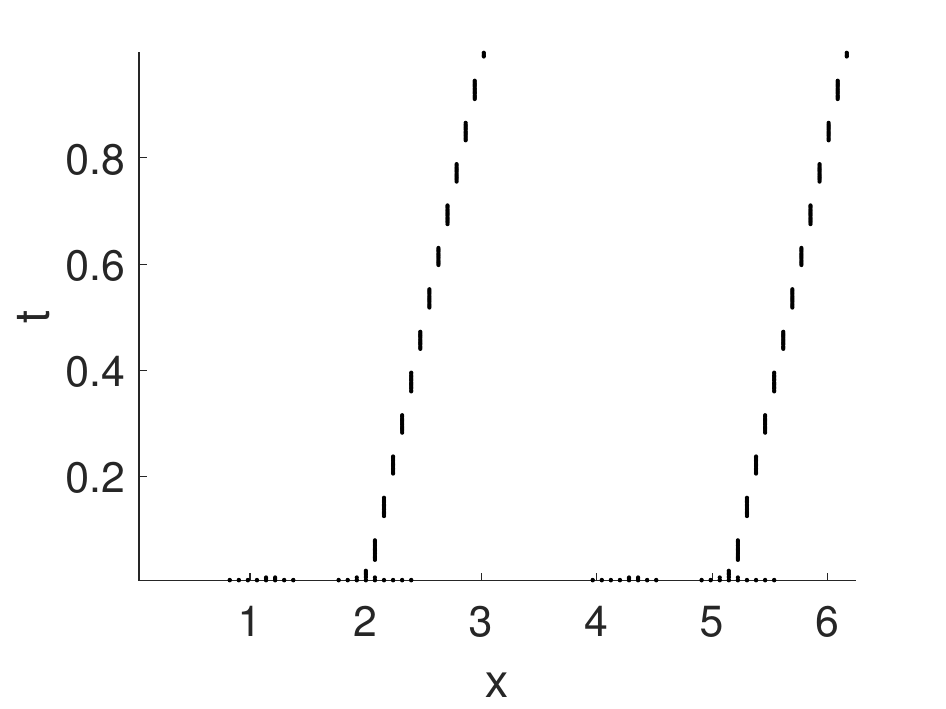}}
			\subfigure[$k=2$.]{
				\includegraphics[width=0.31\linewidth]{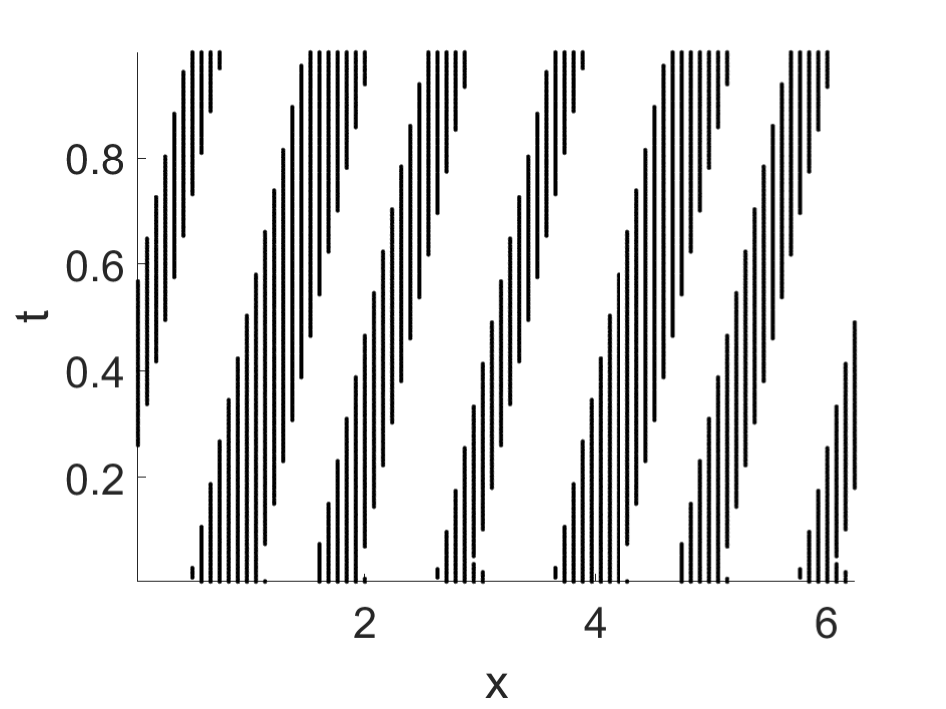}}
			\subfigure[$k=3$.]{
				\includegraphics[width=0.31\linewidth]{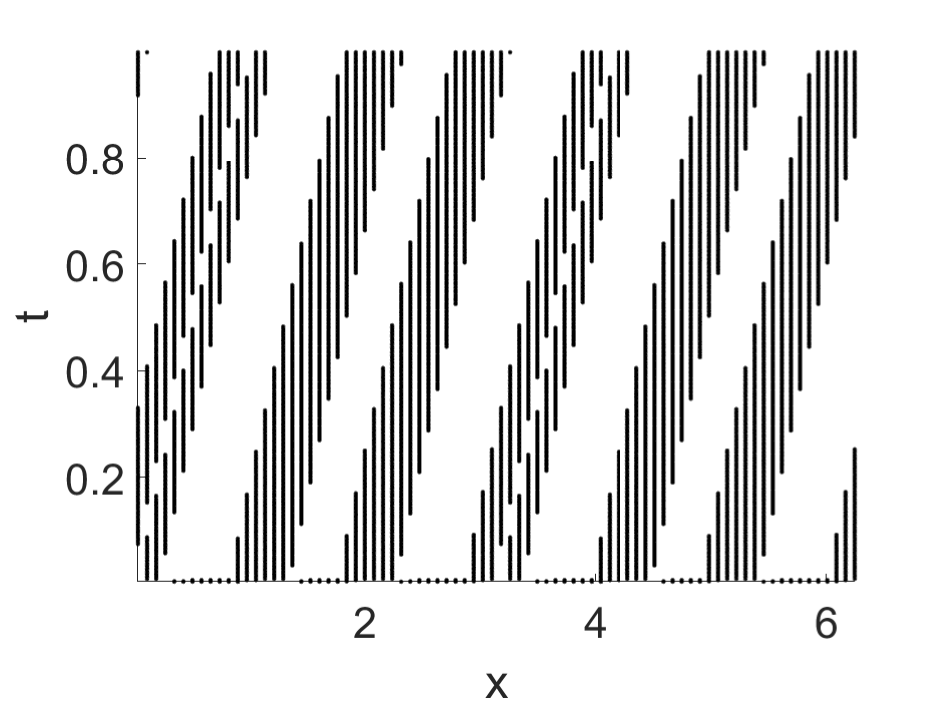}}
			\caption{Example \ref{ex:linear}: Linear equation. The $(x,t)$-distribution of limited cells of the ES1 scheme with $N_x=80$.}
			\label{figlinear}
		\end{figure}

        \begin{table}[htb!]
			\centering
			\caption{
            Example \ref{ex:linear}: Linear equation. $L^2$ errors and orders of different schemes at $T=1$ with SSP-RK4 time discretization.}
			\setlength{\tabcolsep}{3.2mm}{
				\begin{tabular}{|c|cc|cc|cc|}
					\hline $N$ & Base & Order & ES1 & Order & ES2 & Order  \\ 
					\hline
					\multicolumn{7}{|c|}{$k=1$}\\
					
					\hline  
                     20 &3.4344E-02 & -- &4.7510E-02 & -- &4.9140E-02 &-- \\
                    \hline  
                     40 &7.8975E-03 & 2.1206 &1.6998E-02 & 1.4828 &1.7348E-02 &1.5021 \\
                    \hline  
                     80 &1.9083E-03 & 2.0491 &7.1109E-03 & 1.2573 &7.2061E-03 &1.2675 \\
                    \hline  
                     160 &4.7223E-04 & 2.0147 &3.3285E-03 & 1.0952 &3.3531E-03 &1.1037 \\
                    \hline
					
					\multicolumn{7}{|c|}{$k=2$}\\		
					\hline  
                     20 &2.4029E-03 & -- &2.6137E-02 & -- &2.6346E-02 &-- \\
                    \hline  
                     40 &2.9966E-04 & 3.0034 &1.3116E-02 & 0.9948 &1.3142E-02 &1.0035 \\
                    \hline  
                     80 &3.7471E-05 & 2.9995 &6.6074E-03 & 0.9891 &6.6119E-03 &0.9910 \\
                    \hline  
                     160 &4.6844E-06 & 2.9998 &3.3141E-03 & 0.9955 &3.3152E-03 &0.9960 \\
                    \hline

					\multicolumn{7}{|c|}{$k=3$}\\
					
					\hline  
                     20 &1.6849E-04 & -- &2.6557E-02 & -- &2.6698E-02 &-- \\
                    \hline  
                     40 &1.0751E-05 & 3.9700 &1.3234E-02 & 1.0048 &1.3259E-02 &1.0098 \\
                    \hline  
                     80 &6.6631E-07 & 4.0122 &6.6261E-03 & 0.9981 &6.6305E-03 &0.9997 \\
                    \hline  
                     160 &4.1701E-08 & 3.9980 &3.3182E-03 & 0.9978 &3.3192E-03 &0.9983 \\
                    \hline
					
			\end{tabular}} 
			\label{tabRK}
		\end{table}

		\begin{Ex}[Burgers' equation]\label{ex:burgers}
		We consider Burgers' equation
		$$ u_t + \left( \frac{u^2}{2} \right)_x=0. $$
		This flux function is genuinely nonlinear, which makes this equation more challenging than the linear equation. The computational domain is $\Omega=[0,2\pi]$ with periodic boundaries. The initial data is $u(x,0)=0.5 + \sin x$. The discontinuity will appear at $T_b=1$. We also choose two types of entropy pairs 
		$$   (\mathcal U^{(1)},\mathcal F^{(1)})=\left( e^u,(u-1)e^u \right),\quad  (\mathcal U^{(2)},\mathcal F^{(2)})=\left( \frac{u^2}{2},\frac{u^3}{3} \right). $$
		Similarly, we use “Base”, “ES1” and “ES2” to denote the standard SSP multistep DG scheme, the scheme satisfying only the entropy inequality of the first entropy pair, and the scheme satisfying both entropy inequalities, respectively.
		
		First, we simulate this problem until $T=0.6$. The exact solution remains smooth at this moment. The errors and orders are presented in Table \ref{tab2} for $k=1,2,3$, and the optimal convergence rates are obtained. 
		
		Next, we simulate this problem until $T=1.075$. At this moment, a shock wave is produced. Both the Base scheme and the ES1 scheme are tested for this case with $N_x=80$ and a $\mathbb P^2$ approximation. 
		In Fig. \ref{figburgers1}, we plot the cell entropy inequality violation, defined as
		$$ \tilde{\mathcal U}_i -\mathcal U^{\mathrm{up}}|_{I_i}$$
		in each cell $I_i$
		for $(\mathcal U^{(1)},\mathcal F^{(1)})$ at the final step. For this test case, the result of ES2 is similar to that of ES1, and is not presented here.
		A positive value of this quantity indicates a violation of the entropy inequality at the corresponding cell. The numerical results confirm that our proposed scheme ensures the discrete entropy inequality is strictly satisfied across all computational cells. Fig. \ref{figburgers2} shows the evolution of total entropy over time. In contrast to the base scheme, whose total entropy increases after shock formation, our scheme maintains a non-increasing total entropy, in strict accordance with the underlying entropy condition. 
		
		\end{Ex}

		\begin{table}[htb!]
			\centering
			\caption{Example \ref{ex:burgers}: Burgers' equation. $L^2$ errors and orders of different schemes at $T=0.6$.}
			\setlength{\tabcolsep}{3.2mm}{
				\begin{tabular}{|c|cc|cc|cc|}
					\hline $N_x$ & Base & Order & ES1 & Order & ES2 & Order  \\ 
					\hline
					\multicolumn{7}{|c|}{$k=1$}\\
					
					\hline  
					20 &1.8040E-02 & -- &1.8047E-02 & -- &1.8067E-02 &-- \\
					\hline  
					40 &4.6564E-03 & 1.9539 &4.6597E-03 & 1.9535 &4.6619E-03 &1.9543 \\
					\hline  
					80 &1.2129E-03 & 1.9408 &1.2136E-03 & 1.9409 &1.2138E-03 &1.9414 \\
					\hline  
					160 &3.1259E-04 & 1.9561 &3.1272E-04 & 1.9564 &3.1272E-04 &1.9565 \\
					\hline

					\multicolumn{7}{|c|}{$k=2$}\\		
					\hline  
					20 &1.5444E-03 & -- &1.5446E-03 & -- &1.5444E-03 &-- \\
					\hline  
					40 &2.8053E-04 & 2.4608 &2.8066E-04 & 2.4604 &2.8068E-04 &2.4600 \\
					\hline  
					80 &4.0109E-05 & 2.8061 &4.0124E-05 & 2.8063 &4.0126E-05 &2.8063 \\
					\hline  
					160 &5.4394E-06 & 2.8824 &5.4407E-06 & 2.8826 &5.4408E-06 &2.8827 \\
					\hline
					
					\multicolumn{7}{|c|}{$k=3$}\\
					
					\hline  
					20 &3.4563E-04 & -- &3.4578E-04 & -- &3.4595E-04 &-- \\
					\hline  
					40 &1.9241E-05 & 4.1670 &1.9255E-05 & 4.1665 &1.9257E-05 &4.1671 \\
					\hline  
					80 &1.2262E-06 & 3.9720 &1.2267E-06 & 3.9724 &1.2267E-06 &3.9725 \\
					\hline  
					160 &7.8895E-08 & 3.9581 &7.8905E-08 & 3.9585 &7.8906E-08 &3.9585 \\
					\hline
			\end{tabular}} 
			\label{tab2}
		\end{table}

		\begin{figure}[htb!]
			\centering
			\subfigure[Cell entropy inequality violation.]{
				\includegraphics[width=0.45\linewidth]{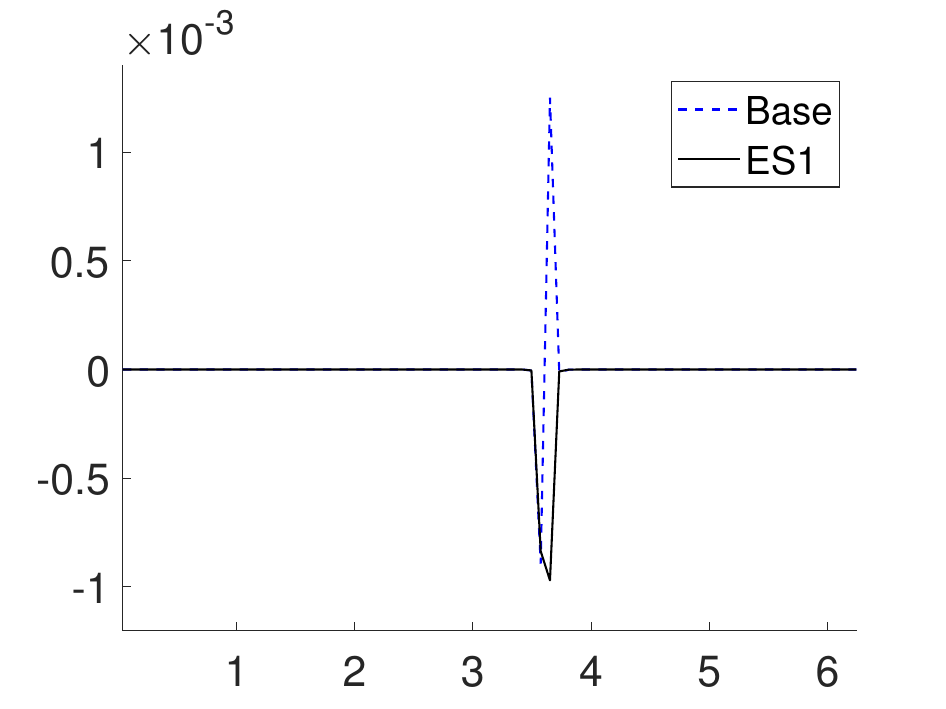} \label{figburgers1}}
			\subfigure[Evolution of total entropy with time.]{
				\includegraphics[width=0.45\linewidth]{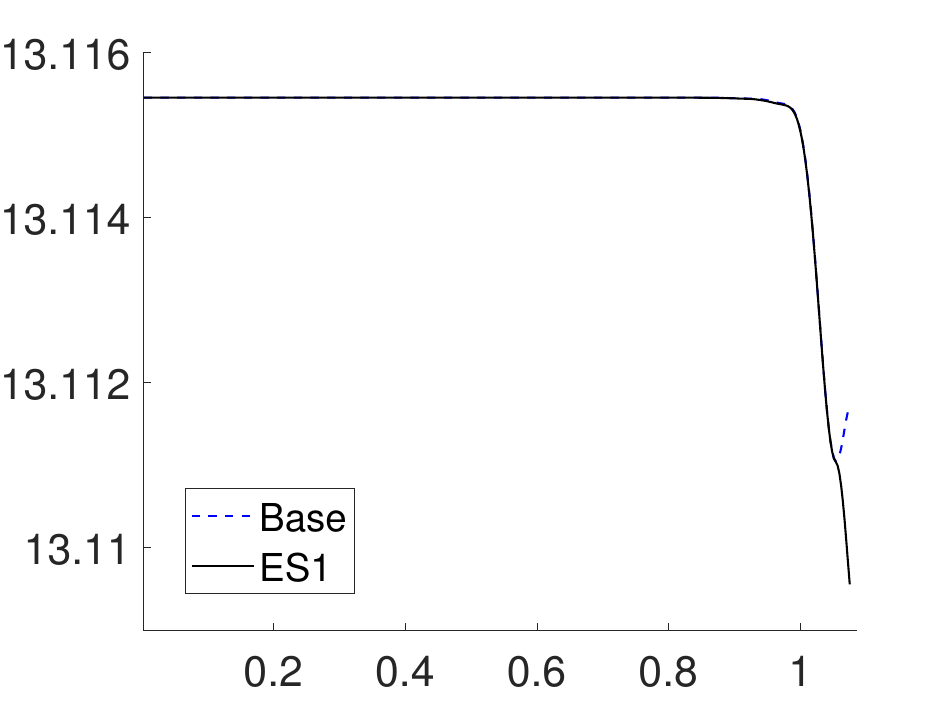} \label{figburgers2}}
			\caption{Example \ref{ex:burgers}: One-dimensional Burgers' equation. The cell entropy inequality violation and evolution of total entropy with time. A $\mathbb P^2$ approximation is used with $N_x=80$.}
			\label{figburgers}
		\end{figure}

		\begin{Ex}
			[Buckley--Leverett equation]
			\label{ex:BL}
		The Buckley--Leverett equation is a widely studied nonconvex test problem in the literature, with the flux function given by
		$$
		f(u) =  \frac{4u^2}{4u^2+\left( 1-u \right) ^2}.
		$$
		This example is computationally challenging due to the non-convexity of $f$. 
		Without careful design, numerical schemes may fail to converge to the unique entropy solution or converge too slowly \cite{kurganov2007adaptive}. The computational domain is $\Omega=[-0.5,0.5]$ with Dirichlet boundaries.       
		We consider the classical Riemann initial data
		$$
		u\left( x,0 \right) =\begin{cases}
			u_L,\quad x\le 0,\\
			u_R,\quad x>0.\\
		\end{cases}
		$$
		Here, we test two kinds of initial conditions:
		\begin{equation*}
			\begin{aligned}
				\text{IC1:}\quad & u_L=-3, \, u_R=3, \\
				\text{IC2:}\quad & u_L=2, \, u_R=-2,
			\end{aligned}
		\end{equation*}
		and three types of entropy functions:
		$$\mathcal{U}^{(1)}=\frac{u^2}{2},\quad \mathcal{U}^{(2)}=\int^u{\arctan 20u\,\mathrm{d}u},\quad \mathcal{U}^{(3)}=\int^u{\arctan 20\left( u-1 \right) \mathrm{d}u}. 
		$$
		These initial conditions and entropy pairs have also been studied in \cite{chen2017entropy, liu2024entropy, liu2024non}. 
		
		For this problem, the BP limiter is employed. 
		We test these three different entropy pairs (denoted as U1, U2, U3) under the two initial conditions, using $N_x=80$ cells up to $T=1$. The reference solution, computed via a fifth-order WENO scheme on a mesh of $10000$ cells, is shown alongside the numerical results in Fig. \ref{figBL}. 
		Our findings indicate that satisfying a single entropy condition is insufficient to guarantee physically correct solutions for both initial conditions. Specifically, only the U2 scheme yields the correct solution for IC1, while only U3 does so for IC2, consistent with prior studies \cite{chen2017entropy, liu2024entropy, liu2024non}.
		Furthermore, we implement a scheme that simultaneously enforces the entropy inequalities of both U2 and U3 (labeled ``U2+U3"). This combined approach successfully captures the correct physical solution for both IC1 and IC2. Crucially, the ``U2+U3" scheme achieves this without introducing excessive numerical dissipation compared to the single-entropy schemes. This example demonstrates the key advantage of our proposed method.
		\end{Ex}

		\begin{figure}[htb!]
			\centering
			\subfigure[IC1.]{
				\includegraphics[width=0.45\linewidth]{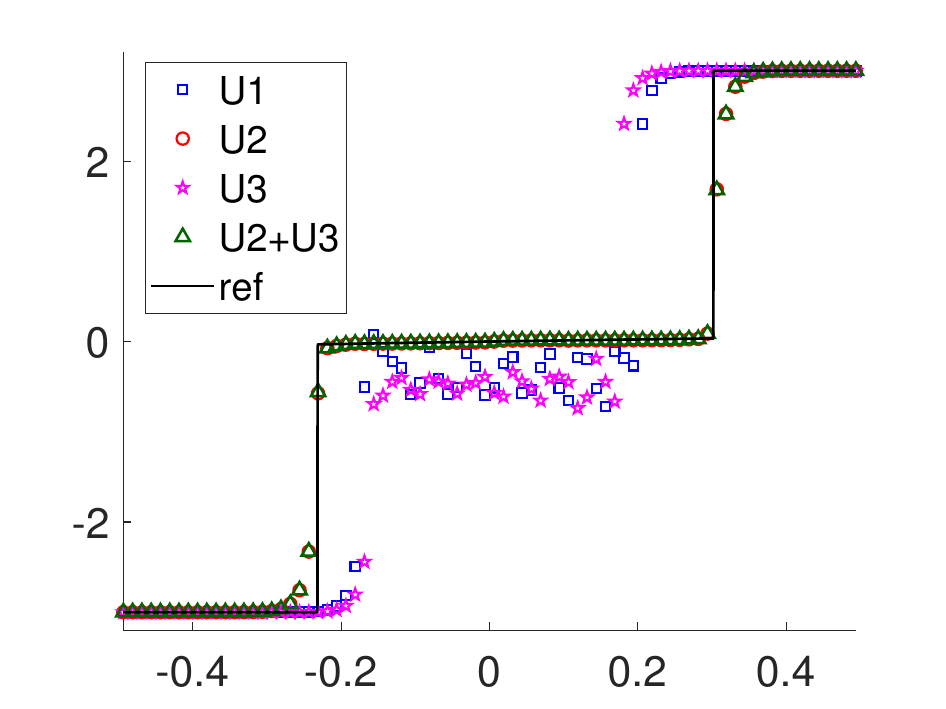}}
			\subfigure[IC2.]{
				\includegraphics[width=0.45\linewidth]{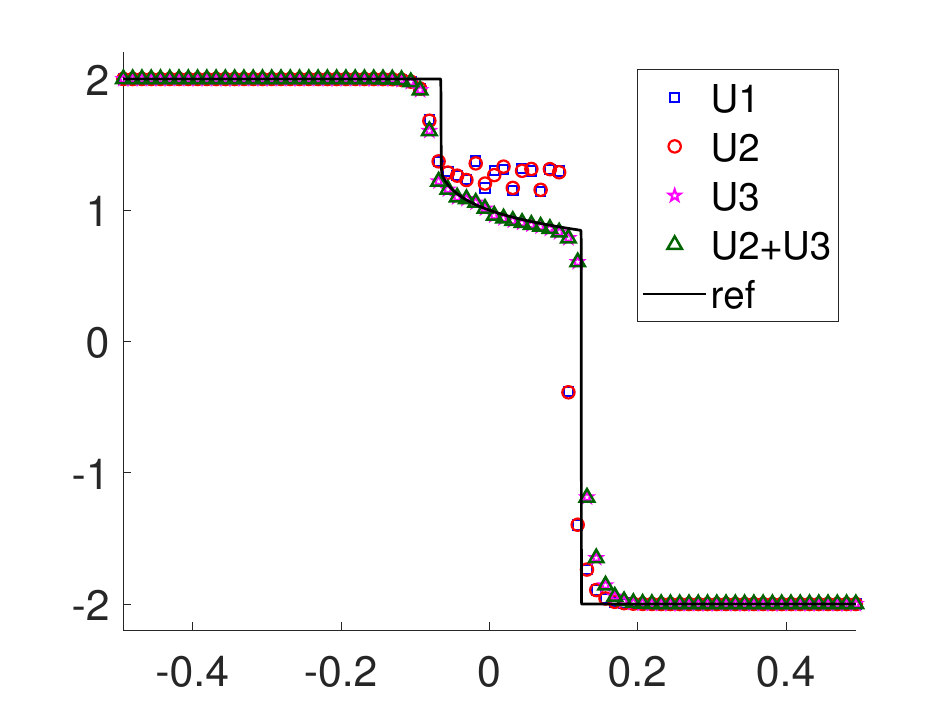}}
			\caption{Example \ref{ex:BL}: Buckley--Leverett equation. Numerical solutions with different entropy functions at $T = 1$ with $N_x=80$.}
			\label{figBL}
		\end{figure}

	\subsection{One-dimensional Euler systems}

	\begin{Ex}
		[Accuracy test]
		\label{ex:acc_sys}
	Here, we consider the smooth exact solution
	$$\rho(x,t)= 1+0.2\sin (x-t),\quad u(x,t)=1,\quad p(x,t)=1.$$
	The computational domain is taken as $\Omega=[0,2\pi]$ with periodic boundaries. In Table \ref{tab3}, we present the $L^2$ error of the density at $T=1$ for $k=1,2,3$. The optimal convergence rates are obtained. 
	\end{Ex}
	
	\begin{table}[htb!]
		\centering
		\caption{Example \ref{ex:acc_sys}: Accuracy test for the one-dimensional Euler equations. The $L^2$ errors and orders of density with different schemes at $T=1$. }
		\setlength{\tabcolsep}{3.2mm}{
			\begin{tabular}{|c|cc|cc|}
				\hline $N_x$ & Base & Order & ES & Order  \\ 
				\hline
				\multicolumn{5}{|c|}{$k=1$}\\
				
				\hline  
				20 &5.6346E-04 & -- &5.6346E-04 & --  \\
				\hline  
				40 &1.3978E-04 & 2.0111 &1.3978E-04 & 2.0111  \\
				\hline  
				80 &3.4878E-05 & 2.0028 &3.4878E-05 & 2.0028  \\
				\hline  
				160 &8.7152E-06 & 2.0007 &8.7152E-06 & 2.0007  \\
				\hline

				\multicolumn{5}{|c|}{$k=2$}\\		
				\hline  
				20 &4.8440E-05 & -- &4.8443E-05 & --  \\
				\hline  
				40 &6.7413E-06 & 2.8451 &6.7481E-06 & 2.8437  \\
				\hline  
				80 &8.6966E-07 & 2.9545 &8.7024E-07 & 2.9550  \\
				\hline  
				160 &1.0961E-07 & 2.9881 &1.0964E-07 & 2.9887  \\
				\hline
				
				\multicolumn{5}{|c|}{$k=3$}\\
				
				\hline  
				20 &2.9257E-07 & -- &2.9356E-07 & --  \\
				\hline  
				40 &1.8381E-08 & 3.9925 &1.8381E-08 & 3.9974  \\
				\hline  
				80 &1.1239E-09 & 4.0316 &1.1548E-09 & 3.9925  \\
				\hline  
				160 &7.0488E-11 & 3.9950 &7.0786E-11 & 4.0281  \\
				\hline
		\end{tabular}} 
		\label{tab3}
	\end{table}

	\begin{Ex}
		[Sod shock tube]
		\label{ex:Sod}
	We consider the Sod shock tube problem originally presented by Sod \cite{sod1978survey}, which has become a standard benchmark test for compressible flow solvers. This test example involves a Riemann problem with an initial discontinuity located at $x=0$ in the computational domain $\Omega= [-1,1]$. The initial conditions are given by 
	$$
	\left( \rho ,u,p \right) =\begin{cases}
		\left( 1,0,1 \right) , & x<0,\\
		\left( 0.125,0,0.1 \right) , & x\ge 0.\\
	\end{cases}
	$$
	The solution consists of a left-propagating rarefaction wave, a contact discontinuity, and a right-propagating shock wave. This test is particularly useful for assessing the scheme's ability to accurately capture different wave structures and maintain sharp discontinuity resolution. 
	
	In Fig. \ref{figSod}, we present the results of density at $T=0.4$ on a mesh with $N_x=200$ cells. Similar to the phenomenon in \cite{chen2017entropy}, the structure of discontinuities is captured very well via the entropy stability. Moreover, we show the cell entropy inequality violation at the final time step. Fig. \ref{figSodUineq} presents the cell entropy inequality violation for the base scheme and the ES scheme. Fig. \ref{figSodtheta} shows the distribution of $\theta_i$. It can be seen that our scheme strictly satisfies the cell entropy inequality, while the Base scheme violates it in some regions, especially near the shock.  The ES limiter modifies the solution both in the rarefaction wave region and near discontinuities, with stronger limiting near the discontinuities. In the smooth rarefaction wave region, $1-\theta_i$ remains below $10^{-5}$, thereby introducing only very small additional dissipation there. Nevertheless, the base scheme exhibits a slight entropy increase even in the rarefaction wave region, and this mild violation is effectively corrected by the ES limiter.    
	\end{Ex}

	\begin{figure}[htb!]
		\centering
		\subfigure[Cell average.]{
			\includegraphics[width=0.45\linewidth]{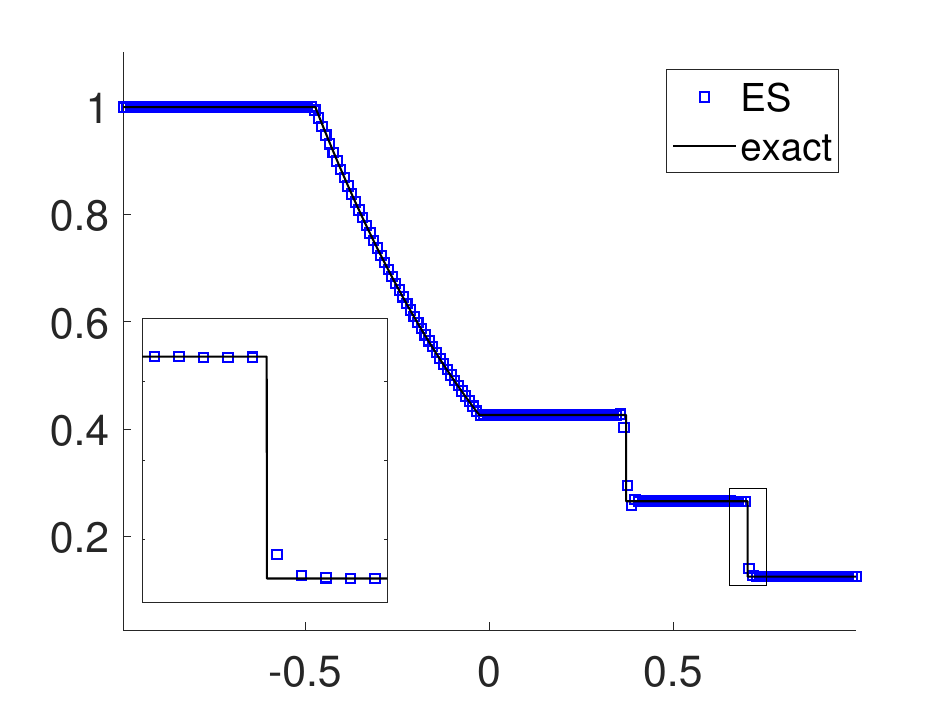}}
		\subfigure[Polynomial.]{
			\includegraphics[width=0.45\linewidth]{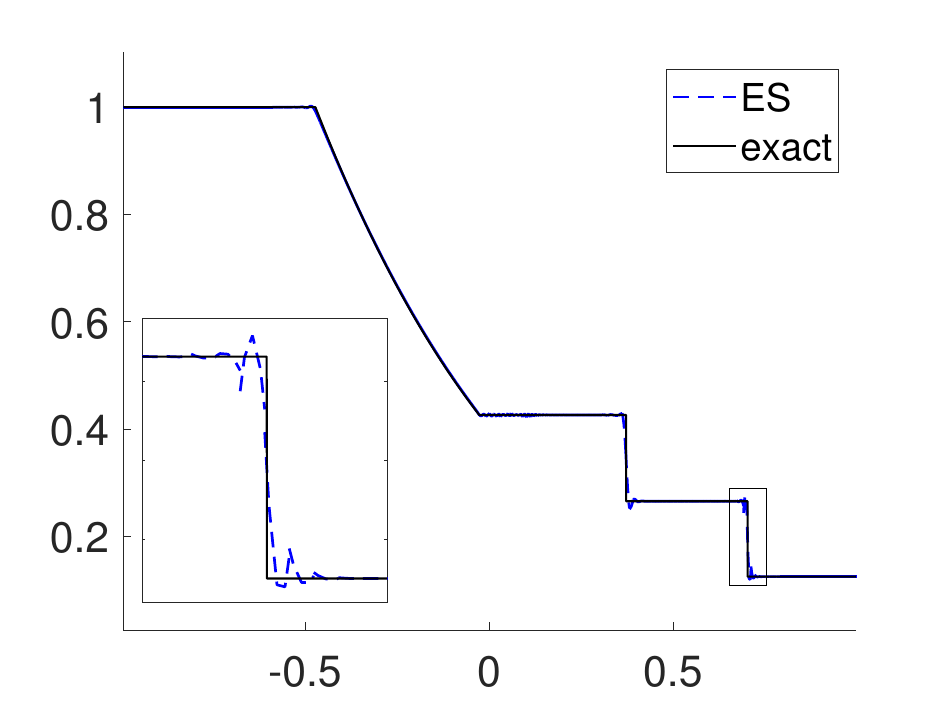}}
		\caption{Example \ref{ex:Sod}: Sod shock tube. The results of density at $T=0.4$ with $N_x=200$.}
		\label{figSod}
	\end{figure}

    \begin{figure}[htb!]
		\centering
		\subfigure[Cell entropy inequality violation.]{
			\includegraphics[width=0.45\linewidth]{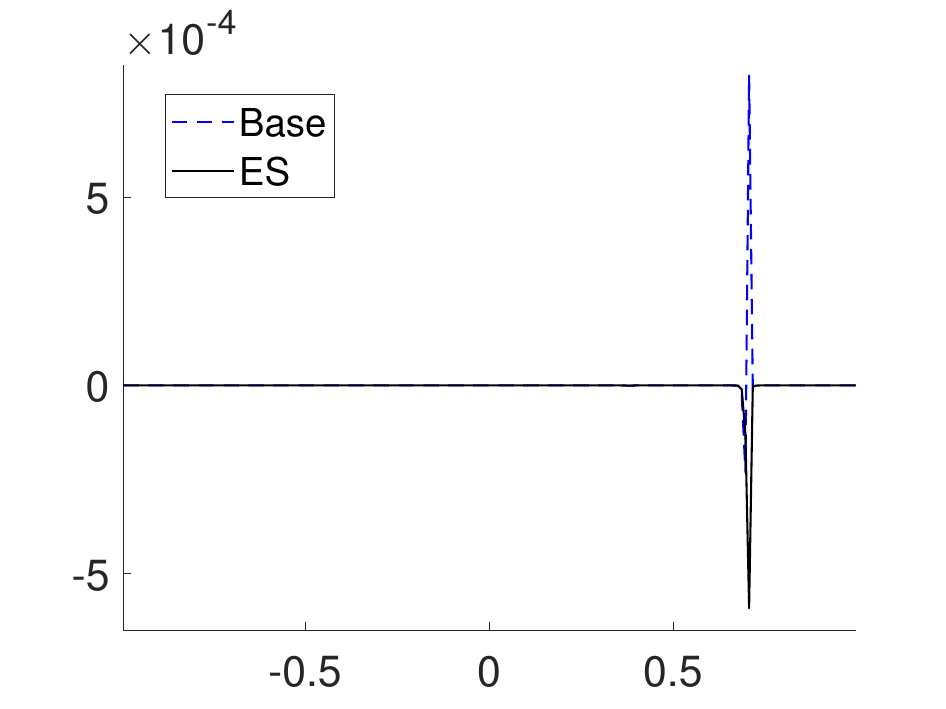}}
		\subfigure[Logarithmic plot of the positive part.]{
			\includegraphics[width=0.45\linewidth]{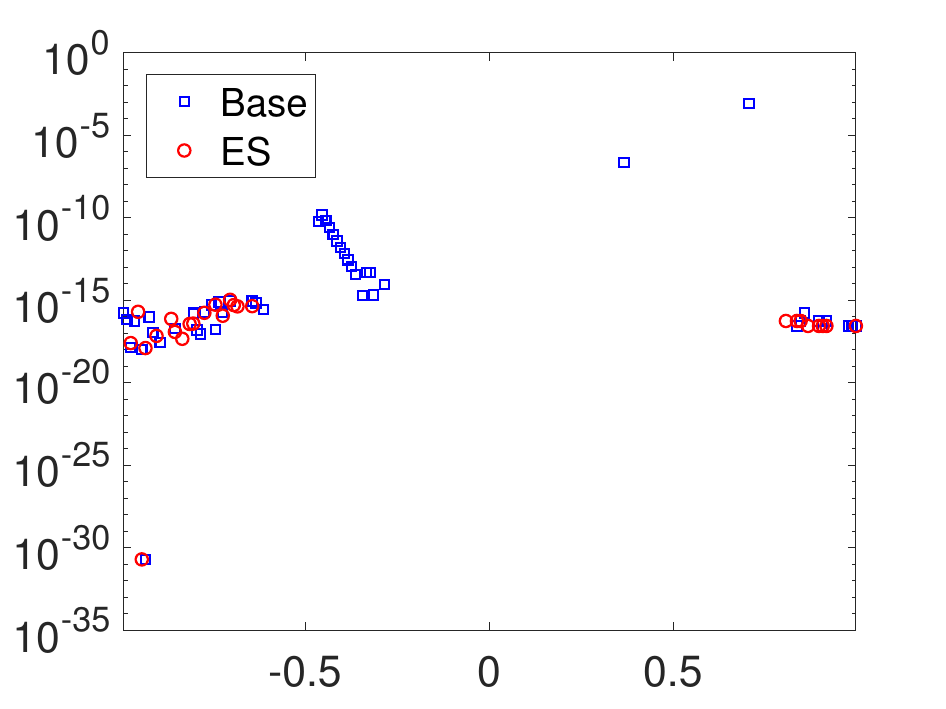}}
		\caption{Example \ref{ex:Sod}: Sod shock tube. The cell entropy inequality violation.}
		\label{figSodUineq}
	\end{figure}

    \begin{figure}[htb!]
		\centering
		\subfigure[$\theta_i$]{
			\includegraphics[width=0.45\linewidth]{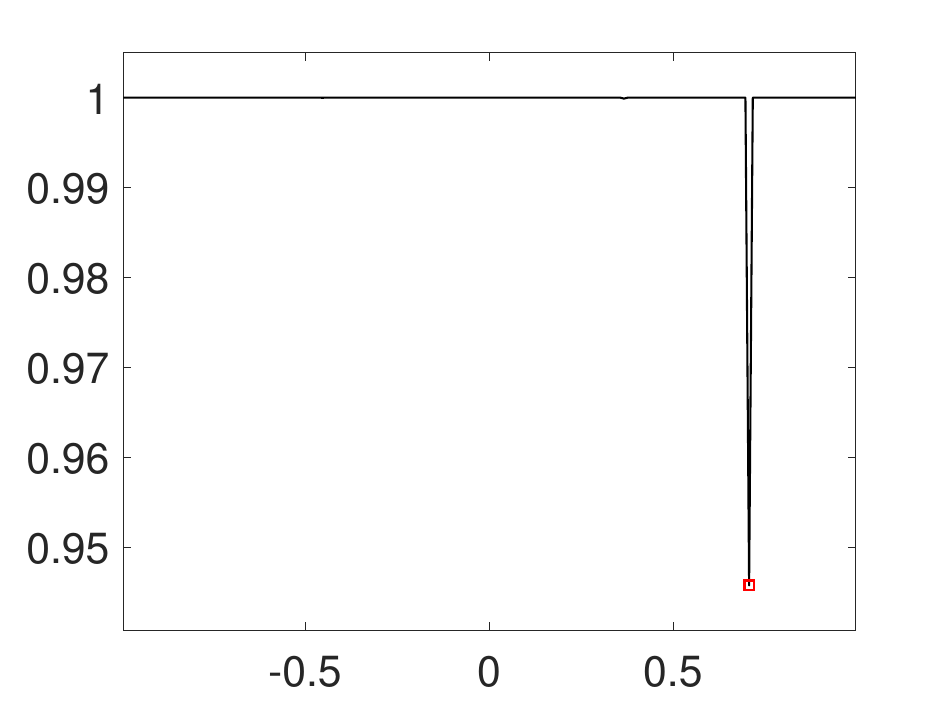}}
		\subfigure[Logarithmic plot of $1-\theta_i$.]{
			\includegraphics[width=0.45\linewidth]{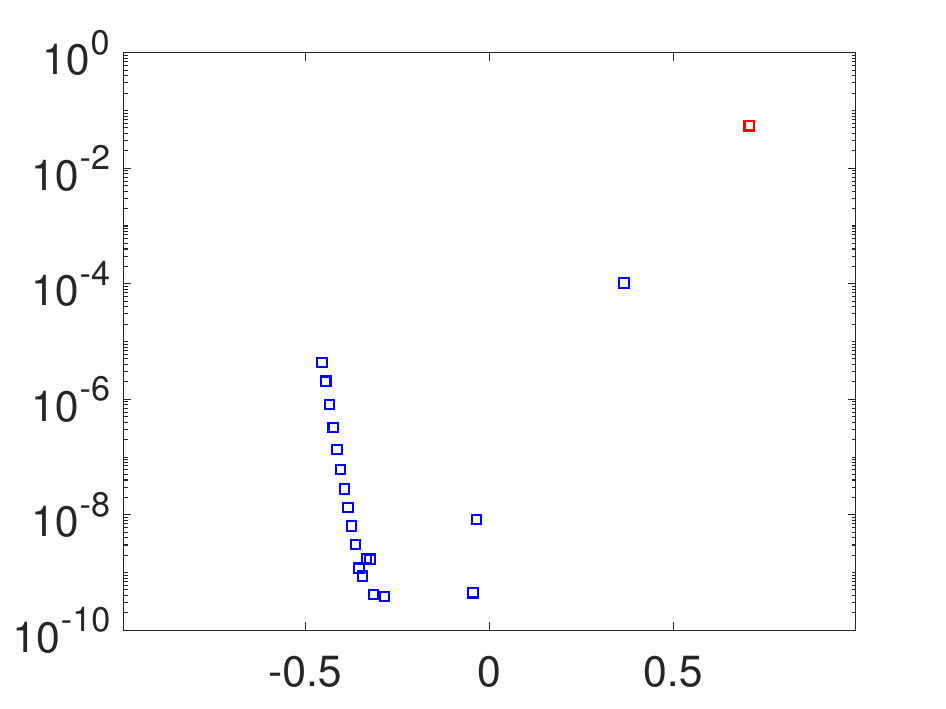}}
		\caption{Example \ref{ex:Sod}: Sod shock tube. The distribution of $\theta_i$. The red squares highlight the minimum values of $\theta_i$, which correspond to the locations of the main discontinuities captured in the flow field.}
		\label{figSodtheta}
	\end{figure}

	\begin{Ex}
		[Shu--Osher problem]
		\label{ex:Shuosher}
	We consider the Shu--Osher problem originally presented by Shu and Osher \cite{shu1988efficient}, which serves as a stringent test for assessing the numerical dissipation and resolution properties of high-order schemes. This test example involves a Mach 3 shock wave interacting with sine-wave density perturbations in the computational domain $\Omega=[-5,5]$.  The initial conditions are given by 
	$$
	\left( \rho ,u,p \right) =\begin{cases}
		\left( 3.857143,2.629369,10.3333 \right) , \quad x<-4,\\
		\left( 1+0.2\sin \left( 5x \right) ,0,1 \right) , \qquad x\ge -4.\\
	\end{cases}
	$$
	As the shock propagates through the density fluctuations, it generates complex wave structures with multiple scales. This test is particularly useful for evaluating the scheme's capability to resolve fine-scale features while maintaining low numerical dissipation in smooth regions. 
	
	In Fig. \ref{figShuosher}, we present the results of density at $T=1.8$ on a mesh with $N_x=200$ cells. The reference solution is computed by the fifth-order finite difference WENO scheme on a mesh with $2000$ cells. It can be seen that our scheme gives a very satisfactory result. Likewise, we plot the cell entropy inequality violation at the final time step.  In Fig. \ref{figShuosherUineq}, we plot the cell entropy inequality violation of the Base scheme and the ES scheme. In Fig. \ref{figShuoshertheta}, we plot the distribution of $\theta_i$.     It appears that the Base scheme also maintains the cell entropy inequality in the discontinuous regions. However, the maximum cell entropy inequality violation of the Base scheme is $5.2746\times 10^{-7}$, while the maximum cell entropy inequality violation of our scheme is 0. 
	\end{Ex}

	\begin{figure}[htb!]
		\centering
		\subfigure[Cell average.]{
			\includegraphics[width=0.45\linewidth]{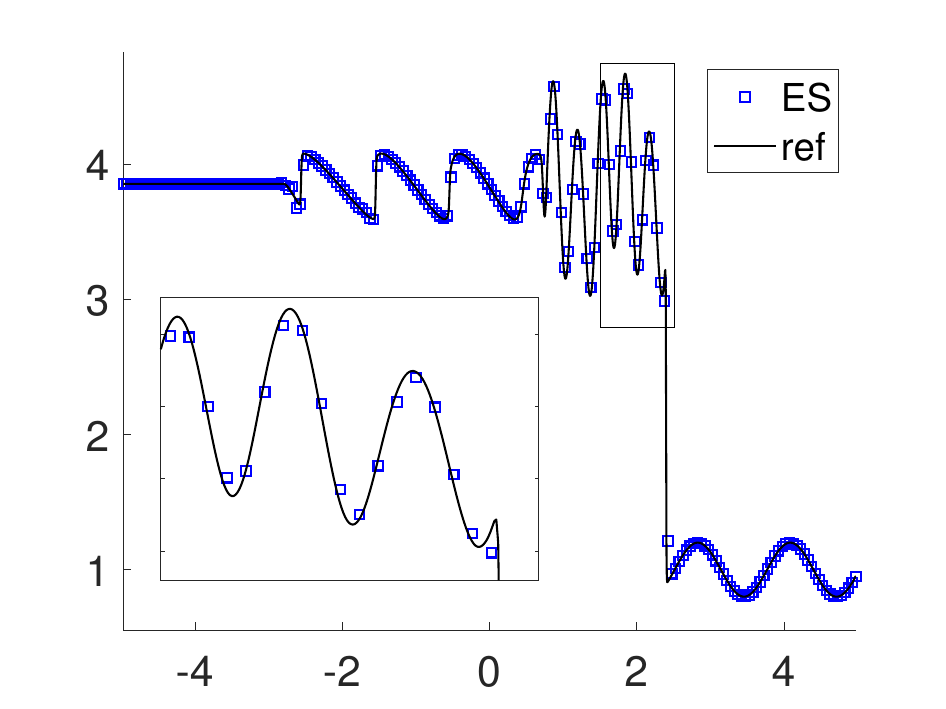}}
		\subfigure[Polynomial.]{
			\includegraphics[width=0.45\linewidth]{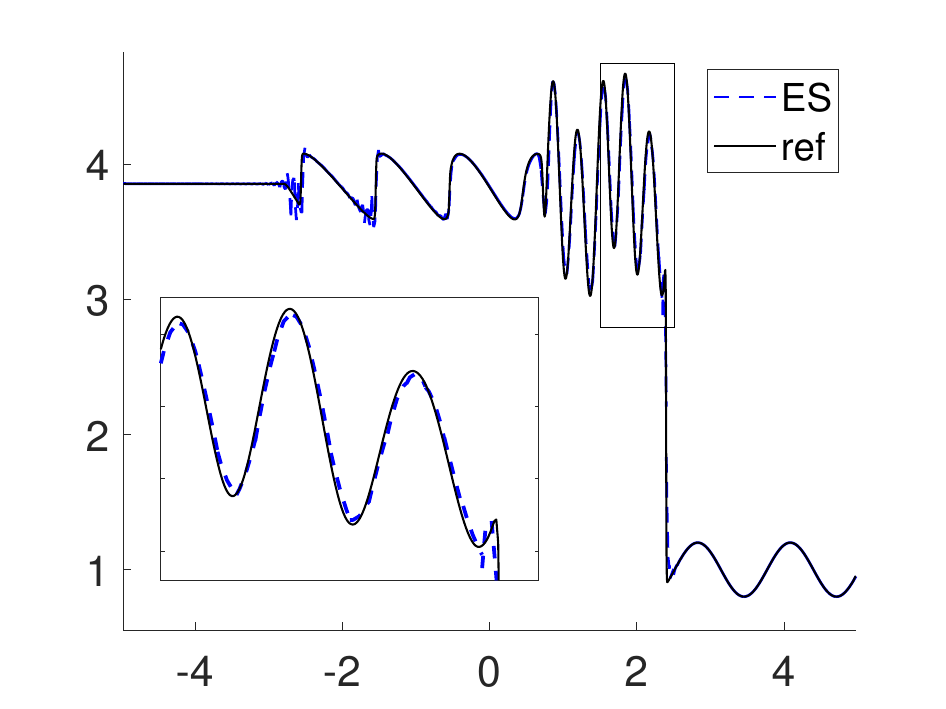}}
		\caption{Example \ref{ex:Shuosher}: Shu--Osher problem. The results of density at $T=1.8$ with $N_x=200$.}
		\label{figShuosher}
	\end{figure}

    \begin{figure}[htb!]
		\centering
		\subfigure[Cell entropy inequality violation.]{
			\includegraphics[width=0.45\linewidth]{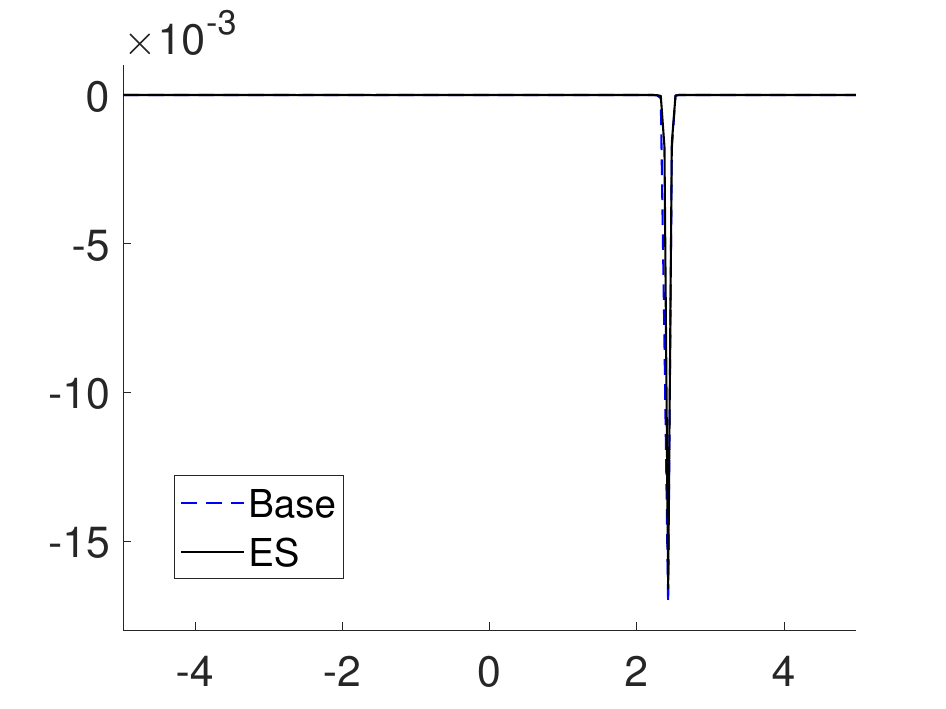}}
		\subfigure[Logarithmic plot of the positive part.]{
			\includegraphics[width=0.45\linewidth]{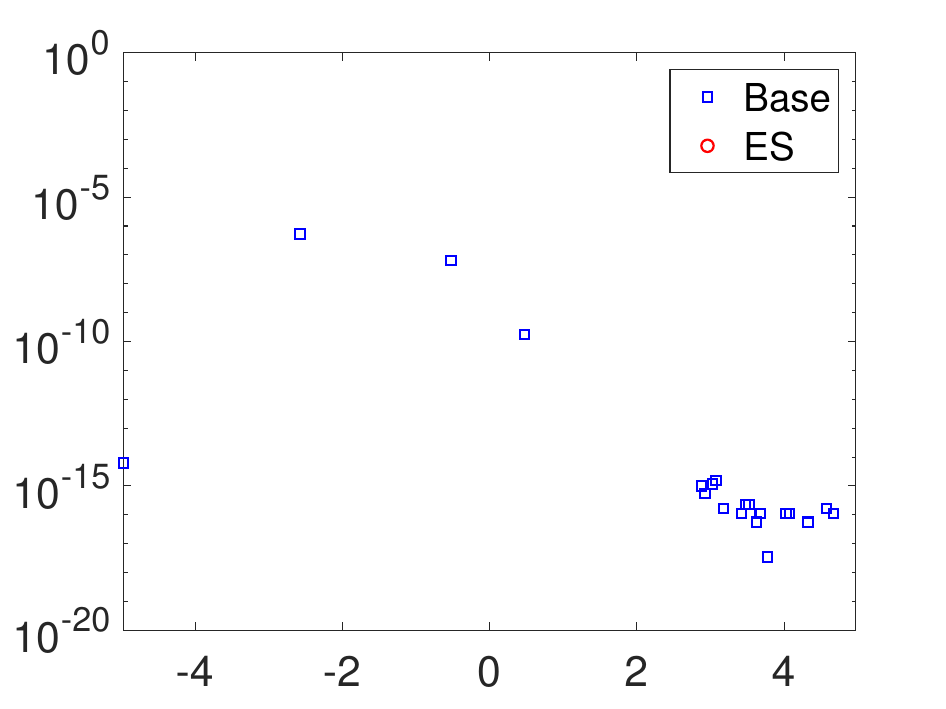}}
		\caption{Example \ref{ex:Shuosher}: Shu--Osher problem. The cell entropy inequality violation.}
		\label{figShuosherUineq}
	\end{figure}

    \begin{figure}[htb!]
		\centering
		\subfigure[$\theta_i$]{
			\includegraphics[width=0.45\linewidth]{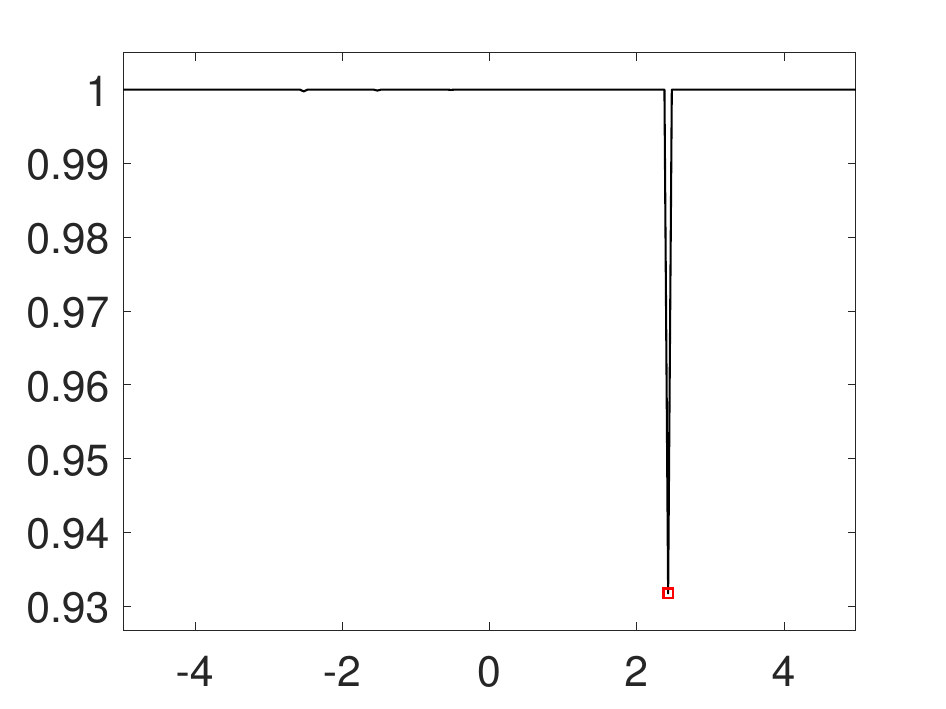}}
		\subfigure[Logarithmic plot of $1-\theta_i$.]{
			\includegraphics[width=0.45\linewidth]{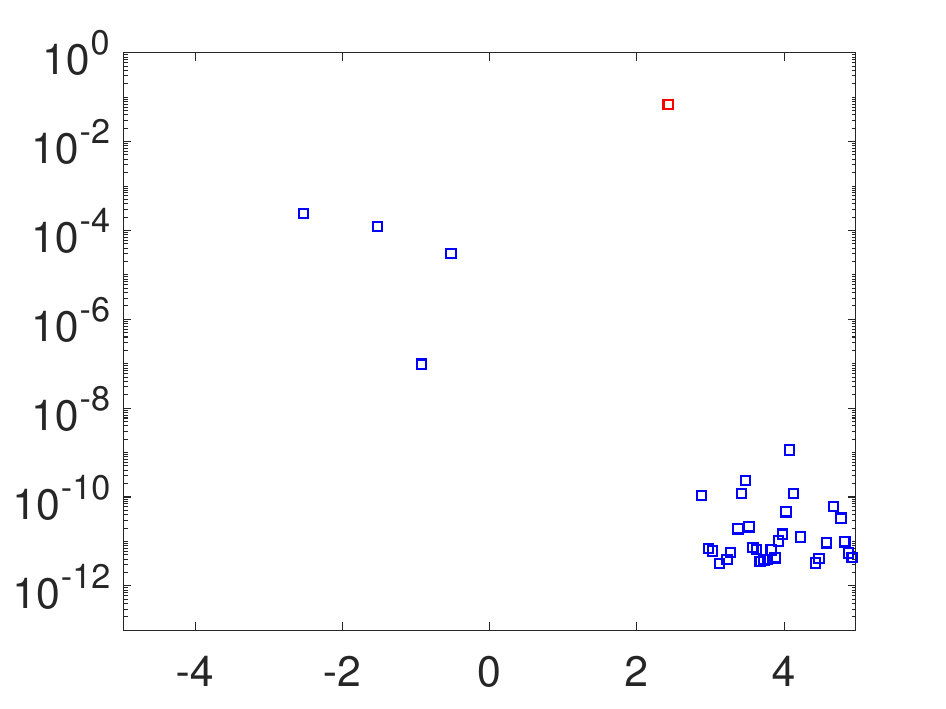}}
		\caption{Example \ref{ex:Shuosher}: Shu--Osher problem. The distribution of $\theta_i$. The red squares highlight the minimum values of $\theta_i$, which correspond to the locations of the main discontinuities captured in the flow field.}
		\label{figShuoshertheta}
	\end{figure}

	\begin{Ex}
		[Two blast waves]
		\label{ex:blast}
	We consider the blast wave interaction problem originally presented by Woodward and Colella \cite{woodward1984numerical}. This test example involves two strong blast waves propagating towards each other in a confined domain $\Omega=[0,1]$ with reflective boundary conditions. The initial conditions are given by $$
	\rho =1,\quad u=0,\quad p=\begin{cases}
		1000, & x<0.1,\\
		0.01,& 0.1\le x\le 0.9,\\
		100,& x>0.9.\\
	\end{cases}
	$$
	The solution features complex wave interactions including multiple shock reflections and contact discontinuities. This test is particularly demanding as it requires accurate resolution of strong shocks and their intricate interaction patterns while maintaining numerical stability. 
	
	For this example, the PP limiter should be employed; otherwise, negative pressure will be introduced. In Fig. \ref{figblast}, we give the results of density at $T=0.038$ on a mesh with $N_x=400$ cells, where the reference solution is obtained with the fifth-order WENO scheme with 2000 cells. Compared to the results in \cite{peng2024oedg, liu2024entropy, liu2024non, wei2024jump}, it can be seen that our scheme gives a really sharp solution, although some minor shocks are present. Moreover, we want to remark here that this problem was not simulated in \cite{chen2017entropy}. In our reproduction attempts, we found that using only the PP limiter is insufficient for stable computation with their scheme, unless a shock limiter is also employed.  In Fig. \ref{figblastUineq}, we plot the cell entropy inequality violation of the ES scheme. We also note here that the cell entropy inequality violation of the ES scheme is negative everywhere, with a maximum value of $-8.8818\times 10^{-16}$. In Fig. \ref{figblasttheta}, we plot the distribution of $\theta_i$.  
	\end{Ex}

	\begin{figure}[htb!]
		\centering
		\subfigure[Cell average.]{
			\includegraphics[width=0.45\linewidth]{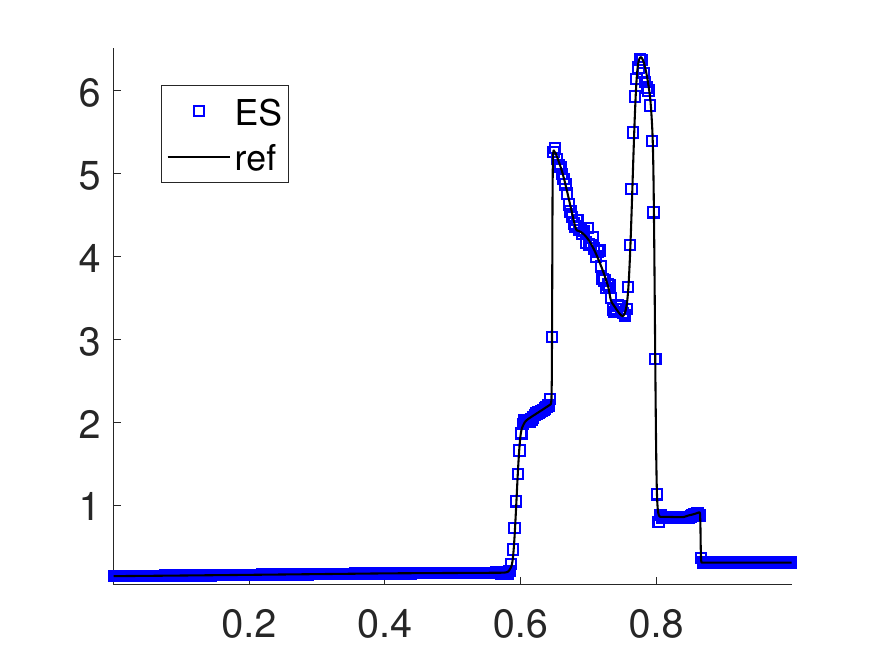}}
		\subfigure[Polynomial.]{
			\includegraphics[width=0.45\linewidth]{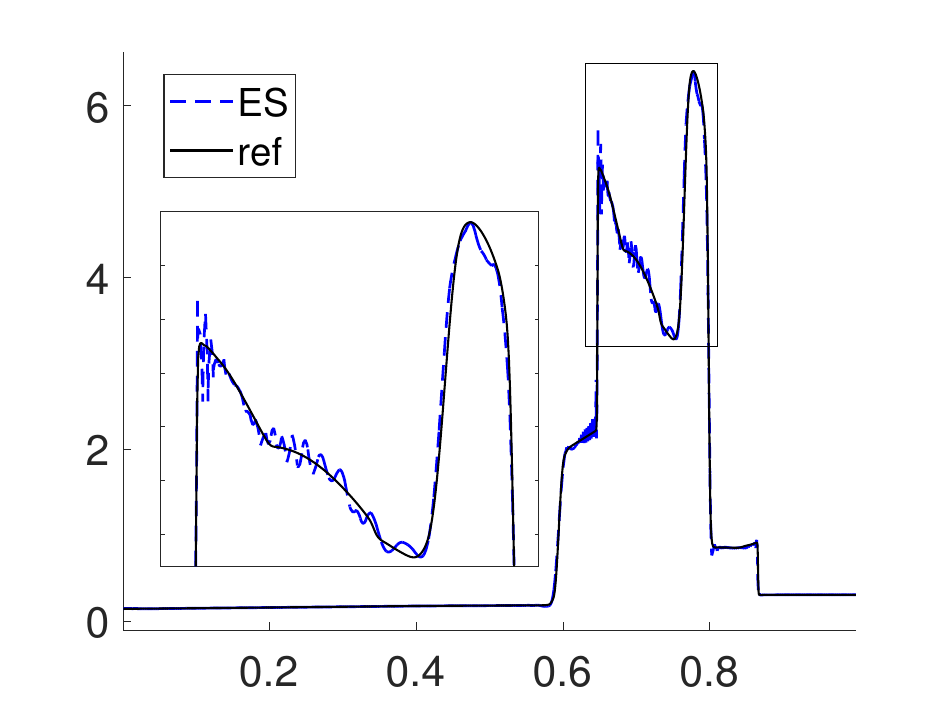}}
		\caption{Example \ref{ex:blast}: Two blast waves. The results of density at $T=0.038$ with $N_x=400$.}
		\label{figblast}
	\end{figure}

    \begin{figure}[htb!]
		\centering
		
			\includegraphics[width=0.45\linewidth]{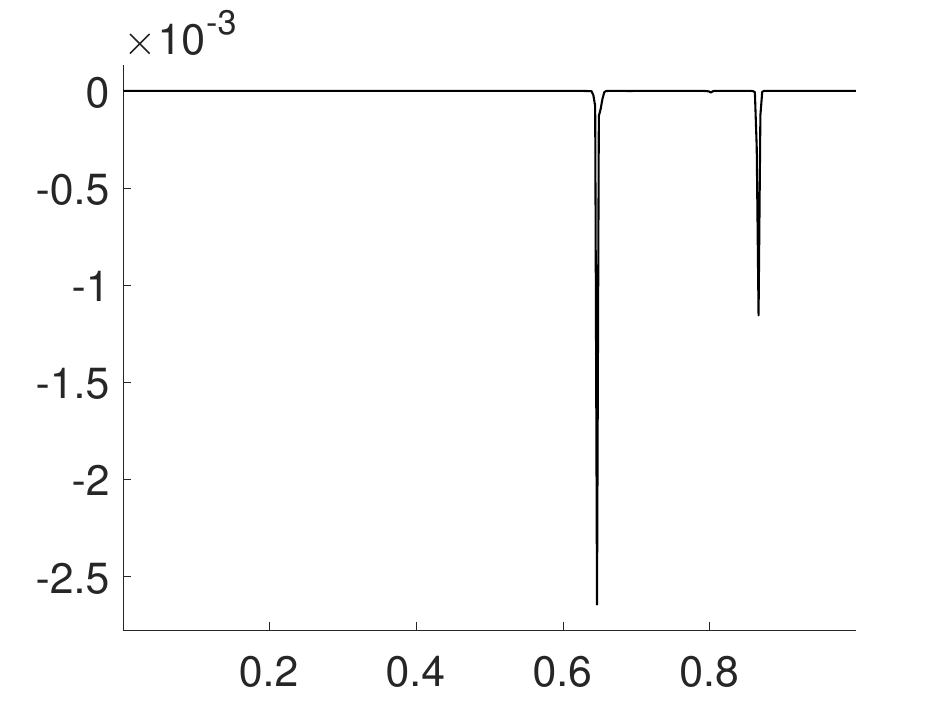}
		
		\caption{Example \ref{ex:blast}: Two blast waves. The cell entropy inequality violation.}
		\label{figblastUineq}
	\end{figure}

    \begin{figure}[htb!]
		\centering
		\subfigure[$\theta_i$.]{
			\includegraphics[width=0.45\linewidth]{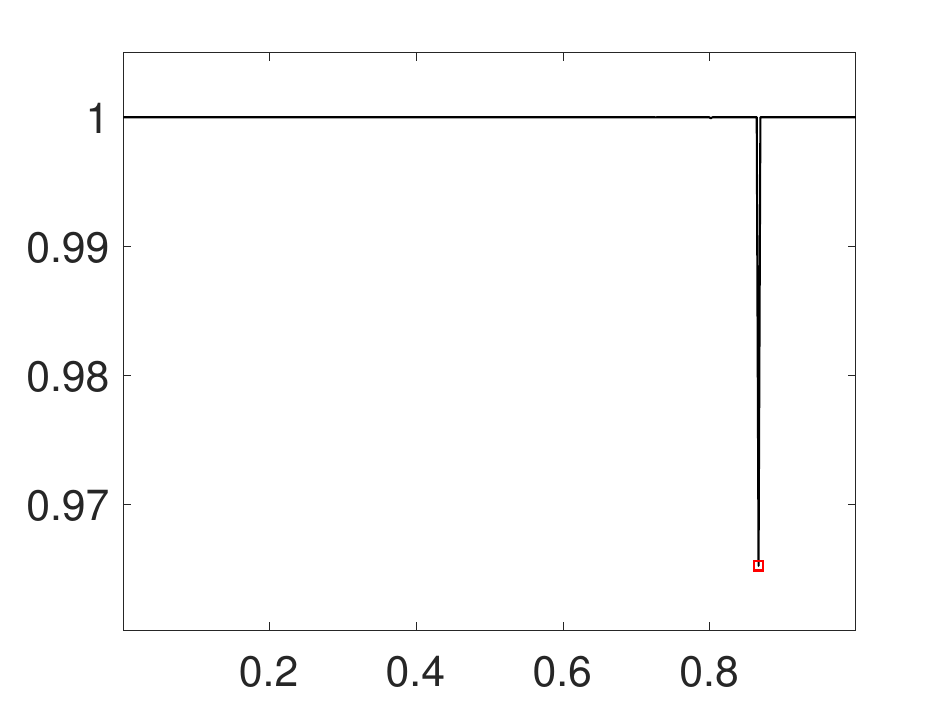}}
		\subfigure[Logarithmic plot of $1-\theta_i$.]{
			\includegraphics[width=0.45\linewidth]{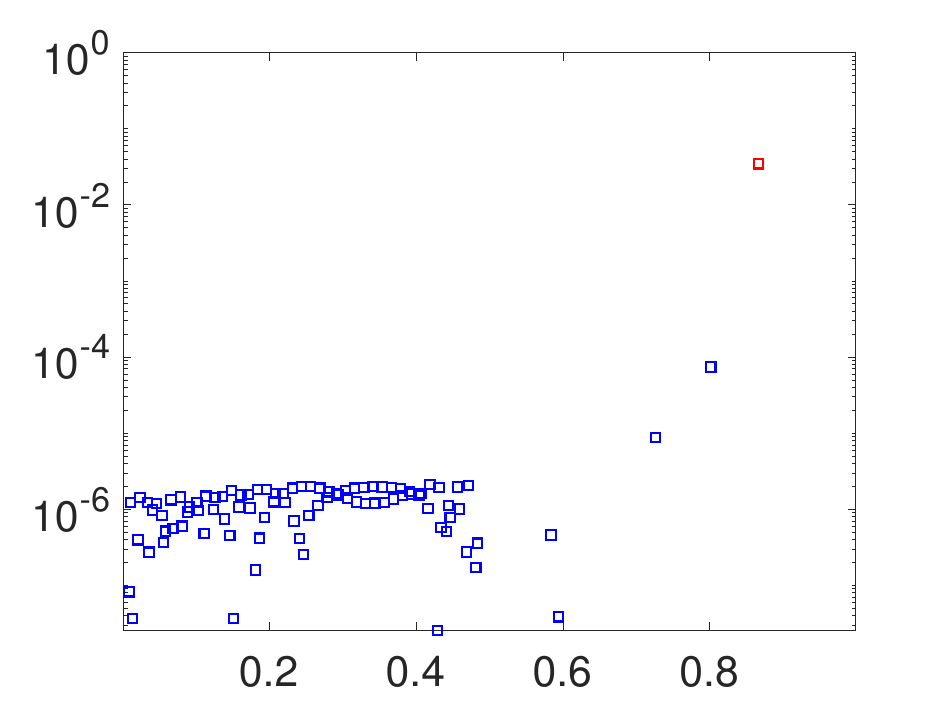}}
		\caption{Example \ref{ex:blast}: Two blast waves. The distribution of $\theta_i$. The red squares highlight the minimum values of $\theta_i$, which correspond to the locations of the main discontinuities captured in the flow field.}
		\label{figblasttheta}
	\end{figure}

	\begin{Ex}
		[Leblanc shock tube]
		\label{ex:Leblanc}
	We consider the Leblanc shock tube problem \cite{loubere2005subcell}, which represents an extremely challenging test case due to the large density and pressure ratios involved. This test example involves an extremely strong shock wave with an initial discontinuity located at $x=0$ in the computational domain $\Omega=[-10,10]$. The initial conditions are given by 
	$$
	\left( \rho ,u,p \right) =\begin{cases}
		\left( 2,0,10^9 \right) , & x<0,\\
		\left( 0.001,0,1 \right) , & x\ge 0.\\
	\end{cases}
	$$ 
	The large density ratio of $10^3$ and pressure ratio of $10^9$ make this problem particularly demanding for numerical schemes, as many methods struggle to maintain stability and accuracy under such severe conditions. This test serves as a stringent assessment of the scheme's robustness in handling extreme flow parameters while preserving positivity and avoiding spurious oscillations. 
	
	For this example, the PP limiter is employed. In Fig. \ref{figLeb1}, we present the result at $T=0.0001$ on meshes with $N_x=800$ and $N_x=6400$ cells, respectively,
	with corresponding zoomed-in views provided in Fig. \ref{figLeb2}. It can be seen that our scheme can compute this problem stably and accurately. Moreover, as the mesh is refined, the numerical solution indeed converges to the exact solution. It is noted that, as shown in \cite{zhang2010positivity}, employing only the PP limiter fails to complete the simulation due to the presence of strong shocks.
	This demonstrates that the ES property enhances the stability of the scheme. Furthermore, in our tests, the semi-discrete ES 
	scheme by Chen and Shu \cite{chen2017entropy} also fails to simulate this problem, even when the PP limiter is employed. 
	This shows the importance and advantage of achieving the ES property in the fully-discrete sense.
	\end{Ex}

	\begin{figure}[htbp!]
		\centering
		\subfigure[$\lg\rho$.]{
			\includegraphics[width=0.45\linewidth]{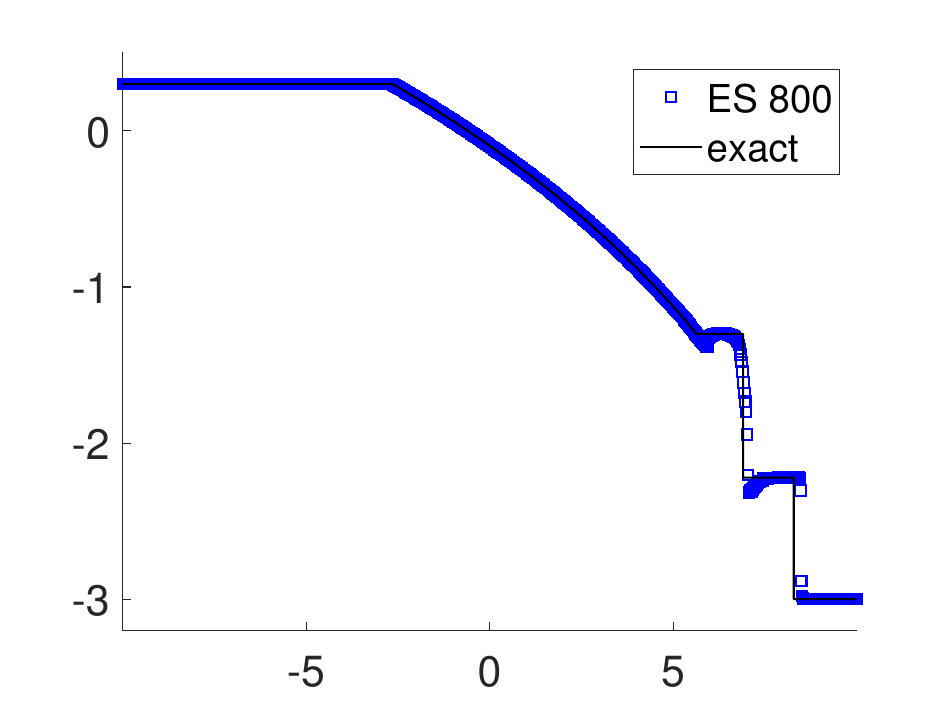}}
		\subfigure[$\lg\rho$.]{
			\includegraphics[width=0.45\linewidth]{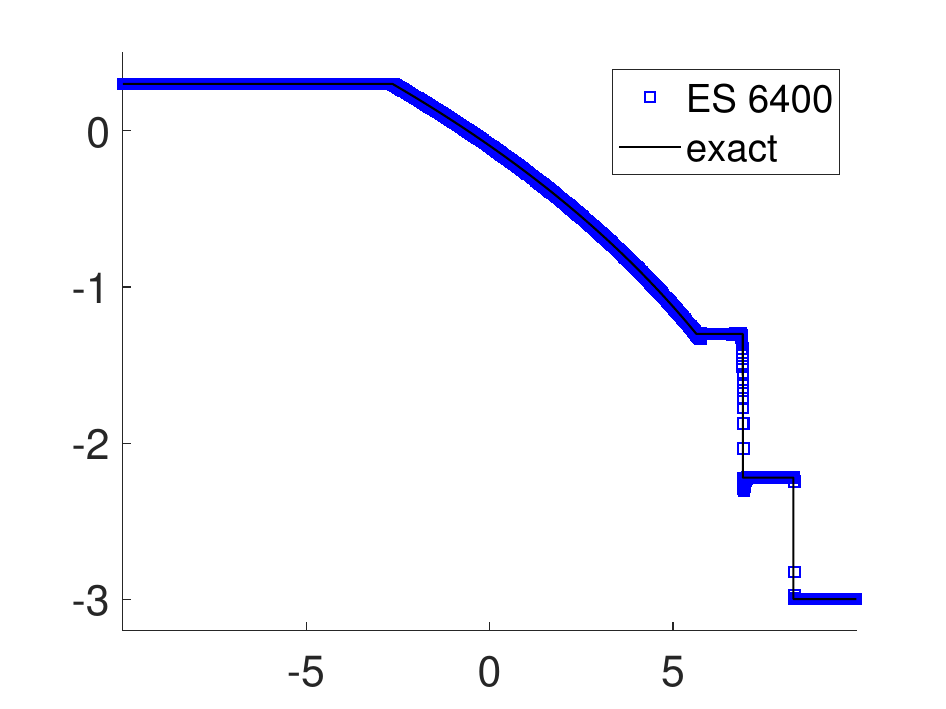}}
		\subfigure[Velocity.]{
			\includegraphics[width=0.45\linewidth]{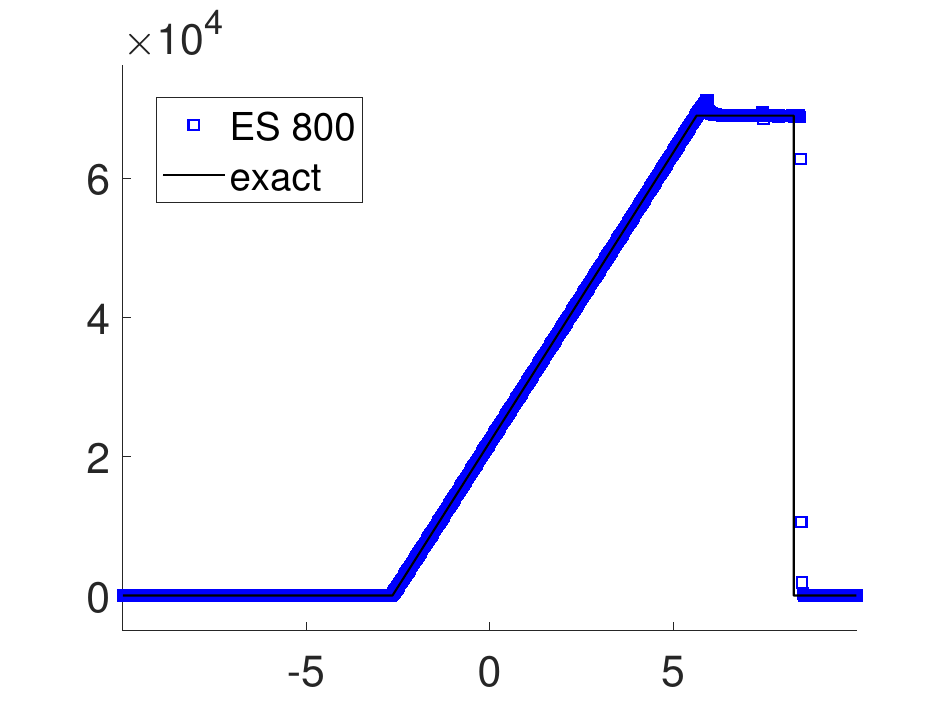}}
		\subfigure[Velocity.]{
			\includegraphics[width=0.45\linewidth]{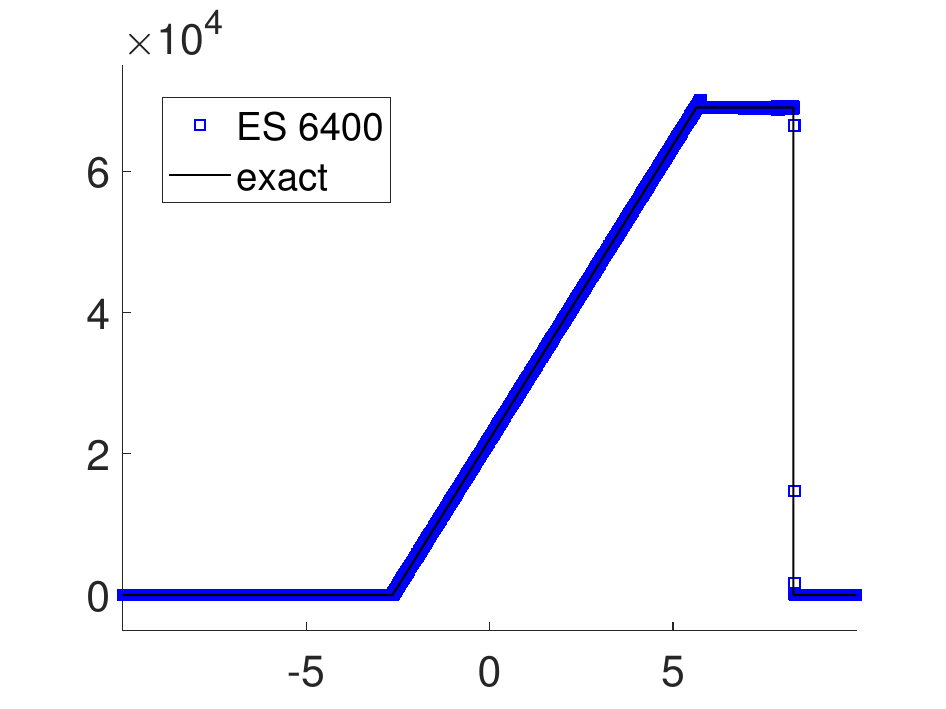}}
		\subfigure[$\lg p$.]{
			\includegraphics[width=0.45\linewidth]{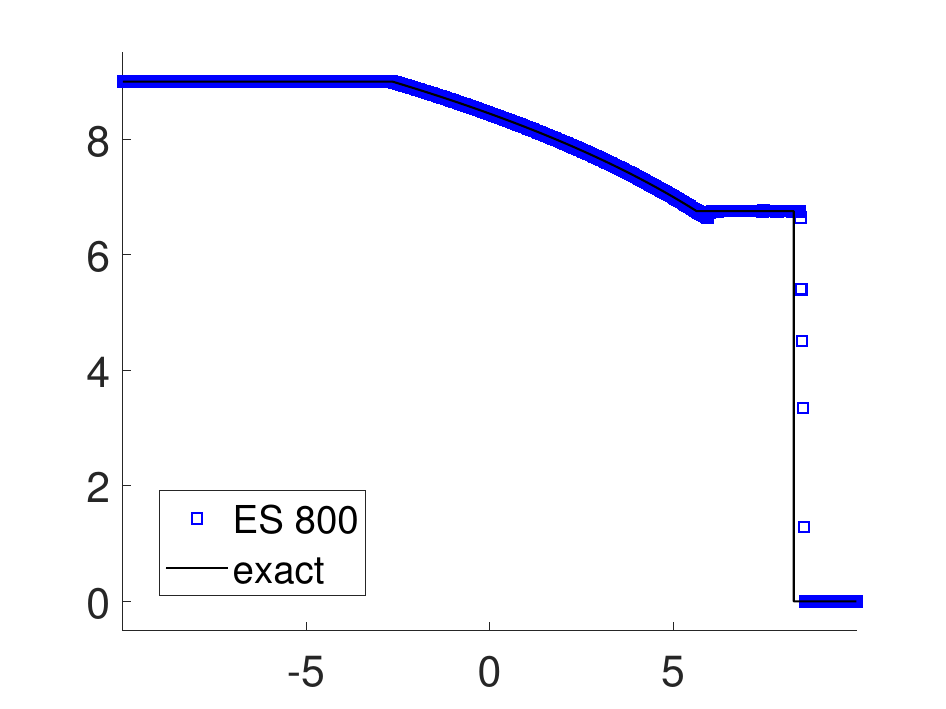}}
		\subfigure[$\lg p$.]{
			\includegraphics[width=0.45\linewidth]{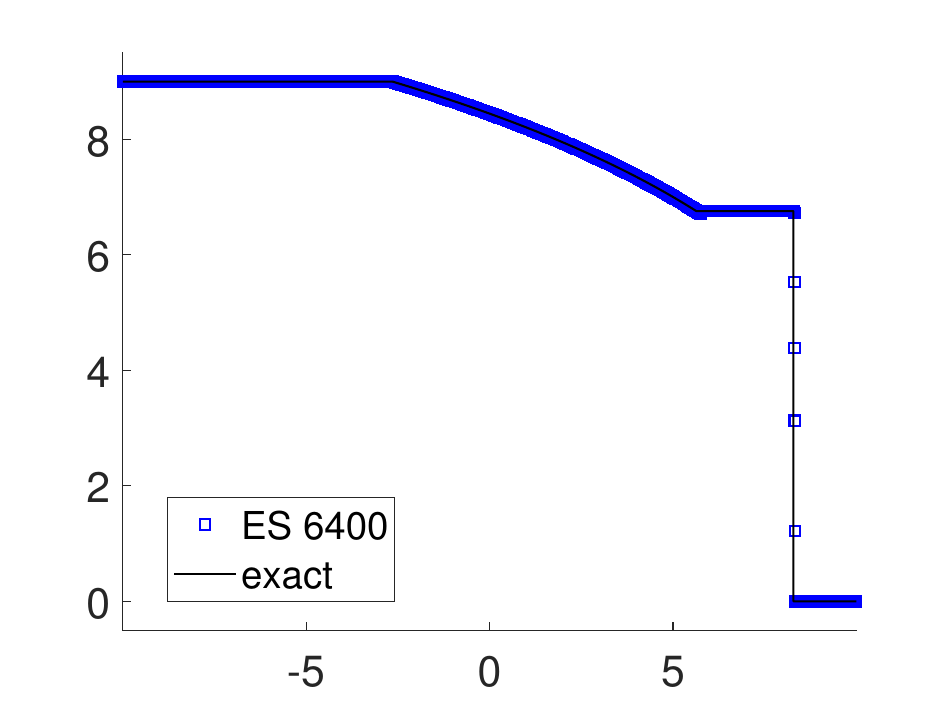}}
		\caption{Example \ref{ex:Leblanc}: Leblanc shock tube. The numerical solution at $T = 0.0001$ with $N_x= 800$ and $N_x=6400$.}
		\label{figLeb1}
	\end{figure}
	
	\begin{figure}[htbp!]
		\centering
		\subfigure[$\lg\rho$.]{
			\includegraphics[width=0.45\linewidth]{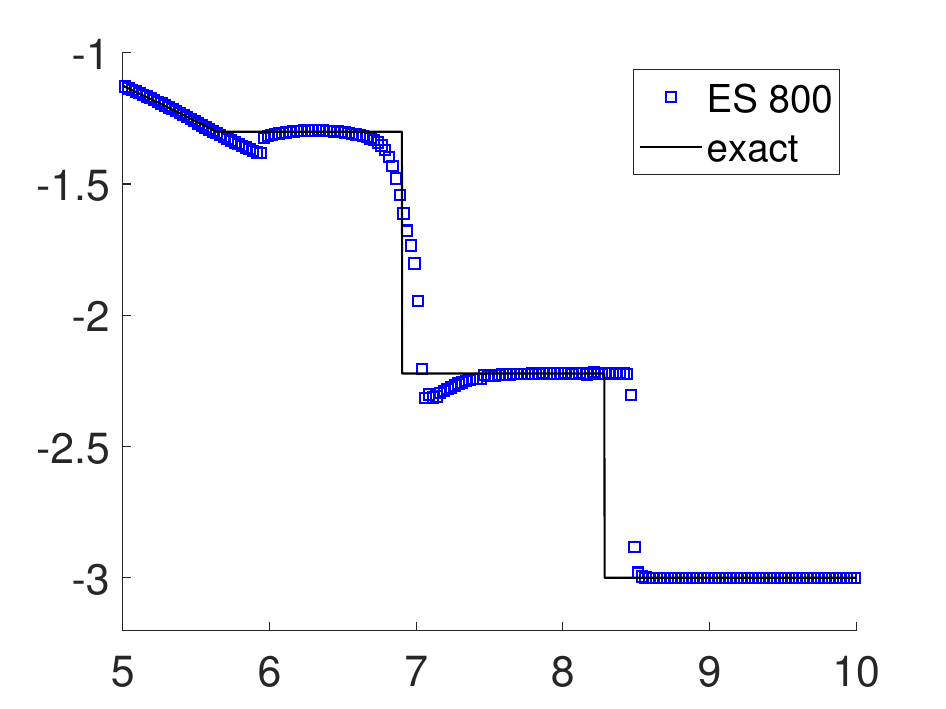}}
		\subfigure[$\lg\rho$.]{
			\includegraphics[width=0.45\linewidth]{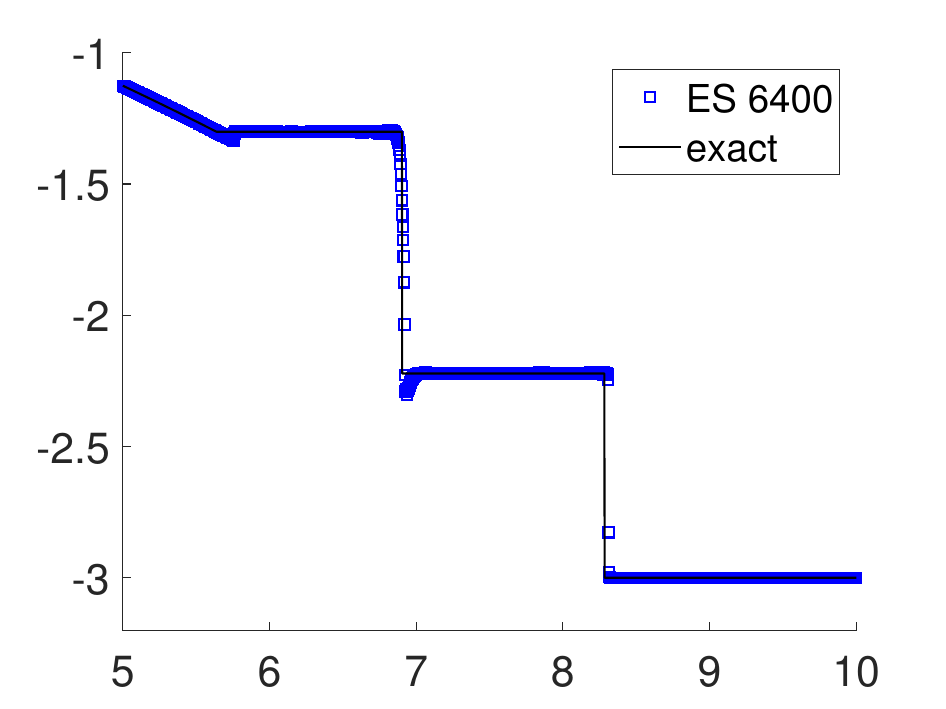}}
		\subfigure[Velocity.]{
			\includegraphics[width=0.45\linewidth]{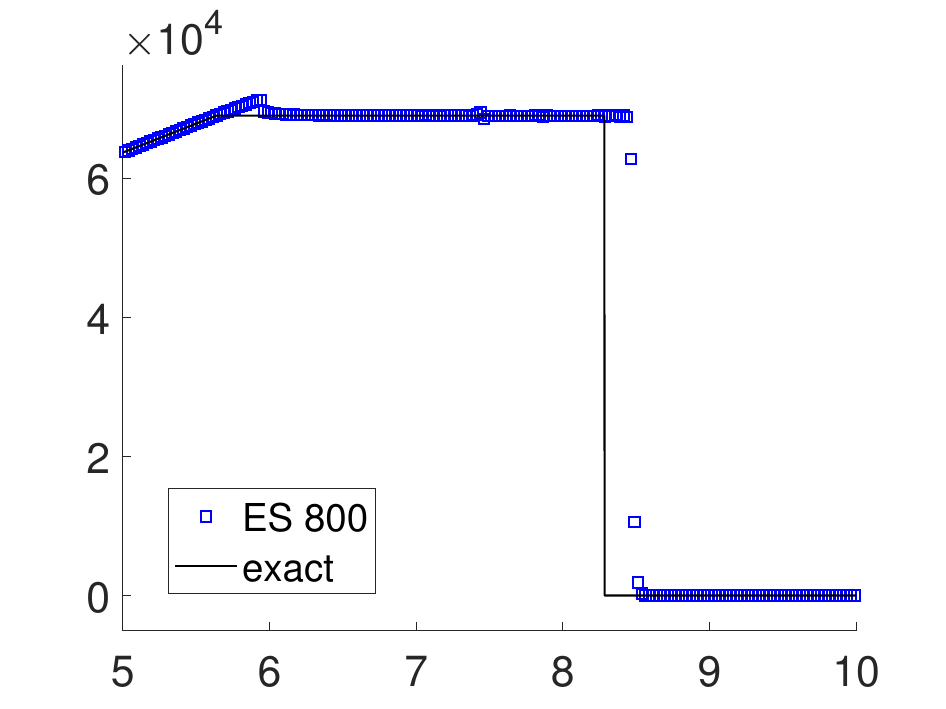}}
		\subfigure[Velocity.]{
			\includegraphics[width=0.45\linewidth]{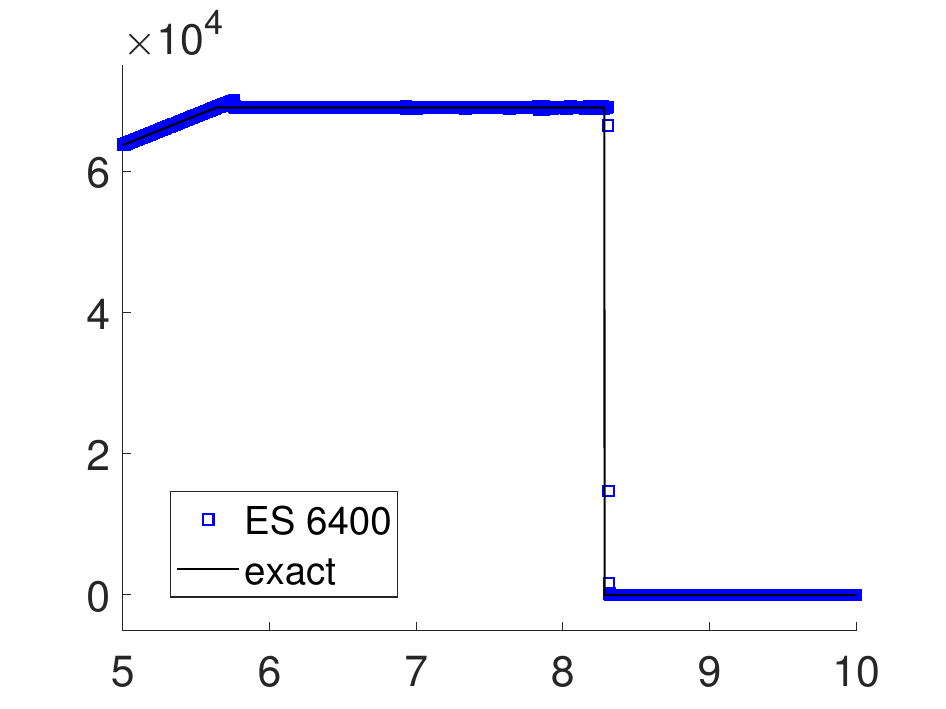}}
		\subfigure[$\lg p$.]{
			\includegraphics[width=0.45\linewidth]{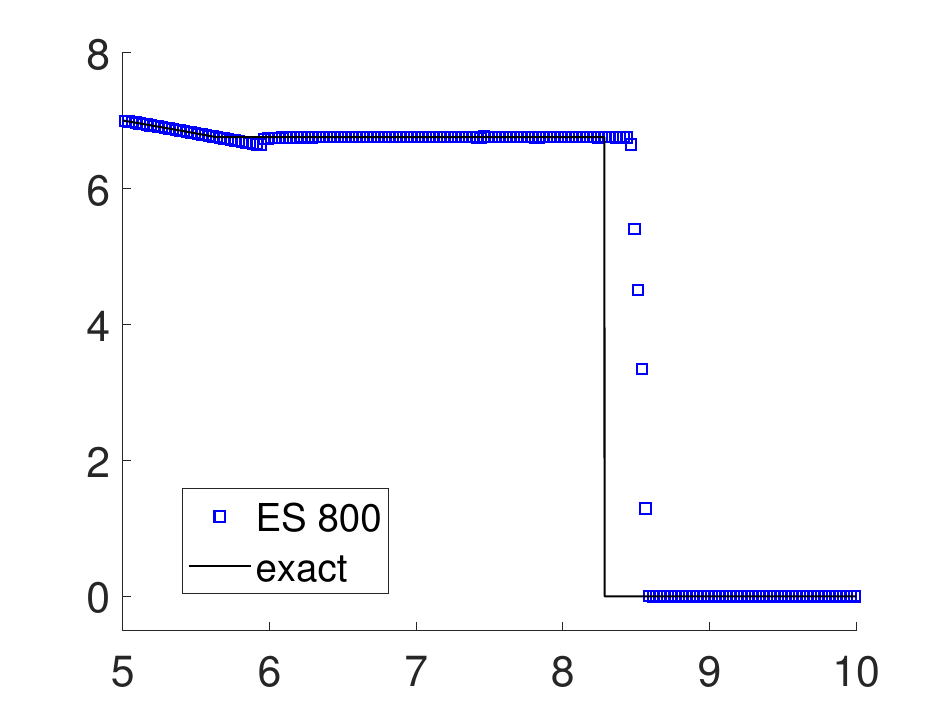}}
		\subfigure[$\lg p$.]{
			\includegraphics[width=0.45\linewidth]{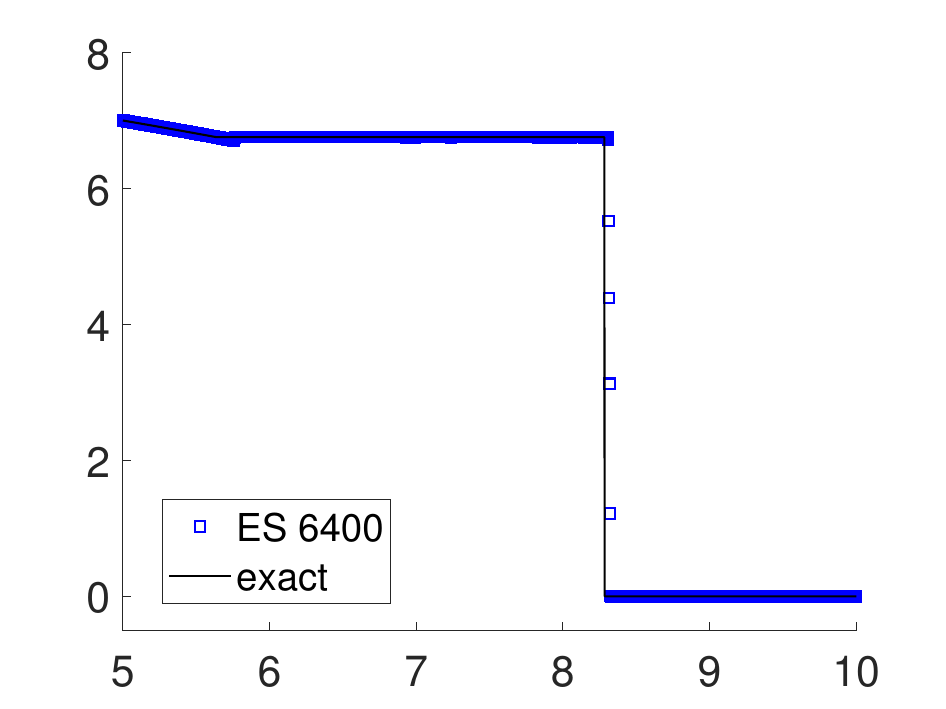}}
		\caption{Example \ref{ex:Leblanc}: Leblanc shock tube. The zoomed-in results at $T = 0.0001$ with $N_x=800$ and $N_x= 6400$. }
		\label{figLeb2}
	\end{figure}

	\subsection{Two-dimensional scalar equations}
	
	\begin{Ex}
		[Two-dimensional Burgers' equation]
		\label{ex:burgers_2D}
	Here, we consider the two-dimensional Burgers' equation
	\begin{equation}
		u_t+\left(\frac{u^2}{2}\right)_x+\left(\frac{u^2}{2}\right)_y=0.
	\end{equation}
	The equation is defined on the computational domain $\Omega=[0,2\pi]\times[0,2\pi]$ with periodic boundaries and initial data $u(x,y,0)=\sin(x+y)$. The entropy pair is chosen as
	$$(\mathcal U,\mathcal F_1,\mathcal F_2)=(e^u,(u-1)e^u,(u-1)e^u).$$
	We simulate the equation until the final time $T=0.3$. The errors and orders for $k=1,2,3$ are presented in Table \ref{tab4}. It can be seen that our scheme also maintains the optimal convergence rate for the two-dimensional case.
	\end{Ex}

	\begin{table}[htb!]
		\centering
		\caption{Example \ref{ex:burgers_2D}: Two-dimensional Burgers' equation. The $L^2$ errors and orders of different schemes at $T=0.3$. }
		\setlength{\tabcolsep}{3.2mm}{
			\begin{tabular}{|c|cc|cc|}
				\hline $N_x\times N_y$ & Base & Order & ES & Order  \\ 
				\hline
				\multicolumn{5}{|c|}{$k=1$}\\
				
				\hline  
				$20\times 20$ &1.9655E-02 & -- &1.9805E-02 & --  \\
				\hline  
				$40\times 40$ &4.7351E-03 & 2.0534 &4.8285E-03 & 2.0362  \\
				\hline  
				$80\times 80$ &1.0715E-03 & 2.1437 &1.0870E-03 & 2.1512  \\
				\hline  
				$160\times 160$ &2.5297E-04 & 2.0827 &2.5446E-04 & 2.0948  \\
				\hline

				\multicolumn{5}{|c|}{$k=2$}\\		
				\hline  
				$20\times 20$ &4.4015E-03 & -- &4.3996E-03 & --  \\
				\hline  
				$40\times 40$ &6.5681E-04 & 2.7444 &6.5702E-04 & 2.7433  \\
				\hline  
				$80\times 80$ &9.9602E-05 & 2.7212 &9.9669E-05 & 2.7207  \\
				\hline  
				$160\times 160$ &1.4385E-05 & 2.7916 &1.4391E-05 & 2.7920  \\
				\hline
				
				\multicolumn{5}{|c|}{$k=3$}\\
				
				\hline  
				$20\times 20$ &6.4805E-04 &-- &6.4865E-04 & --  \\
				\hline  
				$40\times 40$ &8.2437E-05 & 2.9747 &8.2458E-05 & 2.9757  \\
				\hline  
				$80\times 80$ &5.6425E-06 & 3.8689 &5.6432E-06 & 3.8691  \\
				\hline  
				$160\times 160$ &3.4623E-07 & 4.0265 &3.4625E-07 & 4.0266  \\
				\hline
		\end{tabular}} 
		\label{tab4}
	\end{table}

	\begin{Ex}
		[Two-dimensional Buckley--Leverett equation]
		\label{ex:BL_2D}
	We consider the two-dimensional Buckley--Leverett equation. The nonconvex flux functions are given by
	\begin{equation}
		f(u)=\frac{u^2}{u^2+(1-u)^2},\quad g(u)=\frac{u^2(1-5(1-u)^2)}{u^2+(1-u)^2}.
	\end{equation}
	The computational domain is $\Omega=[-1.5,1.5]\times [-1.5,1.5]$ with periodic boundaries. The initial data is
	$$
	u(x,y,0) =\begin{cases}
		1 , & x^2+y^2<0.5,\\
		0 , & \mathrm{otherwise}.\\
	\end{cases}
	$$ 
	We consider two kinds of entropy functions:
	$$\mathcal U^{(1)}=\frac{u^2}{2},\quad \mathcal U^{(2)}=\int^u \arctan(20(u-0.5))\mathrm du.$$
	It should be noted here that $\mathcal U^{(2)}$ is not a classical setup existing in the literature. Here, we extend the one-dimensional choice to the two-dimensional setting as a numerical test. In this problem, the BP limiter is employed. The numerical solutions with different entropy pairs at $T=0.5$ on a mesh with $N_x\times N_y=200\times 200$ cells are presented in Fig. \ref{figBL2D}. It can be observed that with the square entropy $\mathcal U^{(1)}$, minor nonphysical structures appear near the points $(0.5,-1), \ (1,0)$ and $(1,1)$. In contrast, the scheme using $\mathcal U^{(2)}$ gives a more satisfactory result. Moreover, the result of U1+U2 is similar to that of U2, and is not presented here. 
	This comparison underlines the critical role of entropy choice in achieving superior numerical accuracy.
	\end{Ex}

	\begin{figure}[htb!]
		\centering
		\subfigure[U1.]{
			\includegraphics[width=0.45\linewidth]{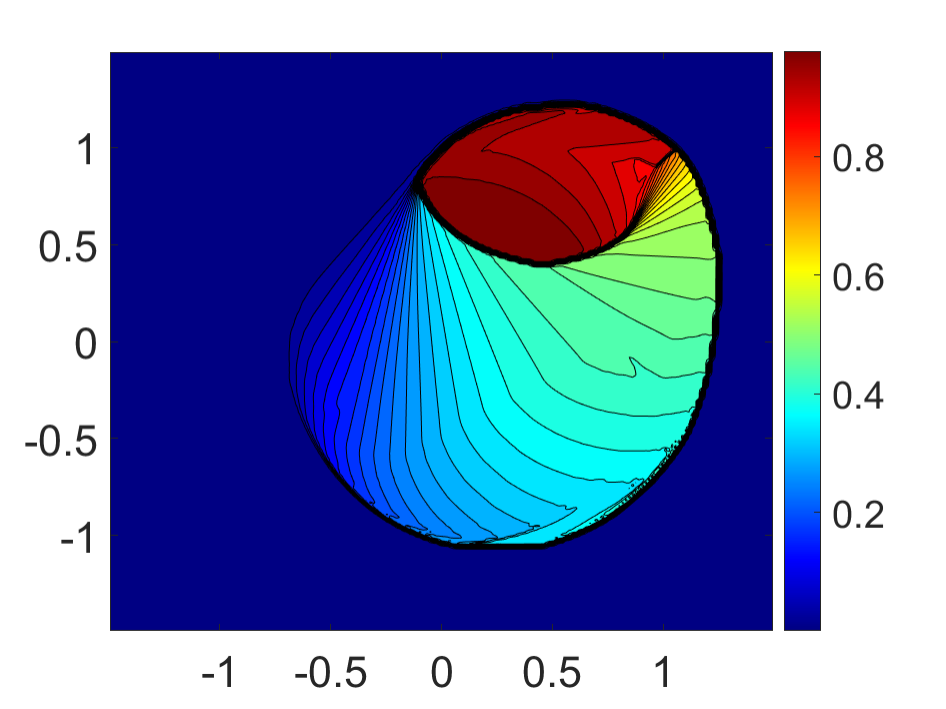}}
		\subfigure[U2.]{
			\includegraphics[width=0.45\linewidth]{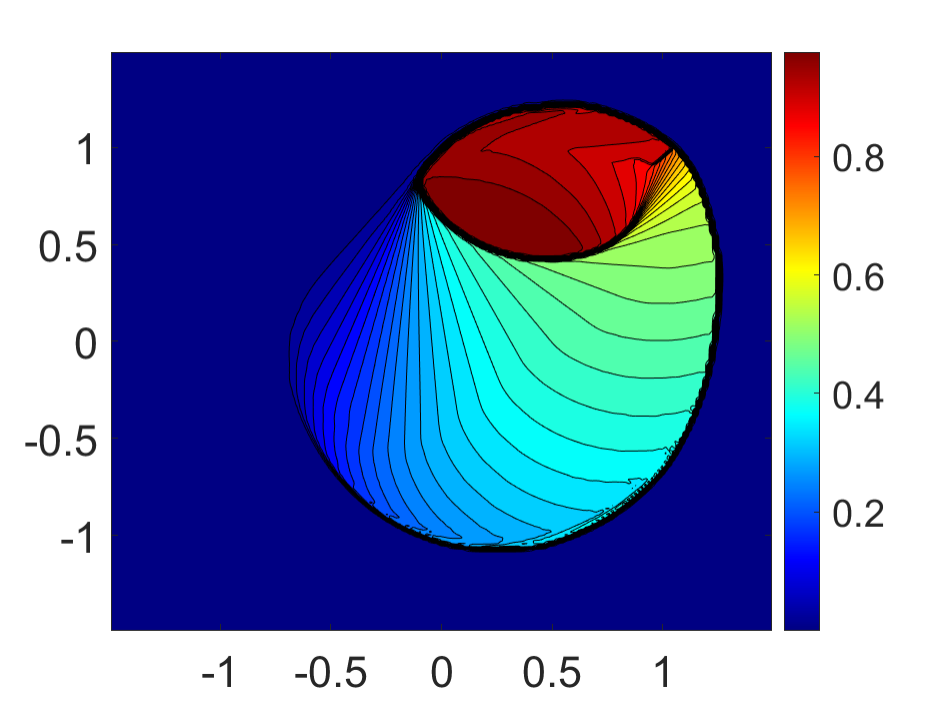}}
		\caption{Example \ref{ex:BL_2D}: Two-dimensional Buckley--Leverett equation. The numerical solution at $T=0.5$ with $N_x\times N_y=200\times 200$. 40 contour lines are used.}
		\label{figBL2D}
	\end{figure}
	
	\subsection{Two-dimensional Euler systems}
	
	\begin{Ex}
		[Isentropic vortex]
		\label{ex:acc_sys2D}
	We consider the two-dimensional smooth isentropic vortex problem to test the accuracy of the scheme for the two-dimensional Euler equations. Initially, an isentropic vortex perturbation centered at $(5,0)$ is added to the mean flow $(\rho_0,u_0,v_0,p_0)=(1,1,0,1)$:
	$$
	\begin{aligned}
		\rho &=\left( 1-\frac{\gamma -1}{16\gamma \pi ^2}\beta ^2e^{2\left( 1-r^2 \right)} \right) ^{1/\left( \gamma -1 \right)},
		\\
		u&=1-\frac{\beta}{2\pi}ye^{1-r^2},\\ v&=\frac{\beta}{2\pi}\left( x-5 \right) e^{1-r^2},
		\\ p&=\rho ^{\gamma},\end{aligned}
	$$
	where $r^2 = (x - 5)^2 + y^2$ and we take the vortex strength $\beta = 5$. 
	It is clear that the exact solution is the passive convection of the vortex with the mean velocity. In the numerical simulation, the computational domain $\Omega = [0,10]\times[-5,5]$ is extended periodically in both directions. We simulate this problem until $T=10$, when the exact solution agrees with the initial condition. In Table \ref{tab5}, we present the $L^2$ errors and orders of density for $k=1,2,3$. The optimal convergence rates are obtained, and the performance of the ES scheme is comparable to that of the Base scheme. 
	\end{Ex}

	\begin{table}[htb!]
		\centering
		\caption{Example \ref{ex:acc_sys2D}: Isentropic vortex. The $L^2$ errors and orders of density with different schemes at $T=10$. }
		\setlength{\tabcolsep}{3.2mm}{
			\begin{tabular}{|c|cc|cc|}
				\hline $N_x\times N_y$ & Base & Order & ES & Order  \\ 
				\hline
				\multicolumn{5}{|c|}{$k=1$}\\
				
				\hline  
				$20\times 20$ &3.3344E-02 & -- &3.3422E-02 & --  \\
				\hline  
				$40\times 40$ &9.0883E-03 & 1.8754 &9.2825E-03 & 1.8482  \\
				\hline  
				$80\times 80$ &1.9747E-03 & 2.2024 &1.9630E-03 & 2.2414  \\
				\hline  
				$160\times 160$ &2.9980E-04 & 2.7195 &2.9802E-04 & 2.7196  \\
				\hline

				\multicolumn{5}{|c|}{$k=2$}\\		
				\hline  
				$20\times 20$ &4.2906E-03 & -- &3.7907E-03 & --  \\
				\hline  
				$40\times 40$ &2.4672E-04 & 4.1202 &2.4041E-04 & 3.9789  \\
				\hline  
				$80\times 80$ &2.5283E-05 & 3.2866 &2.5174E-05 & 3.2555  \\
				\hline  
				$160\times 160$ &3.6150E-06 & 2.8061 &3.6139E-06 & 2.8003  \\
				\hline
				
				\multicolumn{5}{|c|}{$k=3$}\\
				
				\hline  
				$20\times 20$ &5.8466E-04 & -- &5.5364E-04 & --  \\
				\hline  
				$40\times 40$ &1.9759E-05 & 4.8870 &1.9475E-05 & 4.8292  \\
				\hline  
				$80\times 80$ &8.2453E-07 & 4.5828 &8.1693E-07 & 4.5753  \\
				\hline  
				$160\times 160$ &4.0451E-08 & 4.3493 &4.0772E-08 & 4.3246  \\
				\hline
		\end{tabular}} 
		\label{tab5}
	\end{table}

		\begin{Ex}
			[Double Mach reflection]
			\label{ex:DoubleMach}
		We consider the double Mach reflection problem \cite{woodward1984numerical}, which is a standard test for evaluating high-resolution schemes in capturing complex shock interactions. This test example involves a Mach 10 shock wave initially positioned at the bottom boundary and making a $60^\circ$ angle with the $x$-axis. The computational domain is $\Omega = [0,4]\times[0,1]$, with the shock initially located at $(x,y)=(1/6,0)$. The initial condition is given by
		$$
		\left( \rho ,u,v,p \right) =\begin{cases} 
			\left( 8,8.25\cos \left( \displaystyle\dfrac{\pi}{6} \right) ,-8.25\sin \left( \dfrac{\pi}{6} \right) ,116.5 \right) , & x<\dfrac{1}{6}+\dfrac{y}{\sqrt{3}}, \\
			\left( 1.4,0,0,1 \right) , & x\ge \dfrac{1}{6}+\dfrac{y}{\sqrt{3}}. \\
		\end{cases}
		$$
		The inflow and outflow boundary conditions are imposed on the left and right boundaries, respectively. For the bottom boundary, the exact post-shock condition is applied for the portion $0\le x\le 1/6$, while a reflective boundary condition is used for the remainder. For the top boundary, the post-shock condition is imposed for the region $0<  x < 1/6 + (1 + 20t)/\sqrt 3$, and the pre-shock condition is applied for the rest. This test is particularly useful for assessing the scheme's capability to resolve intricate shock structures, including the Mach stem, the reflected shock, and the complex triple-point trajectory that develops as the shock propagates. 
		
		We simulate this problem until $T=0.2$.  For this example, the PP limiter is utilized. In Fig. \ref{figDoubleMach1}, we present the density results with $N_x\times N_y=960\times 240$ and $N_x\times N_y = 1920\times 480$, and the zoomed-in graphs are shown in Fig. \ref{figDoubleMach2}. The results show that our scheme accurately captures the complex shock structures with remarkably low numerical dissipation. Compared to results in the literature \cite{liu2024non, wei2024jump, liu2024entropy, peng2024oedg}, our method resolves significantly finer flow features. Moreover, compared to the results of the semi-discrete ES 
		scheme \cite{chen2017entropy}, our results are more stable. These comparisons indicate that our fully-discrete ES scheme could achieve a certain balance between low dissipation and numerical stability.
		\end{Ex}

		\begin{figure}[htbp!]
			\centering
			\subfigure[$960\times 240$]{
				\includegraphics[width=0.95\linewidth]{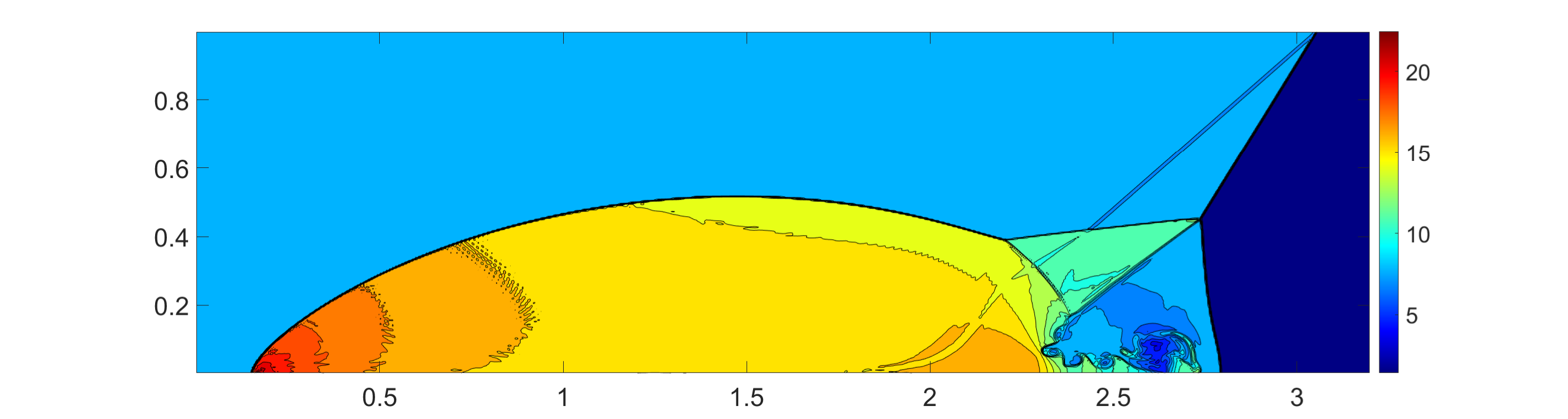}}
			\subfigure[$1920\times 480$]{
				\includegraphics[width=0.95\linewidth]{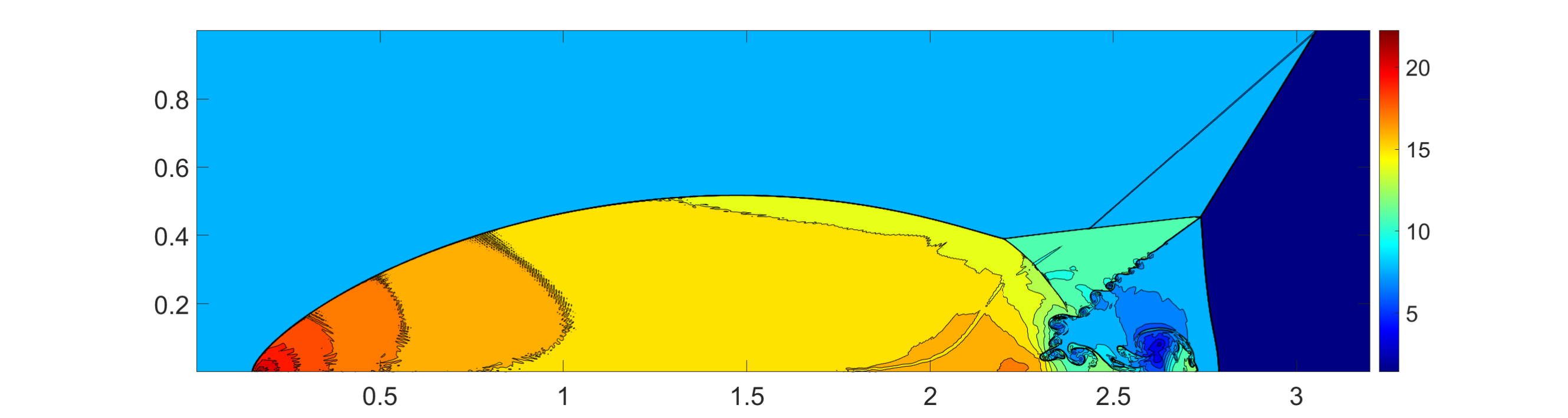}}
			\caption{Example \ref{ex:DoubleMach}: Double Mach reflection. The density at $T = 0.2$. 20 contour lines are used. }
			\label{figDoubleMach1}
		\end{figure}

		\begin{figure}[htbp!]
			\centering
			\subfigure[$960\times 240$]{
				\includegraphics[width=0.45\linewidth]{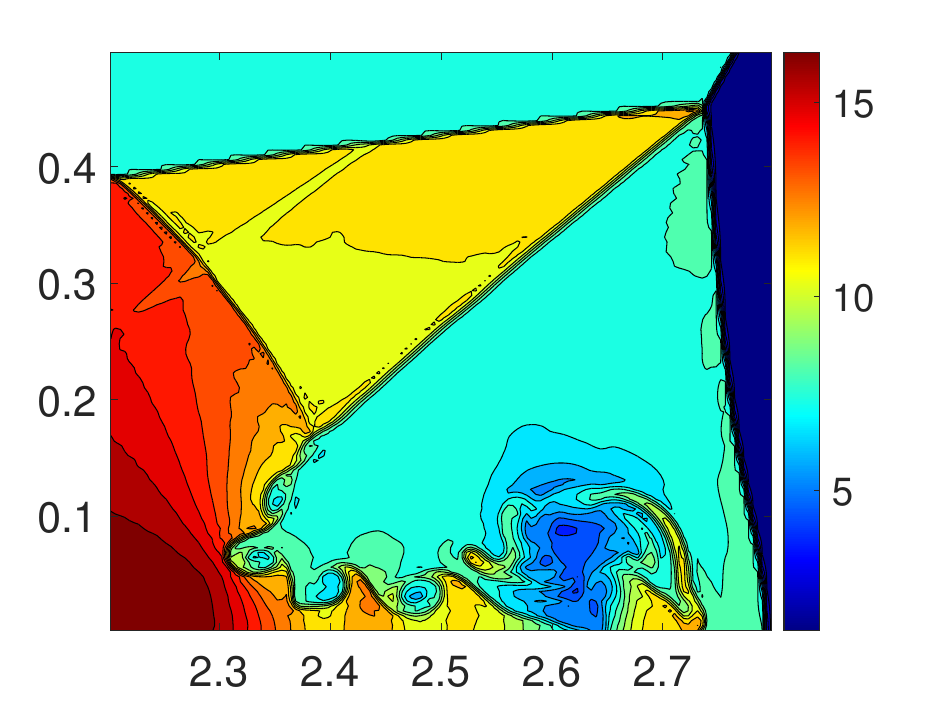}}
			\subfigure[$1920\times 480$]{
				\includegraphics[width=0.45\linewidth]{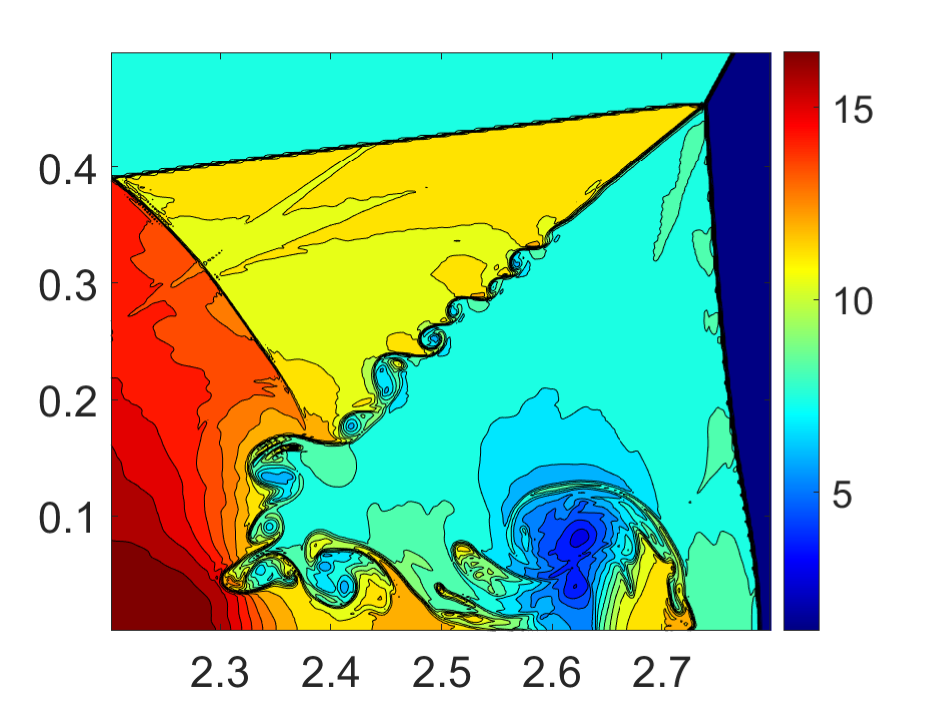}}
			\caption{Example \ref{ex:DoubleMach}: Double Mach reflection. Zoomed-in figure. The numerical solution of density near the Mach stem at $T = 0.2$. 20 contour lines are used. }
			\label{figDoubleMach2}
		\end{figure}

		\begin{Ex}
			[Astrophysical jet]
			\label{ex:jet}
		We consider the high Mach number astrophysical jet problem \cite{zhang2010positivity}. This test is particularly useful for assessing the scheme's robustness and its ability to maintain positivity of density and pressure in extreme flow conditions. This test example involves a Mach 2000 jet propagating into a uniform ambient medium $(\rho,u,v,p)=(0.5,0,0,0.4127)$ in the computational domain $\Omega = [0,1]\times [-0.25,0.25]$. The jet state $(\rho,u,v,p) = (5,800,0,0.4127)$ is injected from the left boundary in the range $\left|y\right|<0.05$. Outflow conditions are applied on all remaining boundaries. The solution exhibits complex shock structures including bow shocks, Mach disks, and internal jet shocks. We note that $\gamma = 5/3$ in this example. For this example, the code can produce negative density or pressure, causing the numerical computation to blow up very easily.
		
		We simulate this problem until $T=0.001$. For this example, the PP limiter is employed. In Fig. \ref{figjet}, we present the results for density, pressure and temperature $p/\rho$ on an $N_x\times N_y=512\times256$ mesh. Our scheme successfully simulates this extremely challenging problem with Mach number 2000 and maintains positivity of density and pressure throughout the computation. Notably, the results exhibit much richer small-scale structures in the shock interaction region compared to existing methods \cite{zhang2010positivity, liu2024entropy, peng2024oedg}, highlighting the superior resolution capability and low-dissipation property achieved through fully-discrete entropy stability.
		\end{Ex}

		\begin{figure}[htbp!]
			\centering
			\subfigure[Density.]{
				\includegraphics[width=0.45\linewidth]{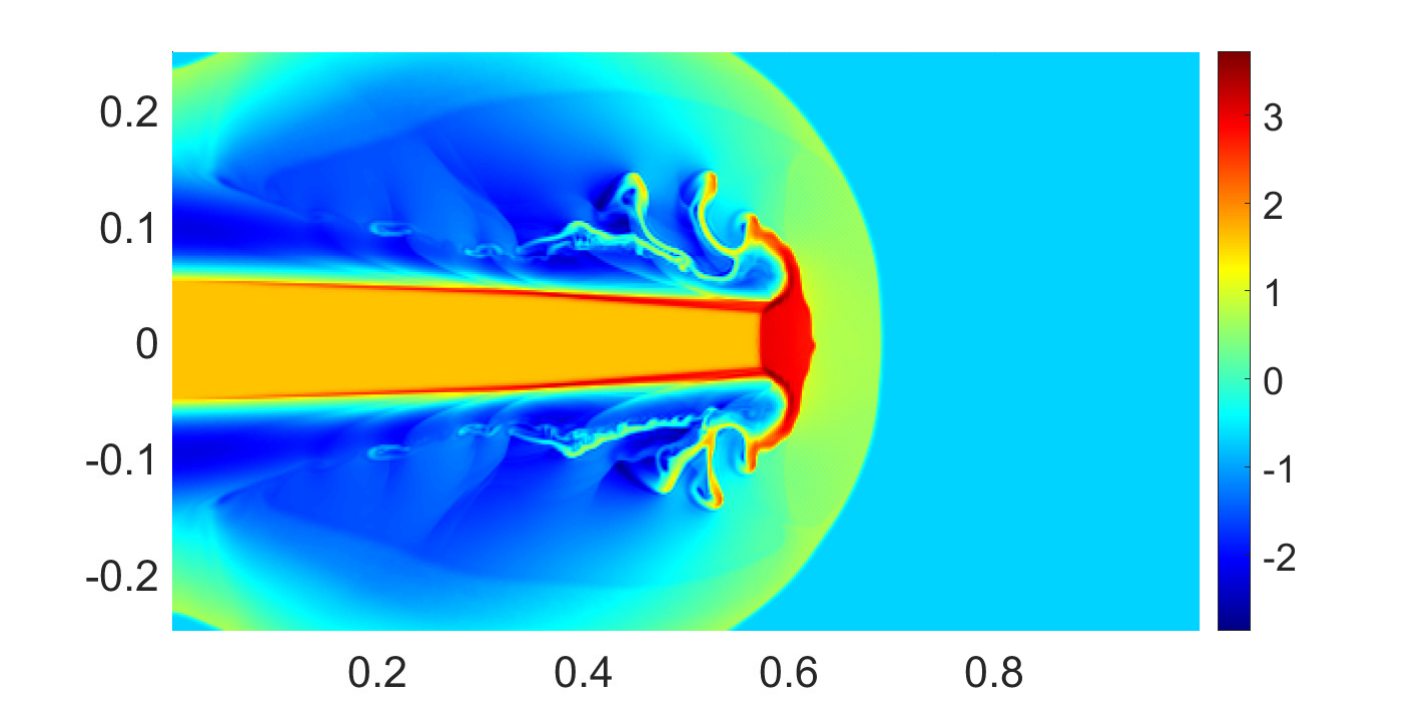}}
			\subfigure[Pressure.]{
				\includegraphics[width=0.45\linewidth]{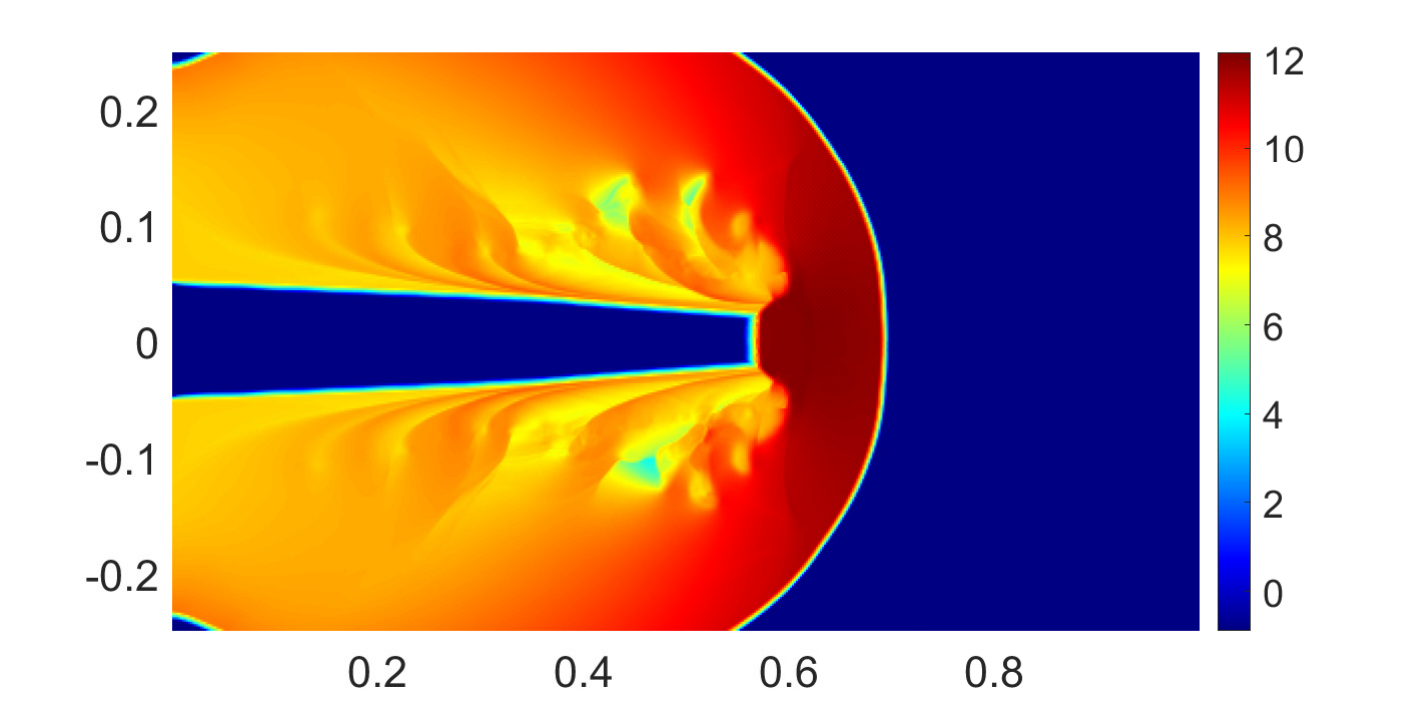}}
			\subfigure[Temperature.]{
				\includegraphics[width=0.45\linewidth]{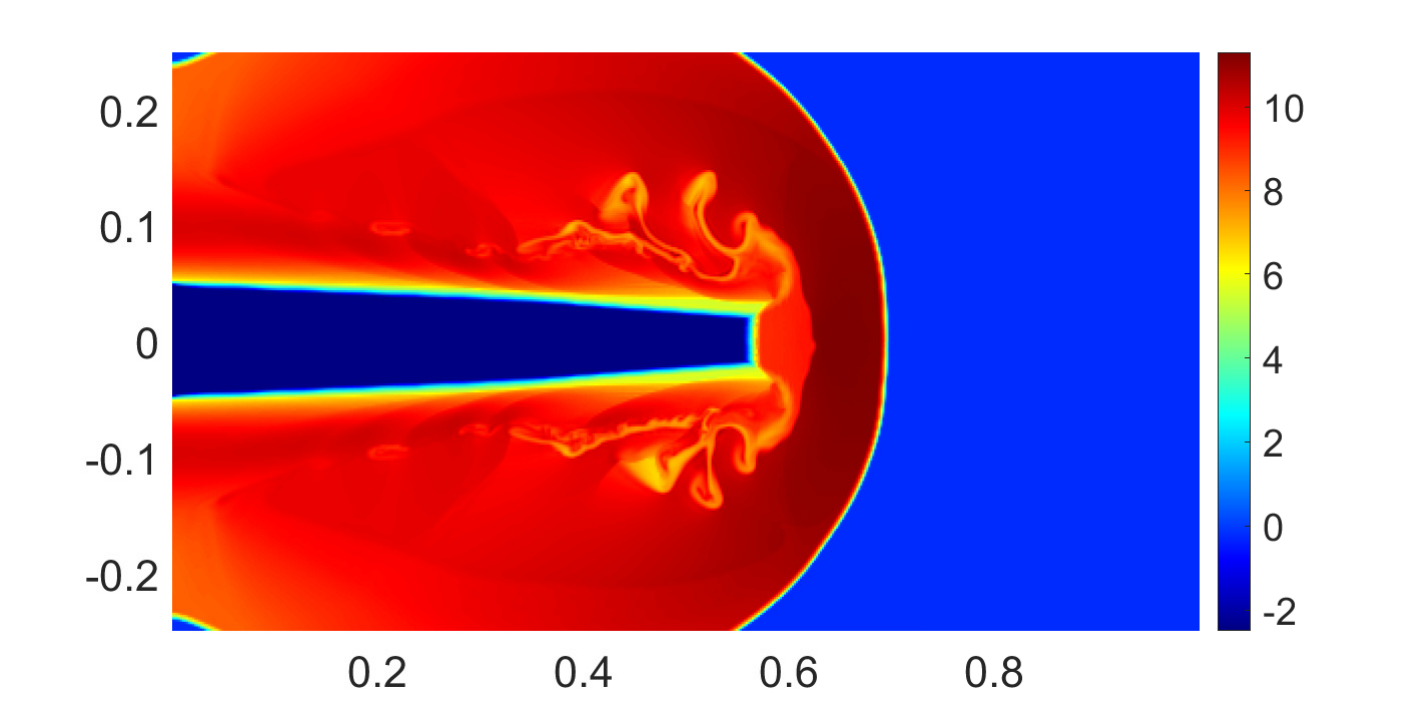}}
			\caption{Example \ref{ex:jet}: Astrophysical jet. The numerical solution  at $T = 0.001$ with $N_x \times  N_y = 512\times 256$. Scales are logarithmic.}
		\label{figjet}
	\end{figure}

	\section{Concluding remarks}\label{sec6}
	
	In this paper, we have developed, to our knowledge, the first fully-discrete entropy-stable explicit DG scheme for hyperbolic conservation laws through a novel limiter-based post-processing approach. The key idea is based on the fact that the cell averages of the classical DG scheme with forward Euler time discretization inherently satisfy an ES-like property, even without any modifications. Building upon this insight, we achieve rigorous entropy stability by compressing the numerical solution toward the cell averages using the Zhang--Shu limiter. High-order temporal accuracy is attained through SSP multistep methods while maintaining entropy stability. Theoretically, we prove the high-order accuracy of the limiter and establish the Lax--Wendroff-type theorems for the scheme.
	
	Numerical experiments validate the theoretical predictions and demonstrate optimal convergence rates. For the Buckley--Leverett problem, our scheme accurately captures multiple classical discontinuous structures simultaneously by imposing multiple entropy inequalities. For the Euler equations, even without shock limiters, our scheme handles extreme test cases, highlighting the crucial importance of fully-discrete entropy stability over semi-discrete approaches.
	
	Future research directions include extending the framework to SSP-RK time discretization and other conservation laws such as ideal MHD equations.
	
	\appendix

{
\section{Existence of proper numerical entropy fluxes}
\label{app:properflux} 
In this appendix, we analyze when the corresponding proper numerical entropy flux $\hat{\mathcal F}$ exists for a given $\hat f$. First, we consider the case that $\hat f$ and $\hat{\mathcal F}$ are both $C^2$, in which the existence of a unique $\hat{\mathcal F}$ is equivalent to the separability of $\hat f$. For the proof of existence in the sufficiency direction, also see \cite{karlsen2014l1}. And the conclusion for more general cases is given in Theorem \ref{thm:properflux}.

    \begin{theorem}
    For a $C^2$ numerical flux $\hat f(\cdot,\cdot)$ on $\mathcal D\times\mathcal D$, where $\mathcal D\subset\mathbb R$ is an open interval, there exists a unique $C^2$ proper numerical entropy flux $\hat{\mathcal F}(\cdot,\cdot)$ if and only if $\hat f$ has the form
    \begin{equation}\label{eq:properf2} \hat f(a,b)=h(a)+g(b). 
    \end{equation}
    \end{theorem}
    \begin{proof}
    If a proper numerical entropy flux exists, then \eqref{eq:properflux} gives
    $$ \hat{\mathcal F}_{12}(a,b)=\mathcal U'(a)\hat f_{12}(a,b),\quad \hat{\mathcal F}_{21}(a,b)=\mathcal U'(b)\hat f_{21}(a,b),\quad \forall a,b\in\mathcal D. $$
    Here, $(\cdot)_{12}$ denotes the mixed derivative with respect to the first and second components, in that order. Similarly, $(\cdot)_{21}$ denotes the mixed partial derivative with respect to the second and first components, in that order.
    Since $\hat{\mathcal F},\hat f\in C^2(\mathcal D\times\mathcal D)$ and $\mathcal U$ is strictly convex, we get $\hat f_{12}(a,b)=0$ for any $a\ne b$.
    By continuity, $\hat f_{12}\equiv0$ on
    $\mathcal D\times\mathcal D$, which implies
    $\hat f(a,b)=h(a)+g(b)$.
    
    Conversely, assume $\hat f(a,b)=h(a)+g(b)$. Fix $c\in\mathcal D$ and, following \cite{karlsen2014l1}, set 
    \begin{equation}\label{eq:Fflux} \hat{\mathcal F}(a,b)=\mathcal F(c)+\int_c^a  \mathcal U'(s)h'(s)\mathrm ds+\int_c^b \mathcal U'(s)g'(s) \mathrm ds. 
    \end{equation}
    Then it is straightforward to verify \eqref{eq:properflux}. On the other hand, consistency of $\hat f$ gives
$h(u)+g(u)=f(u)$. Hence,
    $$ \hat{\mathcal F}(u,u)=\mathcal F(c)+\int_c^u\mathcal U'(s)(h'(s)+g'(s))\mathrm ds=\mathcal F(u),   $$
    which verifies the consistency. Furthermore, let $\hat{\mathcal G}$ be another proper numerical
    entropy flux corresponding to the same numerical flux $\hat f$ and
    entropy pair $(\mathcal U,\mathcal F)$. Then
    \[
        \hat{\mathcal F}_1-\hat{\mathcal G}_1=0,\qquad
        \hat{\mathcal F}_2-\hat{\mathcal G}_2=0,
    \]
    so $\hat{\mathcal G}=\hat{\mathcal F}+C$.
    Consistency gives
    \[
        C=\hat{\mathcal G}(u,u)-\hat{\mathcal F}(u,u)
         =\mathcal F(u)-\mathcal F(u)=0.
    \]
    Hence $\hat{\mathcal F}=\hat{\mathcal G}$.
    \end{proof}
    \begin{remark}
    In general, Tadmor-type entropy stability \cite{tadmor2003entropy} does not guarantee the
    existence of a proper numerical entropy flux.
    As an example, for Burgers' equation, $f(u)=u^2/2$, with the quadratic
    entropy $\mathcal U(u)=u^2/2$, consider the entropy
    conservative flux supplemented with constant dissipation:
    \[
        \hat f^{\mathrm{ES}}(a,b)
        =\frac{a^2+ab+b^2}{6}
         -\frac{\alpha}{2}(b-a).
    \]
    Treating $\alpha$ as a fixed sufficiently large constant makes this flux consistent, Lipschitz continuous, and
    monotone. Although it is entropy stable and belongs to
    $C^\infty(\mathcal D\times \mathcal D)$, we have
    \[
        \hat f^{\mathrm{ES}}_{12}(a,b)=\frac16\ne0.
    \]
    Thus the separability condition in
    Theorem~\ref{thm:properflux} fails, and no locally Lipschitz
    proper numerical entropy flux exists.
    \end{remark}
    \begin{remark}
    For the LF flux \eqref{eq:LF}, treating $\alpha$ as a fixed constant, we have
    \[
        \hat f(a,b)
        =\frac{f(a)+\alpha a}{2}
         +\frac{f(b)-\alpha b}{2}=:h(a)+g(b).
    \]
    Thus, for fixed $\alpha$, the LF flux has precisely the
    separable form in the above theorem and admits a unique proper
    numerical entropy flux, given by \eqref{eq:LFF}. 
    \end{remark}
    \begin{remark}
    The $C^2$ regularity assumptions on $\hat f$ and
    $\hat{\mathcal F}$ can be relaxed to $C^1$, as shown later.
    Another well-known $C^1$ numerical flux is the
    Engquist--Osher flux,
    \[
        \hat f^{EO}(a,b)
        =f(0)
         +\int_0^a\max\{f'(u),0\}\,\mathrm du
         +\int_0^b\min\{f'(u),0\}\,\mathrm du.
    \]
    It has the separable form \eqref{eq:properf2} and therefore
    admits a unique proper numerical entropy flux, given by
    the construction in \eqref{eq:Fflux}.
    \end{remark}

    The above characterization also extends to general monotone 
    numerical fluxes \cite{shu2009discontinuous}. To include nonsmooth Lipschitz numerical fluxes, we interpret \eqref{eq:properflux} as holding almost everywhere. The proof is similar to the $C^2$ case but becomes more complex. We first recall some notation and basic properties of
    distributional derivatives. For a locally Lipschitz function $w$ on
    $\mathcal D\times \mathcal D$, let $T_w=\partial_a\partial_bw$ in the
    distributional sense. Its action extends to compactly
    supported $C^1$ test functions by
    \[
        \langle T_w,\phi\rangle
        =-\int_{\mathcal D}\int_{\mathcal D}w_1\phi_2\,\mathrm da\,\mathrm db
        =-\int_{\mathcal D}\int_{\mathcal D}w_2\phi_1\,\mathrm da\,\mathrm db,
        \qquad \phi\in C_0^1(\mathcal D\times \mathcal D),
    \]
    where $w_1,w_2$ denote the weak partial derivatives. Consequently, if $\Omega\subset\mathcal D\times \mathcal D$ is open,
    $q\in C^1(\Omega)$ is nonzero everywhere, and
    \[
        \langle T_w,q\phi\rangle=0,
        \qquad \forall\,\phi\in C_0^1(\Omega),
    \]
    then $T_w=0$ in $\Omega$. The following lemma implies that a locally Lipschitz function
    with vanishing mixed distributional derivative on either side
    of the diagonal $a=b$ admits a separable representation. 
    The proof uses standard properties
    of weak derivatives.

    \begin{lemma}\label{lem:Tw}
    Define
    \[
        \mathcal D_\pm
        =\{(a,b)\in\mathcal D\times \mathcal D:\pm(a-b)>0\}.
    \]
    For each $\sigma\in\{+,-\}$, if $T_w=0$ in
    $\mathcal D_\sigma$, i.e.
    \[
        \langle T_w,\phi\rangle=0,
        \qquad \forall\,\phi\in C_0^1(\mathcal D_\sigma),
    \]
    then there exist locally Lipschitz functions
    $h_\sigma,g_\sigma$ on $\mathcal D$ such that
    \[
        w(a,b)=h_\sigma(a)+g_\sigma(b),
        \qquad (a,b)\in\mathcal D_\sigma.
    \]
    \end{lemma}

    \begin{theorem}\label{thm:properflux}
    For a consistent, Lipschitz continuous, and monotone flux $\hat f$ on $\mathcal D\times\mathcal D$, where $\mathcal D\subset\mathbb R$ is an open interval, there exists a unique locally Lipschitz continuous proper numerical entropy flux $\hat{\mathcal F}$, satisfying
    \[
        \hat{\mathcal F}_1(a,b)
          =\mathcal U'(a)\hat f_1(a,b),\qquad
        \hat{\mathcal F}_2(a,b)
          =\mathcal U'(b)\hat f_2(a,b)
        \quad\text{a.e. in } \mathcal D\times\mathcal D,
    \]
    and $\hat{\mathcal F}(u,u)=\mathcal F(u)$, if and only if
    there exist locally Lipschitz continuous functions
    $h_\pm,g_\pm:\mathcal D\to\mathbb R$ such that
    \begin{equation}\label{eq:properf}
        \hat f(a,b)=
        \begin{cases}
            h_+(a)+g_+(b), & a>b,\\
            f(a), & a=b,\\
            h_-(a)+g_-(b), & a<b.
        \end{cases}
    \end{equation}
    In particular, 
    \begin{equation}\label{eq:properF}
        \hat{\mathcal F}(a,b)=
        \begin{cases}
            \hat{\mathcal F}_+(a,b), & a>b,\\
            \mathcal F(a),          & a=b,\\
            \hat{\mathcal F}_-(a,b), & a<b
        \end{cases}
    \end{equation}
    with
    \[
        \hat{\mathcal F}_\pm(a,b)
        =\mathcal F(c)
         +\int_c^a\mathcal U'(s)h_\pm'(s)\,\mathrm ds
         +\int_c^b\mathcal U'(s)g_\pm'(s)\,\mathrm ds
    \]
    for a fixed $c\in\mathcal D$.
    \end{theorem}
    
    \begin{proof}
    Suppose that a proper numerical entropy flux exists. The proper identities \eqref{eq:properflux} give
    \[
        \left\langle T_{\hat f},
          \bigl(\mathcal U'(a)-\mathcal U'(b)\bigr)\phi
        \right\rangle
        =
        \left\langle
          \partial_b\hat{\mathcal F}_1
          -\partial_a\hat{\mathcal F}_2,\phi
        \right\rangle
        =0,
        \quad \forall \phi\in C_0^1(\mathcal D_\sigma),\ \sigma\in\{+,-\}.
    \]
    Strict convexity of $\mathcal U$ implies that
    $\mathcal U'(a)-\mathcal U'(b)\ne0$ whenever $a\ne b$; hence $T_{\hat f}=0$ in $\mathcal D_\sigma$, $\sigma\in\{+,-\}$.
    Utilizing Lemma~\ref{lem:Tw} yields \eqref{eq:properf}.
    
    Conversely, assume \eqref{eq:properf}.
    For $u\in\mathcal D$, continuity and consistency of $\hat f$ imply
    \[
        h_\pm(u)+g_\pm(u)=f(u).
    \]
    Fix $c\in\mathcal D$ and define
    $\hat{\mathcal F}_\pm$ as in \eqref{eq:properF}. For each $\sigma\in\{+,-\}$, the following identities
hold almost everywhere in $\mathcal D_\sigma$:
    \[
        \begin{aligned}
        (\hat{\mathcal F}_\sigma)_1(a,b)
        &=\mathcal U'(a)h_\sigma'(a)
          =\mathcal U'(a)\hat f_1(a,b),\\
        (\hat{\mathcal F}_\sigma)_2(a,b)
        &=\mathcal U'(b)g_\sigma'(b)
          =\mathcal U'(b)\hat f_2(a,b).
        \end{aligned}
    \]
    Moreover, using $\mathcal F'=\mathcal U'f'$, we obtain
    \[
        \begin{aligned}
        \hat{\mathcal F}_\pm(u,u)
        &=\mathcal F(c)
          +\int_c^u\mathcal U'(s)
            \bigl(h_\pm'(s)+g_\pm'(s)\bigr)\,\mathrm ds\\
        &=\mathcal F(c)
          +\int_c^u\mathcal F'(s)\,\mathrm ds
          =\mathcal F(u).
        \end{aligned}
    \]
    {
    Both functions $\hat{\mathcal F}_\pm$ are locally Lipschitz
    on $\mathcal D\times\mathcal D$ and have the same diagonal
    trace $\mathcal F$. Since a line segment not contained in
    the diagonal crosses it at most once, their piecewise
    combination in \eqref{eq:properF} is locally Lipschitz
    on $\mathcal D\times\mathcal D$.}
    The above identities and consistency
    show that $\hat{\mathcal F}$ is a proper numerical entropy flux. 
    
    Finally, let $\hat{\mathcal G}$ be another proper numerical
    entropy flux for the same numerical flux and entropy pair.
    Then
    \[
        \nabla(\hat{\mathcal G}-\hat{\mathcal F})=0
        \quad\text{a.e. in }\mathcal D\times \mathcal D.
    \]
    Hence, 
    $\hat{\mathcal G}=\hat{\mathcal F}+C$.
    Consistency gives
    \[
        C=\hat{\mathcal G}(u,u)-\hat{\mathcal F}(u,u)
         =\mathcal F(u)-\mathcal F(u)=0,
    \]
    which yields $\hat{\mathcal F}=\hat{\mathcal G}$.
    \end{proof}
    \begin{remark}
    The Godunov flux does not admit a proper
    numerical entropy flux. Consider Burgers' equation,
    $f(u)=u^2/2$, on a bounded open interval $\mathcal D$
    containing zero. For $a>0>b$, its Godunov flux satisfies
    \[
        \hat f_G(a,b)=\max\left\{\frac{a^2}{2},
                                  \frac{b^2}{2}\right\}.
    \]
    If $\hat f_G(a,b)=h_+(a)+g_+(b)$ on $\mathcal D_+$,
    then for any rectangle
    $[a_1,a_2]\times[b_1,b_2]\subset\mathcal D_+$, we have
    \[
        \hat f_G(a_1,b_1)+\hat f_G(a_2,b_2)-\hat f_G(a_1,b_2)-\hat f_G(a_2,b_1)=0.
    \]
    Choose $s>0$ sufficiently small that
    $[-2s,2s]\subset\mathcal D$, and take
    $a_1=s$, $a_2=2s$, $b_1=-2s$, and $b_2=-s$.
    The resulting rectangle lies in $\mathcal D_+$, but
    \[
        \hat f_G(s,-2s)+\hat f_G(2s,-s)
        -\hat f_G(s,-s)-\hat f_G(2s,-2s)
        =\frac32s^2\ne0.
    \]
    Thus the separability condition in
    Theorem~\ref{thm:properflux} fails, and no corresponding proper numerical entropy flux exists.
    \end{remark}
    \begin{remark}
    This condition has an interesting consequence: existence
    is primarily determined by the structure of the numerical
    flux. If a numerical flux admits a proper numerical
    entropy flux for one strictly convex entropy, then it
    satisfies the above separability condition \eqref{eq:properf}, and a
    corresponding proper numerical entropy flux can therefore
    be constructed for any $C^2$ convex entropy.
    The entropy function determines the specific flux $\hat{\mathcal F}$ obtained
    from the construction, but within the class of strictly
    convex entropies, it does not change the existence condition.
    \end{remark}
    }

		\section{Lax--Wendroff-type Theorems}
		
		Before the proofs, we first recall the boundedness properties of the DG solution, which play a crucial role in the proofs of the theorems.
		\begin{proposition}[Boundedness]\label{assu:bnd}
			There exists a constant $C$ such that for any bounded $u_h, v_h\in V_h^k$,
			\begin{align}
				\left|\bar u_i-\bar v_i\right|&\le C\left\|u_h-v_h\right\|_{L^\infty(B_i)},\label{eq:bnd1}
				\\
				\left|\hat {f}_{i\pm 1/2}(u_h)-\hat {f}_{i\pm 1/2}(v_h)\right|&\le C\left\|u_h-v_h\right\|_{L^\infty(B_i)}. \label{eq:bnd2}
			\end{align}
			Here, the constant $C$ depends on the specific ranges of $u_h$ and $v_h$, but is independent of $\Delta x$.
		\end{proposition}

		Based on this, we can obtain the boundedness of the entropy pair:
		
		\begin{corollary}[Boundedness of the entropy pair]\label{cor:bnd}
			There exists a constant $C$ such that for any bounded $u_h,v_h\in V_h^k$, 
			\begin{align}
				\left|\tilde{\mathcal U}_i(u_h)-\tilde{\mathcal U}_i (v_h)\right|&\le C\left\|u_h-v_h\right\|_{L^\infty(B_i)},\label{eq:bndU1}
				\\
				\left|\hat {\mathcal F}_{i\pm 1/2}(u_h)-\hat {\mathcal F}_{i\pm 1/2}(v_h)\right|&\le C\left\|u_h-v_h\right\|_{L^\infty(B_i)}. \label{eq:bndU2}
			\end{align}
			In particular,
			\begin{align}
				\left|\tilde{\mathcal U}_i^n-\mathcal U\left(u_h\left(x_i\right)\right)\right|&\le C\left\|u_h(x)-u_h(x_i)\right\|_{L^\infty(B_i)},\label{eq:bndUbar}
				\\
				\left|\hat{\mathcal F}_{i\pm 1/2}-\mathcal F(u_h(x_i))\right|&\le C\left\|u_h(x)-u_h(x_i)\right\|_{L^\infty(B_i)}.\label{eq:bndF}
			\end{align}
			Here, the constant $C$ depends on the specific range of $u_h$ and $v_h$, but is independent of $\Delta x$.
		\end{corollary}
		\begin{proof}
			Assume $u_h,v_h\in [m,M]$, and denote $C_\mathcal U=\max\limits_{\eta\in[m,M]}\left| \mathcal U'(\eta)\right|$. For the entropy average $\tilde{\mathcal U}_i$, we have
			$$\begin{aligned}\left| \tilde{\mathcal{U}}_i\left( u_h \right) -\tilde{\mathcal{U}}_i\left( v_h \right) \right|&=\left| \sum_{q=1}^N{\omega _q\left( \mathcal{U} \left( u_h\left( \hat{x}_{i}^{q} \right) \right) -\mathcal{U} \left( v_h\left( \hat{x}_{i}^{q} \right) \right) \right)} \right|
				\\
				&\le C_\mathcal U \sum_{q=1}^N{\omega _q\left| u_h\left( \hat{x}_{i}^{q} \right) -v_h\left( \hat{x}_{i}^{q} \right) \right|} 
				\\
				&\le C_\mathcal U\sum_{q=1}^N{\omega _q\left\| u_h-v_h \right\| _{L^{\infty}\left( B_i \right)}} = C_\mathcal U\left\| u_h-v_h \right\| _{L^{\infty}\left( B_i \right)}.
			\end{aligned}
			$$
            For the locally Lipschitz numerical entropy flux, let
            $L_{\mathcal F}$ be its Lipschitz constant. 
            Then
            \[
                \left|\hat{\mathcal F}(u^-,u^+)
                      -\hat{\mathcal F}(v^-,v^+)\right|
                \le L_{\mathcal F}
                    \left(\left|u^--v^-\right|
                         +\left|u^+-v^+\right|\right).
            \]
            Hence,
            \[
                \left|\hat{\mathcal F}_{i\pm1/2}(u_h)
                      -\hat{\mathcal F}_{i\pm1/2}(v_h)\right|
                \le 2L_{\mathcal F}
                    \left\|u_h-v_h\right\|_{L^\infty(B_i)}.
            \]
            This completes the proof.
		\end{proof}

		Moreover, under Assumptions \ref{assu:conv} and \ref{assu:tvd}, we can further obtain the following useful properties of the entropy pair. They are easy to prove by using the mean-value theorem.
		
		\begin{corollary}[Convergence of the entropy pair]\label{cor:conv}
			Under Assumption \ref{assu:conv}, for the initial condition, the entropy satisfies
			\begin{equation}\label{eq:convergeicU}
				\int_{\mathbb R}(\mathcal U(u_h^0(x))-\mathcal U(u_0(x)))\phi(x)\mathrm dx\to 0\quad \mathrm{as}\ \ \Delta x\to 0
			\end{equation}
			for any $\phi\in C_0^\infty(\mathbb R)$. And $\mathcal U(u_h),\mathcal F(u_h)$ converge boundedly  
			to $\mathcal U(u^\star(x,t)), \mathcal F(u^\star(x,t))$ in the sense of 
			\begin{align}
				\sum\limits_{n}\int_{t^n}^{t^{n+1}}\int_{\mathbb R}(\mathcal U(u_h^{n+1}(x))-\mathcal U(u^\star(x,t)))\phi(x,t)\mathrm dx \mathrm dt&\to 0\quad \mathrm{as}\ \  \Delta x,\Delta t\to 0,\label{eq:convergeU}
				\\ \sum\limits_{n}\int_{t^n}^{t^{n+1}}\int_{\mathbb R}(\mathcal F(u_h^{n+1}(x))-\mathcal F(u^\star(x,t)))\phi(x,t)\mathrm dx \mathrm dt&\to 0\quad \mathrm{as}\ \  \Delta x,\Delta t\to 0 \label{eq:convergeF}
			\end{align}
			for any $\phi\in C_0^\infty(\mathbb R\times\mathbb R^+)$.
		\end{corollary}
		
		\begin{corollary}[TVB-like property of entropy function]
			Under Assumptions \ref{assu:conv} and \ref{assu:tvd}, the entropy of the numerical solution satisfies
			\begin{equation}\label{eq:tvdU}
				\sup\limits_{n}\Delta x\sum\limits_{i} \max\limits_{x\in B_i}\left|\mathcal U(u_h^n(x))-\mathcal U(u_h^n(x_i))\right|\to 0\quad \mathrm{as}\ \ \Delta x,\Delta t\to 0.
			\end{equation}
		\end{corollary}
		
		Using these properties, we can give the proofs of Theorems \ref{thm:LW1} and \ref{thm:LW2}. The proofs follow the technique in \cite{shi2018local}.

        \subsection{The proof of Theorem \ref{thm:LW1}}\label{app:LW1}

		\begin{proof}
			By the definition of the ES limiter for forward Euler time discretization, we have
			$$\tilde{\mathcal U}_i^{n+1}\le \tilde{\mathcal U}_i^{n}-\lambda\left(\hat{\mathcal F}_{i+1/2}^n-\hat{\mathcal F}_{i-1/2}^n\right).$$
			Multiplying each side by the test function $\phi_i^n:=\phi(x_i,t^n)$ and summing over $i,n$, we get
			$$\sum\limits_{i\in X_\phi}\sum\limits_{n\in T_\phi}\left(\tilde{\mathcal U}_i^{n+1}- \tilde{\mathcal U}_i^{n}\right)\phi_i^n+\lambda\left(\hat{\mathcal F}_{i+1/2}^n-\hat{\mathcal F}_{i-1/2}^n\right)\phi_i^n\le 0.$$
			
            Here, $X_\phi=\{-N_x,\ldots,N_x\}$ and
            $T_\phi=\{0,\ldots,N_t\}$, where the corresponding
            space-time rectangle is bounded independently of the mesh
            and contains the support of $\phi$ with a fixed positive
            margin at its spatial and upper temporal boundaries.
            This choice makes the spatial and upper temporal boundary
            terms vanish, including after all fixed index shifts used below. 
			Summation by parts yields
			\begin{equation}\label{eq:SBP}
            \begin{aligned}-\sum\limits_{i\in X_\phi}\Delta x\,\tilde{\mathcal U}_i^0\phi_i^0&-\sum\limits_{i\in X_\phi}\sum\limits_{n\in T_\phi}\Delta x\,\tilde{\mathcal U}_i^{n+1}\left(\phi_i^{n+1}-\phi_i^n\right)
            \\ &+\sum\limits_{i\in X_\phi}\sum\limits_{n\in T_\phi}\Delta t\left(\hat{\mathcal F}_{i+1/2}^n-\hat{\mathcal F}_{i-1/2}^n\right)\phi_i^n\le 0.
			\end{aligned}
            \end{equation}
			We denote the LHS terms as $-T_1-T_2+T_3$. For $T_1$, we have
			\begin{equation}
				\begin{aligned}T_1\,-&\int_{\mathbb R}{\mathcal{U} \left( u_0\left( x \right) \right) \phi \left( x,0 \right) \mathrm{d}x}
					\\
					=&\sum_{i\in X_\phi}{\int_{I_i}{\left(\tilde{\mathcal{U}}_{i}^{0}- \mathcal{U} \left( u_0\left( x \right) \right)  \right) \phi \left( x,0 \right) \mathrm{d}x}}+\sum_{i\in X_\phi}\int_{I_i}\tilde{\mathcal U}_i^0(\phi_i^0-\phi(x,0))\mathrm dx
					\\
					=&\sum_{i\in X_\phi}{\int_{I_i}{\left(\tilde{\mathcal{U}}_{i}^{0}- \mathcal{U} \left( u_0\left( x \right) \right)  \right)\phi \left( x,0 \right) \mathrm{d}x}}+o\left( 1 \right) 
					\\
					=&\sum_{i\in X_\phi}{\int_{I_i}{ \left(\tilde{\mathcal{U}}_{i}^{0} -\mathcal U\left(u_h^0\left( x \right)\right) +\mathcal U\left(u_h^0\left( x \right)\right) -\mathcal U\left( u_0\left( x \right)\right) \right)\phi \left( x,0 \right) \mathrm{d}x}}+o\left( 1 \right).
				\end{aligned}
			\end{equation}
			Utilizing \eqref{eq:convergeicU}, we have
			$$
			\sum_{i\in X_\phi}\int_{I_i} \left(\mathcal U(u_h^0(x))-\mathcal U\left( u_0\left( x \right) \right)\right) \phi \left( x,0 \right)\mathrm dx\to 0.
			$$
			On the other hand, since $u_h$ converges boundedly to $u^\star$, we know there exist $m$ and $M$ such that $u_h\in[m,M]$ uniformly for each $\Delta x,\Delta t$; thus the premises of Proposition \ref{assu:bnd} and Corollary \ref{cor:bnd} are satisfied. By utilizing \eqref{eq:bndUbar}, \eqref{eq:tvd} and \eqref{eq:tvdU}, we have
			$$\begin{aligned}&\left| \sum_{i\in X_\phi}{\int_{I_i}{\left(\tilde{\mathcal U}_i^0 -\mathcal U\left(u_h^0\left(x\right)\right) \right) \phi \left( x,0 \right) \mathrm{d}x}} \right|
				\\ &\le \left\|\phi(\cdot,0)\right\|_{L^{\infty}(\mathbb R)}\sum_{i\in X_\phi}{\int_{I_i}{\left(\left| \tilde{\mathcal U}_{i}^0 -\mathcal U\left (u_h^0\left( x_i \right)\right)\right|+\left|\mathcal U\left(u_h^0\left(x\right)\right)-\mathcal U\left(u_h^0\left(x_i\right)\right)\right| \right)  \mathrm{d}x}}
				\\ &\le C\Delta x\sum_{i\in X_\phi} \left({\max_{B_i}  \left|u_h^0(x)-u_h^0(x_i)\right|+\max\limits_{B_i}\left| \mathcal U\left(u_h^0\left( x \right)\right) -\mathcal U\left(u_h^0\left( x_i \right)\right) \right|}\right) \\
				& \rightarrow 0.
			\end{aligned}
			$$
			Hence, $T_1\to \displaystyle\int_\mathbb R\mathcal U(u_0(x))\phi(x,0)\mathrm dx$. For $T_2$, we have 
			\begin{align*}
					T_2\,-&\int_{\mathbb R^+}{\int_{\mathbb R}{\mathcal{U} \left( u^{\star} \right) \phi _t\,\mathrm{d}x\mathrm{d}t}}
					\\
					=&\sum_{i\in X_\phi}{\sum_{n\in T_\phi}{\int_{t^n}^{t^{n+1}}{\int_{I_i}{\left( \tilde{\mathcal{U}}_{i}^{n+1}\frac{\phi _{i}^{n+1}-\phi _{i}^{n}}{\Delta t}-\mathcal{U} \left( u^{\star} \right) \phi _t \right)}}}}\mathrm{d}x\mathrm{d}t
					\\
					=&\sum_{i\in X_\phi}{\sum_{n\in T_\phi}{\int_{t^n}^{t^{n+1}}{\int_{I_i}{\left( \tilde{\mathcal{U}}_{i}^{n+1}-\mathcal{U} \left( u^{\star} \right) \right) \phi _t\,}}}}\mathrm{d}x\mathrm{d}t+o\left( 1 \right) 
					\\
					=&\sum_{i\in X_\phi}{\sum_{n\in T_\phi}{\int_{t^n}^{t^{n+1}}{\int_{I_i}{\left( \tilde{\mathcal{U}}_{i}^{n+1}-\mathcal{U} \left( u_{h}^{n+1}\left( x \right) \right) +\mathcal{U} \left( u_{h}^{n+1}\left( x \right) \right) -\mathcal{U} \left( u^{\star} \right) \right) \phi _t\,}}}}\mathrm{d}x\mathrm{d}t
                    \\&+o\left( 1 \right). 
			\end{align*}
			Notice that since $\phi\in C_0^\infty(\mathbb R\times\mathbb R^+)$, we also have $\phi_t\in C_0^\infty(\mathbb R\times\mathbb R^+)$. Thus similarly, using \eqref{eq:convergeU}, \eqref{eq:bndUbar}, \eqref{eq:tvd} and \eqref{eq:tvdU}, we have $T_2\to \displaystyle\int_{\mathbb R^+}\int_\mathbb R \mathcal U(u^\star)\phi_t\,\mathrm dx\mathrm dt $. 
			
			For $T_3$, we can write
			\begin{align*}T_3=&-\sum_{i\in X_\phi}\sum_{n\in T_\phi}\Delta t\mathcal{F} \left( u_{h}^{n}\left( x_i \right) \right) \left( \phi _{i+1/2}^{n}-\phi _{i-1/2}^{n} \right) 
				\\
				&+\sum_{i\in X_\phi}\sum_{n\in T_\phi}\Delta t\left( \hat{\mathcal{F}}_{i+1/2}^{n}
				-\mathcal{F} \left( u_{h}^{n}\left( x_i \right) \right) \right) \left( \phi _{i}^{n}-\phi _{i+1/2}^{n} \right) 
				\\&-\sum_{i\in X_\phi}\sum_{n\in T_\phi}\Delta t\left( \hat{\mathcal{F}}_{i-1/2}^{n}-\mathcal{F} \left( u_{h}^{n}\left( x_i \right) \right) \right) \left( \phi _{i}^{n}-\phi _{i-1/2}^{n} \right)
				\\
				=:&\, -T_{31}+T_{32}-T_{33}.
			\end{align*}
			Using \eqref{eq:convergeF}, \eqref{eq:tvd}, and continuity of time translations in $L^1$, we have 
			$$T_{31}=\sum_{i\in X_\phi}\sum_{n\in T_\phi}\int_{t^n}^{t^{n+1}}{\int_{I_i}{\mathcal{F} \left( u_{h}^{n}\left( x_i \right) \right) \frac{\phi _{i+1/2}^{n}-\phi _{i-1/2}^{n}}{\Delta x}\mathrm{d}x\mathrm{d}t}}\rightarrow \int_{\mathbb R^+}{\int_\mathbb R{\mathcal{F} \left( u^{\star} \right) \phi _x\,\mathrm{d}x\mathrm{d}t}}.$$
			For $T_{32}$, by utilizing \eqref{eq:bndF} and \eqref{eq:tvd}, we have
			\begin{align*}
				\left| T_{32} \right|&=\frac{1}{2}\left| \sum_{i\in X_\phi}{\sum_{n\in T_\phi}{\int_{t^n}^{t^{n+1}}{\int_{I_i}{\left( \hat{\mathcal{F}}_{i+1/2}^{n}-\mathcal{F} \left( u_{h}^{n}\left( x_i \right) \right) \right)}\frac{\phi _{i+1/2}^{n}-\phi _{i}^{n}}{\Delta x/2}}}}\mathrm{d}x\mathrm{d}t \right|
				\\
				&\le \frac{1}{2}\sum_{i\in X_\phi}{\sum_{n\in T_\phi}{\int_{t^n}^{t^{n+1}}{\int_{I_i}{\left| \hat{\mathcal{F}}_{i+1/2}^{n}-\mathcal{F} \left( u_{h}^{n}\left( x_i \right) \right) \right|}\left| \phi _x \right|}}}\mathrm{d}x\mathrm{d}t+o\left( 1 \right) 
				\\
				&\le C\sum_{i\in X_\phi}{\sum_{n\in T_\phi}{\int_{t^n}^{t^{n+1}}{\int_{I_i}{\left| \hat{\mathcal{F}}_{i+1/2}^{n}-\mathcal{F} \left( u_{h}^{n}\left( x_i \right) \right) \right|}}}}\mathrm{d}x\mathrm{d}t+o\left( 1 \right) 
				\\
				&\le C\max_{n\in T_\phi} \Delta x\sum\limits_{i\in X_\phi}\max_{B_i} \left| u_{h}^{n}\left( x \right) -u_{h}^{n}\left( x_i \right) \right|+o\left( 1 \right) \rightarrow 0.
			\end{align*}
			Similarly, we have $T_{33}\to 0$. Hence, $ T_3\rightarrow -\displaystyle\int_{\mathbb R^+}{\int_\mathbb R{\mathcal{F} \left( u^{\star} \right) \phi _x\,\mathrm{d}x\mathrm{d}t}}$. Letting $\Delta x,\Delta t\to 0$ in \eqref{eq:SBP}, we finally obtain \eqref{eq:ESweak1}. 
		\end{proof}
		
		\subsection{The proof of Theorem \ref{thm:LW2}}\label{app:LW2}
		
		\begin{proof}
			The definition of the ES limiter with the multistep method implies
			$$
			\tilde{\mathcal{U}}_i^{n+1} \le \sum_{l=1}^m{\left( \alpha _l\,\tilde {\mathcal{U}}_{i}^{n+1-l}-\beta _l\lambda \left( \hat{\mathcal{F}}_{i+1/2}^{n+1-l}-\hat{\mathcal{F}}_{i-1/2}^{n+1-l} \right) \right)}.
			$$
			Multiplying by the test function $\phi_i^n$ and using $\displaystyle\sum\limits_{l=1}^m\alpha_l=1$, we can get
			$$
			\sum_{l=1}^m{ \alpha _l\,\left(\tilde{\mathcal U}_i^{n+1}-\tilde {\mathcal{U}}_{i}^{n+1-l}\right)\phi_i^n+\beta _l\lambda \left( \hat{\mathcal{F}}_{i+1/2}^{n+1-l}-\hat{\mathcal{F}}_{i-1/2}^{n+1-l} \right)\phi_i^n }\le 0.
			$$
			Summing over $i,n$ yields
			\begin{equation}\label{eq:MSall}
				\sum\limits_{l=1}^m\sum_{i\in X_\phi}\sum_{n\in T_{\phi,m}}\alpha _l\Delta x\left( \tilde{\mathcal{U}}_{i}^{n+1}-\tilde{\mathcal{U}}_{i}^{n+1-l} \right) \phi _{i}^{n}+\beta _l\Delta t \left( \hat{\mathcal{F}}_{i+1/2}^{n+1-l}-\hat{\mathcal{F}}_{i-1/2}^{n+1-l} \right) \phi _{i}^{n}\le 0.
			\end{equation}
			Here, $T_{\phi,m}:=\{n\in T_\phi: n\ge m-1 \}$. 
            For each $1\le l\le m$, the first term satisfies
			\begin{align*}
					\sum_{i\in X_\phi}&\sum_{n\in T_{\phi,m}}\Delta x\left( \tilde{\mathcal{U}}_{i}^{n+1}-\tilde{\mathcal{U}}_{i}^{n+1-l} \right) \phi _{i}^{n}=\sum_{i\in X_\phi}{\sum_{n\in T_{\phi,m}}{\sum_{s=1}^{l}{\Delta x\left( \tilde{\mathcal{U}}_{i}^{n+1-l+s}-\tilde{\mathcal{U}}_{i}^{n-l+s} \right) \phi _{i}^{n}}}}
					\\
					=&\sum_{i\in X_\phi}{\sum_{s=1}^{l}{\sum_{n\in T_{\phi,m}}{\Delta x\left( \tilde{\mathcal{U}}_{i}^{n+1-l+s}-\tilde{\mathcal{U}}_{i}^{n-l+s} \right) \phi _{i}^{n}}}}
					\\
					=&-\sum_{i\in X_\phi}{\sum_{s=1}^{l}{\Delta x\tilde{\mathcal{U}}_{i}^{m-1-l+s}\phi _{i}^{m-1}}}-\sum_{i\in X_\phi}{\sum_{s=1}^{l}{\sum_{n\in T_{\phi,m}}{\Delta x\tilde{\mathcal{U}}_{i}^{n+1-l+s}\left( \phi _{i}^{n+1}-\phi _{i}^{n} \right)}}}
					\\
					=&-\sum_{i\in X_\phi}{\sum_{s=1}^{l}{\Delta x\tilde{\mathcal{U}}_{i}^{0}\phi _{i}^{0}}}-\sum_{i\in X_\phi}{\sum_{s=1}^{l}{\sum_{n\in T_{\phi,m}}{\Delta x\tilde{\mathcal{U}}_{i}^{n+1-l+s}\left( \phi _{i}^{n+1-l+s}-\phi _{i}^{n-l+s} \right)}}}
                    \\ &+o\left( 1 \right) 
					\\
					=&-\sum_{i\in X_\phi}{\sum_{s=1}^{l}{\Delta x\tilde{\mathcal{U}}_{i}^{0}\phi _{i}^{0}}}-\sum_{i\in X_\phi}{\sum_{s=1}^{l}{\sum_{n\in T_{\phi}}{\Delta x\tilde{\mathcal{U}}_{i}^{n+1}\left( \phi _{i}^{n+1}-\phi _{i}^{n} \right)}}}+o\left( 1 \right) 
					\\
					=&\,l\left( -\sum_{i\in X_\phi}{\Delta x\tilde{\mathcal{U}}_{i}^{0}\phi _{i}^{0}}-\sum_{i\in X_\phi}{\sum_{n\in T_{\phi}}{\Delta x\tilde{\mathcal{U}}_{i}^{n+1}\left( \phi _{i}^{n+1}-\phi _{i}^{n} \right)}} \right) +o\left( 1 \right) 
					\\
					\to&\, l\left( -\int_{\mathbb R}{\mathcal{U} \left( u_0\left( x \right) \right) \phi \left( x,0 \right) \mathrm{d}x}-\int_{\mathbb R^+}{\int_{\mathbb R}{\mathcal{U} \left( u^{\star} \right) \phi _t\,\mathrm{d}x\mathrm{d}t}} \right) .
			\end{align*}
			Likewise, for each $1\le l\le m$,
            using continuity of time translations in $L^1$ and consistency
            of the entropy flux in each branch, the second term satisfies
				$$\sum_{i\in X_\phi}\sum_{n\in T_{\phi,m}}\Delta t\left( \hat{\mathcal{F}}_{i+1/2}^{n+1-l}-\hat{\mathcal{F}}_{i-1/2}^{n+1-l} \right) \phi _{i}^{n}\to -\int_{\mathbb R^+}\int_\mathbb R \mathcal F(u^\star)\phi_x\,\mathrm dx\mathrm dt.$$
			Hence, we have 
			\begin{equation}\label{eq:MSnew}
				\begin{aligned}\left(\sum\limits_{l=1}^ml\alpha_l\right)&\left(-\int_{\mathbb R} \mathcal U(u_0(x))\phi(x,0)\mathrm dx-\int_{\mathbb R^+}\int_\mathbb R\mathcal U(u^\star)\phi_t\,\mathrm dx\mathrm dt \right)
					\\&-\left(\sum\limits_{l=1}^m\beta_l\right)\left(\int_{\mathbb R^+}\int_\mathbb R \mathcal F(u^\star)\phi_x\,\mathrm dx\mathrm dt\right)\le 0.
				\end{aligned}
			\end{equation}
			Recall that the coefficients $\alpha_l,\beta_l$ of the multistep method \eqref{eq:MS} satisfy \cite{gottlieb2001strong}
			$$
			\sum\limits_{l=1}^m l\alpha_l=\sum\limits_{l=1}^{m}\beta_l.\label{eq:coeff2}
			$$
			By using this equation and the fact that $\alpha_l\ge 0$, we finally get \eqref{eq:ESweak1}. 
		\end{proof}

	\bibliographystyle{siamplain}
	\bibliography{references}
\end{document}